\documentclass[11pt]{article}

\usepackage{amsmath}
\usepackage{amssymb}
\usepackage{latexsym}
\usepackage{enumitem}
\usepackage{mathrsfs}
\usepackage{comment}
\usepackage{color}
\usepackage{bbm}

\newtheorem{assumption}{Assumption}

\def\qed{ \ \vrule width.2cm height.2cm depth0cm\smallskip}
\newenvironment{proof}{\noindent {\bf Proof.\/}}{$\qed$\vskip 0.1in}

\newcommand{\ol}{\overline}
\newcommand{\ul}{\underline}

\newcommand{\ba}{\begin{array}}
\newcommand{\ea}{\end{array}}
\newcommand{\be}{\begin{equation}}
\newcommand{\ee}{\end{equation}}
\newcommand{\bea}{\begin{eqnarray}}
\newcommand{\eea}{\end{eqnarray}}
\newcommand{\beaa}{\begin{eqnarray*}}
\newcommand{\eeaa}{\end{eqnarray*}}

\def\dbA{\mathbb{A}}

\def\dbE{\mathbb{E}}
\def\dbF{\mathbb{F}}

\def\dbL{\mathbb{L}}

\def\dbP{\mathbb{P}}
\def\dbR{\mathbb{R}}
\def\dbS{\mathbb{S}}
\def\dbT{\mathbb{T}}

\def\dbV{\mathbb{V}}
\def\dbX{\mathbb{X}}

\def\a{\alpha}
\def\b{\beta}
\def\g{\gamma}
\def\d{\delta}
\def\e{\varepsilon}

\def\k{\kappa}
\def\l{\lambda}

\def\f{\varphi}
\def\th{\theta}
\def\o{\omega}

\def\G{\Gamma}
\def\D{\Delta}

\def\L{\Lambda}

\def\O{\Omega}
\def\cA{{\cal A}}

\def\cF{{\cal F}}

\def\cH{{\cal H}}

\def\cL{{\cal L}}
\def\cM{{\cal M}}

\def\cP{{\cal P}}

\def\cV{{\cal V}}
\def\cW{{\cal W}}

\def\no{\noindent}

\def\ms{\medskip}

\def\q{\quad}
\def\qq{\qquad}

\def\pa{\partial}
\def\cd{\cdot}
\def\cds{\cdots}

\def\tr{\hbox{\rm tr}}

\def\qed{ \hfill \vrule width.25cm height.25cm depth0cm\smallskip}

\newcommand{\basa}{\begin{assumption}}
\newcommand{\easa}{\end{assumption}}

\newcommand{\bas}{\begin{assum}}
\newcommand{\eas}{\end{assum}}

\def\limsup{\mathop{\overline{\rm lim}}}
\def\liminf{\mathop{\underline{\rm lim}}}

\def\pa{\partial}

 \def\cd{\cdot}
\def\cds{\cdots}

\def\supp{\hbox{\rm supp$\,$}}

\def\tr{\hbox{\rm tr$\,$}}

\def\dis{\displaystyle}

\def\bx{{\bf x}}

\def\1{{\bf 1}}
\def\by{{\bf y}}

\def\:{\!:\!}
\def\reff#1{{\rm(\ref{#1})}}
\def \proof{{\noindent \bf Proof\quad}}

at 9pt

\newtheorem{thm}{Theorem}[section]
\newtheorem{lem}[thm]{Lemma}

\newtheorem{prop}[thm]{Proposition}
\newtheorem{rem}[thm]{Remark}
\newtheorem{eg}[thm]{Example}
\newtheorem{defn}[thm]{Definition}
\newtheorem{assum}[thm]{Assumption}

\numberwithin{equation}{section}

\usepackage[dvipsnames]{xcolor}

\begin{document}

\title{\bf Set Values for Mean Field Games }
\author{Melih \.{I}\c{s}eri\footnote{Department of Mathematics, University of Southern California, United States, melihise@usc.edu.}  \quad Jianfeng Zhang\footnote{Department of Mathematics, University of Southern California, United States, jianfenz@usc.edu. This author is supported in part by NSF grants DMS-1908665 and DMS-2205972. } \footnote{ The authors would like to thank Daniel Lacker for insightful discussions.}}

\maketitle

\begin{abstract}
In this paper we study mean field games with possibly multiple mean field equilibria. Instead of focusing on the individual equilibria, we propose to study the set of values over all possible equilibria, which we call the set value of the mean field game. When the mean field equilibrium is unique, typically under certain monotonicity conditions, our set value reduces to the singleton of the standard value function which solves the master equation. The set value is by nature unique, and we shall establish two crucial properties: (i) the dynamic programming principle, also called time consistency; and  (ii) the convergence of the set values of the corresponding $N$-player games, which can be viewed as a type of stability result. To our best knowledge, this is the first work in  the literature which studies the dynamic value of mean field games without requiring the uniqueness of mean field equilibria. We emphasize that the set value is very sensitive to the type of the admissible controls. In particular, for the convergence one has to restrict to corresponding types of equilibria for the N-player game and for the mean field game. We shall illustrate this point by investigating three cases, two in finite state space models and the other in a  continuous time model with controlled diffusions.
\end{abstract}

\no{\bf MSC2020.} 91A16, 60H30, 91A25, 91A06, 93E20

\vspace{3mm}
\no{\bf Keywords.}  Mean field games, mean field equilibria, set values, dynamic programming principle, closed loop controls, relaxed controls

\vfill\eject

\section{Introduction}
\label{sect-Introduction} \setcounter{equation}{0}

In this paper we study Mean Field Games (MFG, for short) without
monotonicity conditions. There are typically multiple Mean Field
Equilibria (MFE, for short) with possibly different values. Instead of
focusing on the individual equilibria, we propose to study the set of
values over all equilibria, which we call the set value of the
MFG. Note that the set value always exists (with empty set as a
possible value) and is by definition unique. When the MFE is unique,
typically under certain monotonicity conditions, our set value is
reduced to the singleton of the standard value function of the game,
which solves the so called master equation. So the set value can be
viewed as the counterpart of the standard value function for MFGs without monotonicity conditions, 
and it indeed shares many nice properties. In this paper, we 
focus particularly on two crucial properties of the set value:
\begin{itemize}
\item the Dynamic Programming Principle (DPP, for short), or say the
  time consistency;
\item the convergence of the set values of the corresponding
  $N$-player games,  which can be viewed as a type of stability result in terms of model perturbation.
\end{itemize}

\no For general theory of MFGs, we refer to Caines-Huang-Malhame \cite{CHM}, Lasry-Lions \cite{LL}, Lions \cite{Lions}, Cardaliaguet \cite{Cardaliaguet}, Bensoussan-Frehse-Yam \cite{BFY}, and Camona-Delarue \cite{CD1, CD2}.

In standard stochastic control theory, it is well known that the
dynamic value function satisfies the DPP. In fact, this is the
underlying reason for the PDE approach to work. For MFGs under
appropriate monotonicity conditions, the value function (at the
unique MFE) also satisfies the DPP, which, together with the It\^o
formula, leads to the master equation. However, with the presence of
multiple equilibria (see, e.g., Bardi-Fischer \cite{BardiFischer} for
some examples), to our best knowledge this is the first work in the
literature to study the MFG dynamically and to address the time
consistency issue. We show that, when formulated properly, the dynamic
set value function satisfies the DPP.  This also opens the door to a
possible PDE approach for these general games by introducing the so
called set valued PDE. We refer to our work \cite{IZ} for set valued PDEs induced by multivariate stochastic control problems, and Ma-Zhang-Zhang \cite{MZZ} for numerical methods for set valued PDEs,  and we leave their extension to mean field games for future research.  Our
set value approach follows from Feinstein-Rudloff-Zhang \cite{FRZ},
which studies nonzero sum games with finitely many players. See also
the related works Abreu-Pearce-Stacchetti \cite{APS} and Sannikov
\cite{Sannikov} in economics literature, and Feinstein
\cite{Feinstein} which studies the set of equilibria instead of
values.

We note that the set value of games relies heavily on the types of
admissible controls we use.  In this paper we shall consider closed
loop controls.  The open loop equilibria of games are typically time
inconsistent, see e.g. Buckdahn's counterexample in Pham-Zhang
\cite[Appendix E]{PZ} for a two person zero sum game, and
consequently, the set value of games with open loop controls would
violate the DPP. For the MFG, noting that the required symmetry
decomposes the game problem into a standard control problem and a
fixed point problem of measures, and that open loop and closed loop
controls yield the same value function for a standard control problem,
it is possible that the set value with open loop controls still
satisfies the DPP. Nevertheless, bearing in mind the DPP of the set
value for more general (non-symmetric) games, as well as the practical
consideration in terms of the information available to the players, we
shall focus on closed loop controls. There is also a very subtle path
dependence issue. While the game parameters are state dependent, we
may consider both state dependent and path dependent controls. For
general non-zero sum games (not mean field type), \cite{FRZ} shows
that DPP holds for the set value for path dependent controls, but in
general fails for the set value for state dependent controls. For MFGs
with closed loop controls, again due to the required symmetric
properties, the set values for both state dependent controls and path
dependent controls will satisfy the DPP, but they are in general not
equal. For MFGs with closed loop relaxed controls, or say closed loop
mixed strategies, however, it turns out that the state dependent
controls and the path dependent controls induce the same set value
which still satisfies the DPP.

We next turn to the convergence issue. Let  $\dbV$ and $\dbV^N$ denote the set values of the MFG and the corresponding $N$-player games, respectively, under appropriate closed-loop controls. 
Our convergence result reads roughly as follows (the precise form is slightly different):
\bea\label{convergence}
\lim_{N\to\infty} \dbV^N(0, \vec x) = \dbV(0, \mu),\q\mbox{when}\q \mu^N_{\vec x}:= {1\over N} \sum_{i=1}^N \d_{x_i} \to \mu.
\eea
In the realm of master equations, again under certain
monotonicity conditions and hence with unique MFE, one can show that the values of the $N$-player games converge to the
value of the MFG. See Cardaliaguet-Delarue-Lasry-Lions \cite{CDLL}, followed by
Bayraktar-Cohen \cite{BC}, Cardaliaguet \cite{Cardaliaguet2},
Cecchin-Pelino \cite{CP}, Delarue-Lacker-Ramanan \cite{DLR1, DLR2},
Gangbo-Meszaros \cite{GM}, and Mou-Zhang \cite{MZ}, to mention a few. So \reff{convergence} can be viewed as their natural extension to MFGs without monotonicities.

We emphasize again that the set value is very sensitive to the types
of admissible controls. To ensure the convergence, one simple but
crucial observation is that the $N$-player game and the MFG should use
the "same" type of controls (more precisely, corresponding types of
controls in appropriate sense). We illustrate this point by
considering two cases. Note that in the standard literature each
player is required to use the same closed loop control along an
MFE. For the first case, we will obtain the desired convergence by
restricting the $N$-player game to homogeneous equilibria, namely each
player also uses the same closed loop control. In the second case, we
remove such restriction and consider heterogenous equilibria for the
$N$-player games. Note that a closed loop control means the control
depends only on the state.  In this heterogenous case players with the
same state may choose different controls, then one can not expect in
the limit they will have to use the same control\footnote{When the MFE is unique, under appropriate monotonicity conditions, the set value becomes a singleton and it is not sensitive to the type of admissible controls anymore. Consequently, the convergence becomes possible even if the $N$-player games and the MFG use different types of controls, see e.g. \cite{CDLL}}. Indeed, in this case
the limit is characterized by the MFG with closed loop relaxed
controls, or say closed loop mixed strategies, which exactly means
players with the same state may still have a distribution of controls
to choose from. However, since our relax control for MFG is still homogeneous, namely
each player uses the same relax control, the controls for N-player game and for MFG appear to be
in different forms. Our approach is to introduce a new formulation for the MFG, which embeds
the structure of heterogenous controls and shares the same set value as the relax control
formulation of the MFG. For the homogeneous case, we will investigate both a
discrete time model with finite state space and a continuous time
diffusion model with drift controls. But for the heterogeneous case we
will investigate the discrete model only. The continuous model in such
case involves some technical challenges for the convergence and we
shall leave it for future research. We shall point out that, however,
the DPP would hold in much more general models without significant
difficulties.

To ensure the convergence, another main feature is that we define the
set value as the limit of the approximate set values over approximate
equilibria, rather than the true equilibria. We
call the latter the raw set value, and both the set value and the raw
set value satisfy the DPP. However, the raw set value is extremely
sensitive to small perturbations of the game parameters, in fact, in
general even its measurability is not clear, so one can hardly expect
the convergence for the raw set values. In the standard control
theory, the value function is defined as the infimum of controlled
values, which is exactly the limit of values over approximate optimal
controls, rather than the value over true optimal controls which may
not even exist. So our set value, not the raw set value, is the
natural extension of the standard value function in control theory.
Moreover, since we are considering infinitely many players, an
approximate equilibrium means it is approximately optimal for most
players, but possibly with a small portion of exceptions, as
introduced in Carmona \cite{Carmona}.

 We would like to mention that, although it is not the focus of the present paper, the set value is also numerically a lot easier to compute than the raw set value. For example, the duality result for set values in \cite[Section 3.4]{FRZ} (for finite player games) is very useful for constructing efficient numerical algorithms, see \cite{MZZ}. However, this is not feasible for the raw set value which lacks regularity  and thus is hard to approximate in general.

At this point we should mention that, for MFGs without monotonicity conditions,  there have been many publications on the convergence of $N$-player games, in terms of equilibria instead of values. For open loop controls, we refer to Camona-Delarue \cite{CD0}, Feleqi
\cite{Feleqi}, Fischer \cite{Fischer}, Fischer-Silva
\cite{FischerSilva}, Lacker \cite{Lacker1}, Lasry-Lions \cite{LL},
Lauriere-Tangpi \cite{LT}, and Nutz-San Martin-Tan \cite{NST}, to mention a few.  In
particular, \cite{Lacker1} provides the full characterization for the
convergence: any limit of approximate Nash equilibria of $N$-player
games is a weak MFE, and conversely any weak MFE can be obtained as
such a limit. The work \cite{Fischer} is also in this direction. For closed loop
controls, which we are mainly interested in, the situation becomes much more subtle.  The seminal paper Lacker \cite{Lacker2} established the following result:
\bea
\label{Lacker} \mbox{\{Strong MFEs\}} ~ \subset ~ \mbox{\{Limits of
$N$-player approx. equilibria\}}~ \subset ~ \mbox{\{Weak MFEs\}}.
\eea 
Here an MFE is strong if it depends only on the state processes,
and weak if it allows for additional randomness. The left inclusion in \reff{Lacker} was known  to be strict in general. This work has very interesting further developments recently\footnote{These two works \cite{Djete, DanielLuc} were circulated slightly after our present paper.} by  Lacker-Flem
\cite{DanielLuc}  and Djete \cite{Djete}. In particular,  \cite{Djete} shows that the right inclusion in \reff{Lacker} is actually an equality. 

We emphasize again that we are considering the convergence of sets of values, rather than sets of equilibria as in \reff{Lacker}. For standard control problems, the focus is typically to characterize the (unique) value and to find {\it one} (approximate) optimal control, and the player is less interested in finding {\it all} optimal controls since they have the same value. The situation is quite different for games, because different equilibria can lead to different values. Then it is not satisfactory to find just one equilibrium (especially if it is not Pareto optimal). However, for different equilibria which lead to the same value, the players are indifferent on them. So for practical purpose the players would be more interested in finding all possible values\footnote{Another very interesting question is how to choose an optimal (in appropriate sense) value after characterizing the set value. We shall leave this for future research.} and then to find one (approximate) equilibrium for  each value. This is one major motivation that we focus on the set value, rather than the set of all equilibria. We also note that in general the set value could be much simpler than the set of equilibria. For example, in the trivial case that both the terminal and the running cost functions are constants,  the set value is a singleton, while the set of  equilibria consists of all admissible controls.

We should point out that our admissible controls differ from those in \cite{Djete, Lacker2,  DanielLuc}. Roughly speaking, we put two constraints, due to both practical and technical considerations, on the $N$-player approximate equilibria so that the left inclusion in \reff{Lacker} (in terms of values instead of equilibria) becomes an equality. First, for the $N$-player games, \cite{Djete, Lacker2,  DanielLuc} use full information controls $\a_i(t, X^1_t, \cds, X^N_t)$, while we consider symmetric controls $\a_i(t, X^i_t, \mu^N_t)$, where $X^i_t$ is the state of Player $i$, and $\mu^N_t := {1\over N}\sum_{j=1}^N \d_{X^j_t}$ is the empirical measure of all the players' states. Note that, as a principle the controls should depend only on the information the players observe. While both settings are very interesting, since $N$ is large, the full information may not be available in many practical situations. 

The second difference is that we assume each control is Lipschitz continuous in $\mu$, while \cite{Djete, Lacker2,  DanielLuc}  allow for measurable controls. We shall emphasize though we allow the Lipschitz constant to depend on the control, and thus our set value does not depend on any fixed Lipschitz constant. Roughly speaking, we are considering game values which can be approximated by Lipschitz continuous approximate equilibria. This is typically the case in the standard control theory: even if the optimal control is discontinuous, in  most reasonable framework we should be able to find Lipschitz continuous approximate optimal controls.   The situation is more subtle for games. There may exist (closed loop) equilibria whose values cannot be approximated by any Lipschitz continuous approximate equilibria. In fact, when considering all measurable equilibria, the convergence of set values in \reff{convergence} fails in general, see Example \ref{eg-Lipschitz} and Remark \ref{rem-Lipschitz} below. While clearly more general and very interesting mathematically, such measurable equilibria are hard to implement in practice, since inevitably we have all sorts of errors in terms of the information, or say, data.  Their numerical computation is another serious challenge. For example, in the popular machine learning algorithm,  the key idea is to approximate the controls via composition of linear functions and the activation function, then by definition the optimal controls/equilibria provided by these algorithms are (locally) Lipschitz continuous. That is, the game values falling out of our set value are essentially out of reach of these algorithms, see e.g. \cite{MZZ}.  Moreover, as a consequence of our constraints,  our proof of \reff{convergence} is technically a lot easier than the compactness arguments for \reff{Lacker} used in  \cite{Djete, Lacker2,  DanielLuc}.

Finally we would like to mention some other approaches for MFGs with
multiple equilibria. One is to add sufficient (possibly infinite
dimensional) noise so that the new game will become non-degenerate and
hence have unique MFE, see e.g. Bayraktar-Cecchin-Cohen-Delarue
\cite{BCCD1, BCCD2}, Delarue \cite{Delarue}, Delarue-Foguen Tchuendom
\cite{DF}, Foguen Tchuendom \cite{F}. Another approach is to study a
special type of MFEs, see e.g.  Cecchin-Dai Pra-Fisher-Pelino
\cite{CDFP}, Cecchin-Delarue \cite{CecchinDelarue}, and
\cite{DF}. Another interesting work is Possamai-Tangpi \cite{PT} which
introduces an additional parameter function $\L$ such that the MFE
corresponding to any fixed $\L$ is unique and then the desired
convergence is obtained.

The rest of the paper is organized as follows. In Section
\ref{sect-state} we introduce the set value for an MFG in a discrete
time model on finite state space and establish the DPP, and in Section
\ref{sect-N1} we prove the convergence for the corresponding
$N$-player games with homogeneous equilibria. Sections
\ref{sect-relax} and \ref{sect-N2} are devoted to MFGs with relaxed
controls and the corresponding $N$-player games with heterogenous
equilibria. In Section \ref{sect-Diffusion} we study a continuous time
model with controlled diffusions. Finally in Appendix we provide some examples, discuss the subtle
path dependence issue, and complete some technical proofs.

\section{Mean field games on finite space with closed loop controls}
\label{sect-state} \setcounter{equation}{0}

In this section we consider an MFG on finite space (both time and
state are finite) with closed loop controls, and for simplicity we
restrict to state dependent setting. Since the game typically has
multiple MFEs which may induce different values, see Example
\ref{eg-statepath} below for an example, we shall introduce the set
value of the game over all MFEs. Our goal is to establish the DPP for
the MFG set value, and we shall show in the next section that the set
values of the corresponding $N$-player games converge to the MFG set
value.

\subsection{The basic setting}
\label{sect-FiniteSymmetricMFG}
Let $\dbT := \{0,\cdots, T\}$ be the set of discrete times;
$\dbT_t:= \{t,\cds,T\}$ for $t\in \dbT$; $\dbS$ the finite state
space\footnote{We may allow the state space $\dbS_t$ to depend on time $t$ and all the results in this paper will remain
  true.} with size $|\dbS|=d$; $\cP(\dbS)$ the set of probability
measures on $\dbS$, equipped with the $1$-Wasserstein distance
$W_1$. Since $\dbS$ is finite, $W_1$ is equivalent to the total
variation distance\footnote{More precisely, the total variation
  distance is ${1\over 2} W_1$ for the $W_1$ in \reff{W1}.} which is
convenient for our purpose: by abusing the notation $W_1$, \bea
\label{W1} W_1(\mu, \nu) := \sum_{x\in \dbS} |\mu(x) - \nu(x)|,\q \mu,
\nu \in \cP(\dbS).  \eea Let $\cP_0(\dbS)$ denote the subset of
$\mu\in \cP(\dbS)$ which has full support, namely $\mu(x)>0$ for all
$x\in\dbS$.  Moreover, let $\dbA\subset \dbR^{d_0}$ be a measurable
set from which the controls take values; and $q: \dbT \times \dbS
\times \cP(\dbS)\times \dbA \times \dbS \to (0, 1)$ be a transition
probability function: \beaa \sum_{\tilde x\in \dbS} q(t, x, \mu, a;
\tilde x) = 1,\q\forall (t, x, \mu, a)\in \dbT \times \dbS \times
\cP(\dbS)\times \dbA.  \eeaa

We shall use the weak formulation which is more convenient for closed
loop controls. That is, we fix the canonical space and consider
controlled probability measures on it. To be precise, let $\O:= \dbX
:= \dbS^{T+1}$ be the canonical space; $X: \dbT\times \O\to \dbS$ the
canonical process: $X_t(\o) = \o_t$; $\dbF:=\{\cF_t\}_{t\in \dbT}:=
\dbF^X$ the filtration generated by $X$; and $\cA_{state}$ the set of
state dependent admissible controls $\a: \dbT \times \dbS \to
\dbA$. Introduce the concatenation for controls: \bea
\label{Oplus} (\a\oplus_{T_0} \tilde \a)(s,x):= \a(s,x) \1_{\{s<T_0\}} +
\tilde \a(s,x) \1_{\{s\ge T_0\}},\q \a, \tilde \a\in \cA_{state}.  \eea
It is clear that $\a\oplus_{T_0} \tilde \a\in \cA_{state}$. Given $(t,
\mu, \a)\in \dbT\times \cP(\dbS) \times \cA_{state}$, let
$\dbP^{t,\mu, \a}$ denote the probability measure on $\cF_T$
determined recursively by: for $s=t,\cds, T$, 
\bea
\label{Ptmua} 
\left.\ba{c} \dis\dbP^{t,\mu,\a} \circ X_t^{-1} = \mu,\q \dbP^{t,\mu,\a}(X_{s+1}=\tilde x| X_s = x) = q(s, x, \mu^\a_s, \a(s,
x); \tilde x);\vspace{0.5em}\\ 
\dis \mbox{where}\q \mu^\a_s:= \dbP^{t,\mu,\a} \circ X_s^{-1}.  
\ea\right.  
\eea 
We note that $\mu^\a:=\{\mu^\a_s\}_{s\in \dbT_t}$
are uniquely determined and $X$ is a Markov chain on $\dbT_t$ under
$\dbP^{t,\mu,\a}$. We also note that $\mu^\a$ depends on $(t,\mu)$ as well, but we omit it for notational simplicity.  However, the distribution of $\{X_s\}_{s=0,\cds,
t-1}$ is not specified and is irrelevant, and $\{\a_s\}_{0\le s<t}$ is
also irrelevant.  Moreover, given $\{\mu_\cd\}:=\{\mu_s\}_{s\in
\dbT_t}$, $x\in \dbS$, and $\tilde \a\in \cA_{state}$, let
$\dbP^{\{\mu_\cd\}; t, x, \tilde \a}$ denote the probability measure
on $\cF_T$ determined recursively by: for $s=t,\cds, T-1$,
\begin{equation}
  \label{Ptmuax}
  \dis \dbP^{\{\mu_\cd\}; t, x, \tilde \a}(X_t=x) = 1,\q
  \dbP^{\{\mu_\cd\}; t, x, \tilde \a}(X_{s+1}= \bar x| X_s = \tilde x) =
  q(s, \tilde x, \mu_s, \tilde\a(s, \tilde x); \bar x).
\end{equation}
As in the standard MFG literature, here we are assuming that the
population uses the common control $\a$ while the individual player is
allowed to use a different control $\tilde \a$.

We remark that, since we assume $q>0$, then for any $(t, \mu)$ and
$\a$, $\mu^{\a}_s\in \cP_0(\dbS)$ for all $s>t$. For the convenience
of presentation, in this section we shall restrict our discussion to
the case $\mu\in \cP_0(\dbS)$. The general case that the initial
measure $\mu$ is not fully supported can be treated fairly easily, as
we will do in Section \ref{sect-Diffusion} below. The situation with
degenerate $q$, however, is more subtle and we shall leave  it for future
research.

We finally introduce the cost functional for the MFG: for the $\mu^\a = \{\mu^\a_\cd\}$
in \reff{Ptmua},
\begin{equation}
  \begin{aligned}
    &J(t, \mu, \a; x, \tilde \a) :=
    J(\mu^\a; t, x, \tilde \a) ,\qq v(\{\mu_\cd\}; s, x) :=
    \inf_{\tilde\a\in \cA_{state}} J(\{\mu_\cd\}; s, x, \tilde \a);
    \\ & \mbox{where}\q
    J(\{\mu_\cd\}; s, x, \tilde \a) := \dbE^{\dbP^{\{\mu_\cd\}; s, x,
        \tilde\a}}\Big[G(X_T, \mu_T) + \sum_{r=s}^{T-1} F(r, X_r, \mu_r,
    \tilde\a(r, X_r))\Big].
    \label{J}    
  \end{aligned}
\end{equation}
Here, since $\dbT$ and $\dbS$ are finite, $F$ and $G$ are arbitrary
measurable functions satisfying \beaa \inf_{a\in \dbA}F(t, x,
\mu,a)>-\infty\q\mbox{ for all}~ (t,x,\mu).  \eeaa

We remark that here $v(\{\mu_\cd\};\cd,\cd)$
is the value function of a standard stochastic control problem with
parameter $\{\mu_\cd\}$. In particular, in continuous time models,
$\mu^\a$ and $v(\mu^\a; \cd, \cd)$ will satisfy the Fokker-Planck
equation and the HJB equation, respectively.

\begin{defn}
\label{defn-MFEstate} Given $(t, \mu)\in \dbT \times \cP_0(\dbS)$, we
say $\a^*\in \cA_{state}$ is a state dependent MFE at $(t, \mu)$,
denoted as $\a^*\in \cM_{state}(t,\mu)$, if \bea
\label{MFEstate} J(t, \mu, \a^*; x, \a^*) = v(\mu^{\a^*}; t, x),
\q\mbox{for all}~ x\in \dbS.  \eea
\end{defn} In this and the next section, we will use the following
conditions.
\begin{assum}
  \label{assum-reg}
  
  (i) $q\ge c_q$ for some constant $c_q>0$;

  (ii) $q$ is Lipschitz continuous in $(\mu, a)$, with a Lipschitz
  constant $L_q$;

  (iii) $F, G$ are bounded by a constant $C_0$ and uniformly continuous
  in $(\mu, a)$, with a modulus of continuity function $\rho$.
\end{assum}

\subsection{The raw set value $\dbV_0$ } We introduce the raw set
value for the MFG over all state dependent MFEs: \bea
\label{V0} \dbV_0(t,\mu) := \Big\{J(t, \mu, \a^*; \cd, \a^*): \a^* \in
\cM_{state}(t,\mu)\Big\} \subset \dbL^0(\dbS; \dbR).  \eea Here the
elements of $\dbV_0(t,\mu)$ are functions from $\dbS$ to $\dbR$, which
coincide with $\dbR^d$ by identifying $\f\in \dbL^0(\dbS; \dbR)$ with
$(\f(x): x\in \dbS)\in \dbR^d$. We call $\dbV_0(t,\mu)$ the raw set
value and we will introduce the set value $\dbV(t,\mu)$ of the MFG in
the next subsection.

Next, for any $T_0\in \dbT_t$, $\psi \in \dbL^0(\dbS\times
\cP_0(\dbS); \dbR)$, we introduce the MFG on $\{t,\cds, T_0\}$:
\begin{equation}
  J(T_0,\psi; t, \mu, \a; x, \tilde\a) :=
  \dbE^{\dbP^{\mu^\a; t, x, \tilde\a}}\Big[\psi(X_{T_0}, \mu^\a_{T_0}) +
  \sum_{s=t}^{T_0-1} F(s, X_s, \mu^\a_s, \tilde\a(s, X_s))\Big].
  \label{JT0} 
\end{equation}
In the obvious sense we define $\a^*\in \cM_{state}(T_0, \psi; t,\mu)$
by: for any $x\in \dbS$, 
\begin{equation}
  J(T_0,\psi; t, \mu, \a^*; x, \a^*) = v(T, \psi;
  \mu^{\a^*}; t,x):= \inf_{\tilde\a\in \cA_{state}} J(T,\psi; t, \mu,
  \a^*; x, \tilde\a).
  \label{MFEstateT0} 
\end{equation}
At below we will repeatedly use the following simple fact due to the
tower property of conditional expectations:
\begin{equation}
  \dis J(t, \mu, \a; x, \tilde\a) = J(T_0, \psi; t, \mu,
  \a; x, \tilde\a),\hspace{0.5em}
  \mbox{where}\q \psi(y, \nu) := J(T_0, \nu, \a; y, \tilde\a).
  \label{tower}  
\end{equation}
The following time consistency of MFE is the essence of the DPP for
the raw set value.

\begin{prop}
\label{prop-DPP} Fix $0\le t< T_0\le T$ and $\mu\in \cP_0(\dbS)$. For
any $\a^*, \tilde \a^*\in \cA_{state}$, denote $\hat \a^*:=
\a^*\oplus_{T_0} \tilde \a^*$ and $\psi(y, \nu):= J(T_0, \nu,
\tilde\a^*; y, \tilde\a^*)$. Then $\hat\a^* \in \cM_{state}(t,\mu)$ if
and only if $\a^* \in \cM_{state}(T_0, \psi; t,\mu)$ and $\tilde \a^*
\in \cM_{state}(T_0, \mu^{\a^*}_{T_0})$.
\end{prop}
\begin{proof}
  (i) We first prove the if part.  Let
  $\a^* \in \cM_{state}(T_0, \psi; t,\mu)$ and
  $\tilde \a^* \in \cM_{state}(T_0, \mu^{\a^*}_{T_0})$.  For arbitrary
  $\a\in \cA_{state}$ and $x\in \dbS$, by \reff{tower} we have
   \beaa
  &&J(t, \mu, \hat\a^*; x, \a) = \dbE^{\dbP^{\mu^{\a^*}; t, x,
      \a}}\Big[ J(T_0, \mu^{\a^*}_{T_0}, \tilde \a^*; X_{T_0}, \a) +
  \sum_{s=t}^{T_0-1} F(s, X_s, \mu^{\a^*}_s, \a(s, X_s))\Big]\\ 
  &&\ge
  \dbE^{\dbP^{\mu^{\a^*}; t, x, \a}}\Big[ J(T_0, \mu^{\a^*}_{T_0},
  \tilde \a^*; X_{T_0}, \tilde \a^*) + \sum_{s=t}^{T_0-1} F(s, X_s,
  \mu^{\a^*}_s, \a(s, X_s))\Big]\\ 
  &&= \dbE^{\dbP^{\mu^{\a^*};t, x,
      \a}}\Big[ \psi(X_{T_0}, \mu^{\a^*}_{T_0}) + \sum_{s=t}^{T_0-1}
  F(s, X_s, \mu^{\a^*}_s, \a(s, X_s))\Big]\\ 
  && = J(T_0, \psi; t, \mu,
  \a^*; x, \a) \ge J(T_0, \psi; t, \mu, \a^*; x, \a^*) = J(t, \mu,
  \hat\a^*; x, \hat \a^*), 
  \eeaa 
  where the first inequality is due to
  $\tilde \a^*\in \cM_{state}(T_0, \mu^{\a^*}_{T_0})$ and the second
  inequality is due to $\a^*\in \cM_{state}(T_0, \psi; t,\mu)$. Then
  $\hat\a^* \in \cM_{state}(t,\mu)$.

  (ii) We now prove the only if part. Let
  $\hat\a^* \in \cM_{state}(t,\mu)$. For any $\a\in \cA_{state}$, we
  have $\a\oplus_{T_0}\tilde \a^* \in \cA_{state}$.  Then, since
  $\hat\a^*\in \cM_{state}(t,\mu)$, for any $x\in \dbS$, by
  \reff{tower} we have \beaa J(T_0,\psi; t, \mu, \a^*; x, \a^*) = J(t,
  \mu, \hat\a^*; x, \hat\a^*) \le J(t, \mu, \hat\a^*; x,
  \a\oplus_{T_0} \tilde\a^*) = J(T,\psi; t, \mu, \a^*; x, \a).  \eeaa
  This implies that $\a^*\in \cM_{state}(T_0, \psi; t,\mu)$.

  Moreover, note that $\a^*\oplus_{T_0} \a \in \cA_{state}$ and again
  since $\hat\a^*\in \cM_{state}(t,\mu)$, we have 
  \beaa 
  &&\dbE^{\dbP^{\mu^{\a^*};t, x, \a^*}}\Big[J(T_0, \mu^{\a^*}_{T_0},
  \tilde\a^*; X_{T_0}, \tilde \a^*) + \sum_{s=t}^{T_0-1} F(s, X_s,
  \mu^{\a^*}_s, \a^*(s, X_s))\Big] \\ 
  &&= J(t, \mu, \hat\a^*; x,
  \hat\a^*) \le J(t, \mu, \hat\a^*; x, \a^*\oplus_{T_0} \a) \\ 
  &&=
  \dbE^{\dbP^{\mu^{\a^*}; t, x, \a^*}}\Big[J(T_0, \mu^{\a^*}_{T_0},
  \tilde\a^*; X_{T_0}, \a) + \sum_{s=t}^{T_0-1} F(s, X_s, \mu^{\a^*}_s,
  \a^*(s, X_s))\Big].  
  \eeaa 
  This implies that, recalling the $v$ in
  \reff{J} and by the standard stochastic control theory, 
  \bea
  \label{vDPP} \dbE^{\dbP^{\mu^{\a^*}; t, x, \a^*}}\Big[J(T_0,
  \mu^{\a^*}_{T_0}, \tilde\a^*; X_{T_0}, \tilde \a^*) \Big] &\le&
  \inf_{\a\in\cA_{state}} \dbE^{\dbP^{\mu^{\a^*};t, x,
      \a^*}}\Big[J(T_0,
  \mu^{\a^*}_{T_0},\tilde \a^*; X_{T_0}, \a) \Big]\nonumber\\
  &=&\dbE^{\dbP^{\mu^{\a^*};t, x, \a^*}}\Big[v(\mu^{\hat\a^*}; T_0,X_{T_0}) \Big].  
  \eea 
  On the other hand, by definition
  $v(\mu^{\hat\a^*}; T_0, \tilde x) \le J(T_0, \mu^{\a^*}_{T_0},
  \tilde \a^*; \tilde x, \tilde\a^*)$ for all $\tilde x\in\dbS$. Then
  \beaa J(T_0, \mu^{\a^*}_{T_0}, \tilde\a^*; X_{T_0}, \tilde \a^*) =
  v(\mu^{\hat\a^*}; T_0, X_{T_0}),\q \dbP^{\mu^{\a^*}; t, x,
    \a^*}\mbox{-a.s.}  
    \eeaa Since $q>0$, then clearly
  $\dbP^{\mu^{\a^*}; t, x, \a^*}(X_{T_0}=\tilde x)>0$ for all
  $\tilde x\in \dbS$. Thus $J(T_0, \mu^{\a^*}_{T_0}, \tilde \a^*;
  \tilde x,\tilde \a^*) = v(\mu^{\hat\a^*}; T_0, \tilde x)$, for all
  $\tilde x\in \dbS$.  This implies that
  $\tilde \a^* \in \cM_{state}(T_0, \mu^{\a^*}_{T_0})$.
\end{proof}

We then have the following DPP.

\begin{thm}
  \label{thm-DPP0} For any $0\le t< T_0\le T$, and
  $\mu\in \cP_0(\dbS)$, we have
  \begin{equation}
    \begin{aligned}
      \dbV_0(t,\mu) := \Big\{J(T_0, \psi; t,
      \mu, \a^*; \cd, \a^*): \mbox{for all}~\psi \in \dbL^0(\dbS\times
      \cP_0(\dbS); \dbR)~\mbox{and}~\a^* \in \cA_{state}
      \\ \dis \mbox{such that}~\psi(\cd, \mu^{\a^*}_{T_0}) \in
      \dbV_0(T_0,\mu^{\a^*}_{T_0})~\mbox{and}~\a^*\in
      \cM_{state}(T_0, \psi; t,\mu)\Big\}.
      \label{DPP0}  
    \end{aligned}
  \end{equation}
\end{thm}

\begin{proof}
  Let $\tilde \dbV_0(t,\mu)$ denote the right side of
  \reff{DPP0}. First, for any
  $J(T_0, \psi; t, \mu, \a^*; \cd, \a^*)\in \tilde \dbV_0(t,\mu)$ with
  desired $\psi, \a^*$ as in \reff{DPP0}. Since
  $\psi(\cd, \mu^{\a^*}_{T_0}) \in \dbV_0(T_0, \mu^{\a^*}_{T_0})$,
  there exists $\tilde \a^*\in \cM_{state}(T_0, \mu^{\a^*}_{T_0})$
  such that
  $\psi(\cd, \mu^{\a^*}_{T_0}) = J(T_0, \mu^{\a^*}_{T_0}, \tilde \a^*;
  \cd, \tilde \a^*)$. By Proposition \ref{prop-DPP} we have
  $\hat \a^*:= \a^*\oplus_{T_0} \tilde \a^* \in
  \cM_{state}(t,\mu)$. Then, by \reff{tower},
  $J(T_0, \psi; t, \mu, \a^*; \cd, \a^*) = J(t, \mu, \hat\a^*; \cd,
  \hat\a^*) \in \dbV_0(t,\mu)$, and thus
  $\tilde \dbV_0(t,\mu)\subset \dbV_0(t,\mu)$.

  On the other hand, let $J(t, \mu, \a^*; \cd, \a^*)\in \dbV_0(t,\mu)$
  with $\a^* \in \cM_{state}(t,\mu)$. Introduce
  $\psi(x, \nu) := J(T_0, \nu, \a^*; x, \a^*)$. By Proposition
  \ref{prop-DPP} again we see that
  $\a^* \in \cM_{state}(T_0, \psi; t,\mu)$ and
  $\a^* \in \cM_{state}(T_0, \mu^{\a^*}_{T_0})$, and the latter
  implies further that
  $\psi(\cd, \mu^{\a^*}_{T_0}) \in \dbV_0(T_0,
  \mu^{\a^*}_{T_0})$. Then by the definition of $\tilde \dbV_0(t,\mu)$
  that
  $J(t, \mu, \a^*; \cd, \a^*) = J(T_0, \psi; t, \mu, \a^*; \cd,
  \a^*)\in \tilde \dbV_0(t,\mu)$. That is,
  $\dbV_0(t,\mu)\subset \tilde \dbV_0(t,\mu)$.
\end{proof}

\subsection{The set value $\dbV_{state}$ } While Theorem
\ref{thm-DPP0} is elegant, the raw set value $\dbV_0(t, \mu)$ is very
sensitive to small perturbations of the coefficients $F, G$ and the
variable $\mu$. Indeed, even the measurability of the subset
$\dbV_0(t,\mu) \subset \dbR^d$ and the measurability of the mapping
$\mu\mapsto \dbV_0(t, \mu)$ are not clear to us. Moreover, in general
it does not look possible to have the convergence of the raw set value
of the corresponding $N$-player games to $\dbV_0(t,\mu)$. Therefore,
in this subsection we shall modify $\dbV_0(t,\mu)$ and introduce the
set value $\dbV_{state}(t,\mu)$ of the MFG as follows.

\begin{defn}
  \label{defn-MFEe}
  (i) For any $(t,\mu)\in \dbT \times \cP_0(\dbS)$ and $\e>0$, let
  $\cM^\e_{state}(t,\mu)$ denote the set of $\a^*\in \cA_{state}$ such
  that \bea
  \label{MFEe} J(t, \mu, \a^*; x, \a^*) \le v(\mu^{\a^*};t,
  x)+\e,\q\mbox{for all} ~x\in \dbS.  \eea

  (ii) The set value of the MFG at $(t,\mu)$ is defined as: \bea
  \label{Vtmu} &&\dis \qq\qq\qq\qq \dbV_{state}(t,\mu) :=
  \bigcap_{\e>0} \dbV_{state}^{\e}(t,\mu),\q\mbox{where}\\ &&\dis
  \!\!\!\!\!\!\!\!\!\!  \dbV_{state}^{\e}(t,\mu) := \Big\{\f\in
  \dbL^0(\dbS; \dbR): \|\f - J(t,\mu, \a^*; \cd, \a^*)\|_\infty \le
  \e~\mbox{for some}~ \a^*\in \cM^\e_{state}(t,\mu)\Big\}.\nonumber
  \eea
\end{defn}

Recall \reff{J}, then \reff{MFEe} and \reff{Vtmu} imply that \bea
\label{MFEev} 0\le J(t, \mu, \a^*; x, \a^*) - v( \mu^{\a^*}; t, x)\le
\e,\q \|\f - v(\mu^{\a^*}; t, \cd)\|_\infty \le 2\e.  \eea So we may
alternatively define $\dbV^\e_{state}(t, \mu)$ by using $\|\f -
v(\mu^{\a^*}; t, \cd)\|_\infty \le \e$.

\begin{rem}
  \label{rem-Ve}
  (i) In the case that there is only one player, namely $q, F, G$ do
  not depend on $\mu$, $\dbP^{\mu^{\a^*};t, x, \a} = \dbP^{t, x, \a}$
  does not depend on $\mu$ and $\a^*$. Let \beaa V(t,x) := \inf_{\a\in
    \cA_{state}} \dbE^{\dbP^{t,x,\a}} \Big[G(X_T) + \sum_{s=t}^{T-1}
  F(s, X_s, \a(s, X_s))\Big] \eeaa denote the value function of the
  standard stochastic control problem.  One can easily see that, when
  there exists an optimal control $\a^*$,
  $\dbV_0(t, \mu)=\dbV_{state}(t,\mu) = \{V(t, \cd)\}$. However, when
  there is no optimal control, we still have
  $\dbV_{state}(t,\mu) = \{V(t, \cd)\}$ but
  $\dbV_0(t, \mu) = \emptyset$. So 
  the natural extension of the value function $V$ is the set value $\dbV_{state}$, not $\dbV_0$.

  (ii) We remark that
  $\bigcap_{\e>0} \cM^\e_{state}(t,\mu) = \cM_{state}(t, \mu)$,
  however, in general it is possible that $\dbV_{state}(t, \mu)$ is
  strictly larger than $\dbV_0(t,\mu)$. Indeed, $\dbV_{state}(t,\mu)$
  can be even larger than the closure of $\dbV_0(t,\mu)$, where the
  latter is still empty when there is no optimal control.
\end{rem}

Similarly, given $T_0$ and $\psi$, $\cM^\e_{state}(T_0,\psi; t,\mu)$
denotes the set of $\a^*\in \cA_{state}$ such that \bea
\label{MFEe2} J(T_0,\psi; t, \mu, \a^*; x, \a^*) \le \inf_{\a\in
\cA_{state}}J(T_0,\psi; t, \mu, \a^*; x, \a)+\e,\q\forall ~x\in \dbS.
\eea The DPP remains true for $\dbV_{state}$ after appropriate
modifications as follows.

\begin{thm}
\label{thm-DPP1} Under Assumption \ref{assum-reg} (i), for any $0\le
t< T_0\le T$ and $\mu\in \cP_0(\dbS)$, \bea
\label{DPP1} \left.\ba{c} \dis \dbV_{state}(t,\mu) := \bigcap_{\e>0}
\Big\{\f\in \dbL^0(\dbS; \dbR): \|\f - J(T_0,\psi; t,\mu, \a^*; \cd,
\a^*)\|_\infty \le \e\\ \dis \mbox{for some}~\psi \in
\dbL^0(\dbS\times \cP_0(\dbS); \dbR)~\mbox{and}~\a^* \in \cA_{state}
~\mbox{such that}\\ \dis ~\psi(\cd, \mu^{\a^*}_{T_0}) \in
\dbV_{state}^{\e}(T_0, \mu^{\a^*}_{T_0}), ~\a^*\in \cM^\e_{state}(T_0,
\psi; t,\mu)\Big\}.  \ea\right.  \eea
\end{thm} 
This theorem can be proved by modifying the arguments in
Theorem \ref{thm-DPP0} and Proposition \ref{prop-DPP}. However, since
the proof is very similar to that of Theorem \ref{thm-DPP2} below,
except that the latter is in the more complicated path dependent
setting, we thus postpone it to Appendix.

 \section{The $N$-player game with homogeneous equilibria}
 \label{sect-N1}
  \setcounter{equation}{0} 
 In this section we study the
$N$-player game whose set value will converge to $\dbV_{state}$.
  
\subsection{The $N$-player game} 
\label{sect-stateN}
Set $\O^N:= \dbX^N$ with canonical
processes ${\vec X} = (X^1,\cds, X^N)$, where $X^i$ stands for the
state process of Player $i$.  The empirical measure of ${\vec X}$ is
denoted as: with the Dirac measure $\d_\cd$, \bea
\label{muN} \mu^N_t := \mu^N_{\vec X_t}\q\mbox{where}\q \mu^N_{\vec x}
:= {1\over N}\sum_{i=1}^N \d_{x_i} \in \cP(\dbS), ~\mbox{for}~\vec x
=(x_1,\cds, x_N)\in \dbS^N.  \eea The player $i$ will have control
$\a^i$. In the literature, a closed loop control $\a^i$ may depend on
the full information $\vec X$. However, since we are talking about
large $N$, in practice it may not be feasible for each player to
observe all other players' states individually. Moreover, in the MFG
setting the population state is characterized by its distribution, not
by each player's individual state. So in this section we consider only
symmetric controls, namely $\a^i$ depends on his/her own state $X^i$
and on the others through the empirical measure $\mu^N$.

In order to have the desired convergence, we introduce another
parameter $L\ge 0$.  Denote
\begin{equation}
  \label{cAstate} \cA^L_{state}:=\Big\{ \a: \dbT \times \dbS \times
  \cP(\dbS)\to \dbA: \big|\a(t,x, \mu) - \a(t, x, \nu)\big|\le L
  W_1(\mu,\nu), \forall t, x, \mu, \nu\Big\},  
\end{equation}
and $\cA^\infty_{state}:= \bigcup_{L\ge 0} \cA^L_{state}$. Given
$t\in \dbT$, ${\vec x}\in \dbS^N$, and
${\vec \a} = (\a^1,\cds, \a^N)\in (\cA^\infty_{state})^N$, let
$\dbP^{t,{\vec x}, {\vec \a}}$ denote the probability measure on
$\cF^{\vec X}_T$ determined recursively by: for $s=t,\cds, T-1$,
\begin{equation}
  \label{Ptxa} \dbP^{t,{\vec x}, {\vec \a}}(\vec X_t = \vec x) = 1,~\!
  \dbP^{t,{\vec x}, {\vec \a}}(\vec X_{s+1}=\vec x'' | \vec X_s = \vec
  x') = \prod_{i=1}^N q(s, x_i', \mu^N_s, \a^i(s, x_i', \mu^N_s);
  x''_i),  
\end{equation}
and the cost function of Player $i$ is: \bea
\label{Ji} J_i(t, \vec x, \vec \a) := \dbE^{\dbP^{t,{\vec x}, {\vec
\a}}}\Big[G(X^i_T, \mu^N_T) + \sum_{s=t}^{T-1} F(s, X^i_s, \mu^N_s,
\a^i(s, X^i_s, \mu^N_s))\Big].  \eea

\begin{rem}
  \label{rem-cAL}
  (i) It is obvious that $\cA^0_{state} = \cA_{state}$ for the
  $\cA_{state}$ in the previous subsection. For the MFG, there is no
  need to consider $\cA^\infty_{state}$. Indeed, given
  $(t, \mu)\in \dbT\times \cP_0(\dbS)$, for any
  $\a\in \cA^\infty_{state}$, let $\dbP^{t, \mu, \a}$ be defined as in
  \reff{Ptmua}: again denoting
  $\mu^\a_s:= \dbP^{t,\mu,\a} \circ X_s^{-1}$, \beaa \dbP^{t,\mu,\a}
  \circ X_t^{-1} = \mu,\q \dbP^{t,\mu,\a}(X_{s+1}=\tilde x| X_s = x) =
  q(s, x, \mu^\a_s, \a(s, x, \mu^\a_s); \tilde x).  \eeaa Introduce
  $\tilde \a(s, x) := \a(s, x, \mu^\a_s)$. Then
  $\tilde \a\in \cA_{state}$ and one can easily verify that
  $\mu^{\tilde \a} = \mu^\a$. In particular, the set value
  $\dbV_{state}(t, \mu)$ will remain the same by allowing
  $\a\in \cA^\infty_{state}$. For the $N$-player game, however, since
  $\mu^N$ is random, the dependence on $\mu^N$ makes the difference.

  (ii) In the literature one typically uses
  $\mu^{N, -i}_t := {1\over N-1} \sum_{j\neq i}\d_{X^j_t}$, rather
  than $\mu^N_t$, in \reff{Ptxa} and \reff{Ji}. The convergence
  results in this section will remain true if we use $\mu^{N, -i}$
  instead. However, we find it more convenient to use $\mu^N_t$.
\end{rem}

There is another crucial issue concerning the equilibria. Note that an
MFE requires by definition that each player takes the same control
$\a^*$. To achieve the desired convergence, for the $N$-player game it
is natural to consider only the homogeneous equilibria:
$\a_1=\cds=\a_N$, which we will do in the rest of this section.  We
note that, for a homogeneous control $\a$, the $\dbP^{t, \vec x,
\a}:=\dbP^{t, \vec x, (\a,\cds, \a)}$ in \reff{Ptxa} and $J_i(t, \vec
x, \a):=J_i(t, \vec x, (\a,\cds,\a))$ in \reff{Ji} are also symmetric
in $\vec x$, or say invariant in terms of its empirical measure: \bea
\label{JN} \left.\ba{c} \dis \dbP^{t, \vec x, \a} = \dbP^{t,
\mu^N_{\vec x}, \a},\q J_i(t, \vec x, \a) = J^N(t, x_i, \mu^N_{\vec
x}, \a).  \ea\right.  \eea

\begin{defn}
\label{defn-NE} 
For any $\e>0, L\ge 0$, we say $\a^*\in \cA^L_{state}$
is a homogeneous state dependent $(\e, L)$-equilibrium of the
$N$-player game at $(t, \vec x)$, denoted as $\a^*\in
\cM^{N,\e,L}_{state}(t, \vec x)$, if: \bea
\label{NE} \left.\ba{c} \dis J_i(t, \vec x, \a^*) \le v^{N, L}_i(t,
\vec x, \a^*) := \inf_{\tilde \a\in \cA^L_{state}} J_i(t, \vec x,
(\a^*, \tilde \a)_i) +\e,\q i=1,\cds, N,\\ \dis \mbox{where $(\a,
\tilde \a)_i$ denote the vector $\vec\a$ such that $\a^i = \tilde \a$
and $\a^j = \a$ for all $j\neq i$}.  \ea\right.  \eea
\end{defn}

In light of \reff{JN}, clearly $\cM^{N,\e,L}_{state}(t, \vec x)$ is
law invariant:
$\cM^{N,\e,L}_{state}(t, \vec x) =\cM^{N,\e,L}_{state}(t, \vec x')$
whenever $\mu^N_{\vec x} = \mu^N_{\vec x'}$. Thus, by abusing the
notation, we may denote
$\cM^{N,\e,L}_{state}(t, \vec x)=\cM^{N,\e,L}_{state}(t, \mu^N_{\vec
  x})$ and call $\a^*$ a homogeneous state dependent
$(\e, L)$-equilibrium at $(t, \mu^N_{\vec x})$.
 
 Note again that $q>0$, then similar to Subsection
\ref{sect-FiniteSymmetricMFG}, for convenience in this section we
restrict to only those $\vec x$ such that $\mu^N_{\vec x}$ has full
support, and we denote \bea
\label{SN0} \dbS^N_0:= \big\{\vec x\in \dbS^N: \mu^N_{\vec x} \in
\cP_0(\dbS)\big\},\q \cP_N(\dbS) := \big\{\mu^N_{\vec x}: \vec x\in
\dbS^N_0\big\}\subset \cP_0(\dbS).  \eea We now define the set value
of the homogeneous $N$-player game: recalling \reff{JN},
\begin{equation}
  \begin{aligned}
    \label{VtmuN} \dbV^N_{state}(t,\mu) &:=\bigcap_{\e>0}
    \dbV^{N,\e}_{state}(t,\mu):= \bigcap_{\e>0} \bigcup_{L\ge 0}
    \dbV^{N,\e,L}_{state}(t,\mu),~ \forall (t,\mu)\in \dbT\times
    \cP_N(\dbS),~\mbox{where}\\ \dis \dbV^{N,\e,L}_{state}(t,\mu) &:=
    \Big\{\f\in \dbL^0(\dbS; \dbR): \exists \a^*\in
    \cM^{N,\e,L}_{state}(t,\mu) ~\mbox{s.t.}~ \|\f - J^N(t, \cd, \mu,
    \a^*)\|_\infty \le \e\Big\}.
  \end{aligned}
\end{equation}

\begin{rem}
\label{rem-L}
Note that we require $\tilde \a\in  \cA^L_{state}$ in \reff{NE} for the same $L$, so $\bigcup_{L\ge 0}
    \dbV^{N,\e,L}_{state}(t,\mu)$ at above is in general different from $ \dbV^{N,\e,\infty}_{state}(t,\mu)$, which is defined in an obvious way by requiring  $\a^*, \tilde \a\in  \cA^\infty_{state}$ in \reff{NE}. See also Remark \ref{rem-regularity} (ii) below.
\end{rem}

\subsection{Convergence of the empirical measures}
\label{sect-FiniteSymConv}
\begin{thm}
\label{thm-FinSymMeasureConv} Let Assumption \ref{assum-reg} (ii)
hold. Then, for any $L\ge 0$, there exists a constant $C_L$, which
depends only on $T, d, L_q$, and $L$ such that, for any $t\in \dbT$,
$\vec x\in \dbS_0^N$, $\mu\in \cP_0(\dbS)$, $\a, \tilde \a\in
\cA^L_{state}$, and $s\ge t$, $i=1,\cds, N$, \bea
\label{FinSymMeasureEst1} &\dis \dbE^{\dbP^{t, \vec x, (\a, \tilde
\a)_i}}\big[\cW_1(\mu^N_s, \mu^\a_s)\big] \le C_L
\th_N,\q\mbox{where}\q\th_N:= W_1(\mu^N_{\vec x}, \mu) + {1\over
\sqrt{N}};\\
\label{FinSymMeasureEst2} &\dis \cW_1\Big(\dbP^{t, \vec x, (\a, \tilde
\a)_i}\circ (X^i_s)^{-1}, ~ \dbP^{\mu^\a; t, x_i, \tilde \a}\circ
X_s^{-1}\Big) \le C_L\th_N.  \eea
\end{thm}
\begin{proof} We first recall Remark \ref{rem-cAL} and extend all the
notations in Subsection \ref{sect-FiniteSymmetricMFG} to those $\a\in
\cA^L_{state}$ in the obvious sense.  Fix $t, i$ and denote
$\dbP^{N}:= \dbP^{t,\vec x ,(\a,\tilde \a)_i}$.

{\it Step 1.} We first prove \reff{FinSymMeasureEst1} for
$s=t+1$. Note that $X^1_{t+1},\cds, X^N_{t+1}$ are independent under
$\dbP^{N}$.  By \reff{W1}, we have \bea
\label{muNaest1} &&\dis \dbE^{\dbP^{N}}\big[W_1(\mu^N_{t+1},
\mu^\a_{t+1})\big] = \sum_{\tilde x\in \dbS}
\dbE^{\dbP^{N}}\big[|\mu^N_{t+1}(\tilde x)- \mu^\a_{t+1}(\tilde
x)|\big] \nonumber\\ &&\dis \le \sum_{\tilde x\in \dbS}
\Big(\dbE^{\dbP^{N}}\big[|\mu^N_{t+1}(\tilde x)- \mu^\a_{t+1}(\tilde
x)|^2\big] \Big)^{1\over 2}\nonumber\\ &&\dis = \sum_{\tilde x\in
\dbS} \Big[Var^{\dbP^{N}}\big[\mu^N_{t+1}(\tilde x)\big] +
\big(\dbE^{\dbP^{N}}\big[\mu^N_{t+1}(\tilde x)- \mu^\a_{t+1}(\tilde
x)\big]\big)^2\Big]^{1\over 2}\\ &&\dis = \sum_{\tilde x\in \dbS}
\Big[{1\over N^2}\sum_{j=1}^N
Var^{\dbP^{N}}\big[\1_{\{X^j_{t+1}=\tilde x\}}\big] + \big({1\over
N}\sum_{j=1}^N \dbP^{N}(X^j_{t+1}=\tilde x) -\mu^\a_{t+1}(\tilde
x)\big)^2\Big]^{1\over 2}\nonumber\\ &&\dis \le {C\over \sqrt{N}} +
\sum_{\tilde x\in \dbS}\big|{1\over N}\sum_{j=1}^N
\dbP^{N}(X^j_{t+1}=\tilde x) -\mu^\a_{t+1}(\tilde x)\big|.\nonumber
\eea Note that, by the desired Lipschitz continuity of $q$ in $\mu$
and that $|\dbS|=d$ is finite, \beaa &&\dis \big|{1\over
N}\sum_{j=1}^N \dbP^{N}(X^j_{t+1}=\tilde x) -\mu^\a_{t+1}(\tilde
x)\big| \\ &&\dis = \Big|{1\over N} \sum_{x\in \dbS} \Big[\sum_{j\neq
i}q(t, x, \mu^N_{\vec x}, \a(t, x, \mu^N_{\vec x}); \tilde x)
\1_{\{x_j=x\}} + q(t, x, \mu^N_{\vec x}, \tilde\a(t, x, \mu^N_{\vec
x}); \tilde x) \1_{\{x_i=x\}}\Big]\\ &&\dis\qq - \sum_{x\in \dbS} q(t,
x, \mu, \a(t, x, \mu); \tilde x)\mu(x)\Big|\\ &&\dis \le \Big|{1\over
N} \sum_{x\in \dbS} \sum_{j =1}^N q(t, x, \mu^N_{\vec x}, \a(t, x,
\mu^N_{\vec x}); \tilde x) \1_{\{x_j=x\}} - \sum_{x\in \dbS} q(t, x,
\mu, \a(t, x, \mu); \tilde x)\mu(x)\Big|\\ &&\dis \qq +{1\over N}
\sum_{x\in \dbS} \big|q(t, x, \mu^N_{\vec x}, \a(t, x, \mu^N_{\vec
x}); \tilde x) - q(t, x, \mu^N_{\vec x}, \tilde\a(t, x, \mu^N_{\vec
x}); \tilde x)\big|\1_{\{x_i=x\}} \\ &&\dis \le \big|\sum_{x\in \dbS}
q(t, x, \mu^N_{\vec x}, \a(t, x, \mu^N_{\vec x}); \tilde x)
\mu^N_{\vec x}(x)- \sum_{x\in \dbS} q(t, x, \mu, \a(t, x, \mu); \tilde
x)\mu(x)\big| +{1\over N}\\ &&\dis \le \sum_{x\in
\dbS}\Big[|\mu^N_{\vec x}(x) - \mu(x)| + C_LW_1(\mu^N_{\vec x}, \mu)
\mu(x)\Big] +{1\over N} \le C_L\th_N.  \eeaa Then, $\dis
\dbE^{\dbP^{N}}\big[\cW_1(\mu^N_{t+1}, \mu^\a_{t+1})\big] \le {C\over
\sqrt{N}} + C_L\th_N \le C_L \th_N.$

{\it Step 2.} We next prove \reff{FinSymMeasureEst1} by induction. For
any $s=t,\cds, T-1$, by Step 1 we have \beaa
\dbE^{\dbP^{N}}\big[\cW_1(\mu^N_{s+1}, \mu^\a_{s+1})\big| \cF^{\vec
X}_s\big] \le C_L \Big[W_1(\mu^N_s, \mu^\a_s) + {1\over
\sqrt{N}}\Big],\q \dbP^{N}\mbox{-a.s.}  \eeaa Then \beaa
\dbE^{\dbP^{N}}\big[\cW_1(\mu^N_{s+1}, \mu^\a_{s+1})\big] =
\dbE^{\dbP^{N}}\Big[ \dbE^{\dbP^{N}}\big[\cW_1(\mu^N_{s+1},
\mu^\a_{s+1})\big| \vec X^N_s\big] \Big] \le C_L
\dbE^{\dbP^{N}}\big[\cW_1(\mu^N_s, \mu^\a_s)\big] + {C_L\over
\sqrt{N}}.  \eeaa Since $T$ is finite, by induction we obtain
\reff{FinSymMeasureEst1} immediately.

{\it Step 3.} We now prove \reff{FinSymMeasureEst2}. Denote \beaa
\k_s:= W_1\Big(\dbP^N\circ (X^i_s)^{-1}, ~ \dbP^i\circ
X_s^{-1}\Big)\q\mbox{where}\q \dbP^i:= \dbP^{\mu^\a; t, x_i, \tilde
\a}.  \eeaa Then $\k_t =0$, and for $s=t,\cds, T-1$, \beaa &&\dis
\k_{s+1} = \sum_{\tilde x\in \dbS} \big|\dbP^N(X^i_{s+1} = \tilde x)
-\dbP^i(X_{s+1}=\tilde x)\big|\\ &&\dis =\sum_{\tilde x\in \dbS} \Big|
\dbE^{\dbP^N}\big[q(s, X^i_{s}, \mu^N_s, \tilde \a(s, X^i_s, \mu^N_s);
\tilde x)\big] - \dbE^{\dbP^i}\big[q(s, X_{s}, \mu^\a_s, \tilde \a(s,
X_s, \mu^\a_s); \tilde x)\big]\Big|\\ &&\dis \le \sum_{\tilde x\in
\dbS} \Big| \dbE^{\dbP^N}\big[q(s, X^i_{s}, \mu^N_s, \tilde \a(s,
X^i_s, \mu^N_s); \tilde x)\big] - \dbE^{\dbP^N}\big[q(s, X^i_{s},
\mu^\a_s, \tilde \a(s, X^i_s, \mu^\a_s); \tilde x)\big]\Big|\\
&&\q+\sum_{\tilde x\in \dbS} \Big| \dbE^{\dbP^N}\big[q(s, X^i_{s},
\mu^\a_s, \tilde \a(s, X^i_s, \mu^\a_s); \tilde x)\big] -
\dbE^{\dbP^i}\big[q(s, X_{s}, \mu^\a_s, \tilde \a(s, X_s, \mu^\a_s);
\tilde x)\big]\Big|\\ &&\le C_L \dbE^{\dbP^N}\big[W_1(\mu^N_s,
\mu^\a_s)\big]+\sum_{x, \tilde x\in \dbS} q(s, x, \mu^\a_s, \tilde
\a(s, x, \mu^\a_s); \tilde x)\big|\dbP^N(X^i_s = x) - \dbP^i(X_s =
x)\big|\\ &&\le C_L \th_N+\k_s, \eeaa where the last inequality thanks
to \reff{FinSymMeasureEst1}. Now by induction one can easily prove
\reff{FinSymMeasureEst2}.
\end{proof}

\subsection{Convergence of the set values}

We first study the convergence of the cost functions.  Recall the
$\th_N$ in \reff{FinSymMeasureEst1} and the functions $v$ in \reff{J}
and $v^{N, L}_i$ in \reff{NE}.
\begin{thm}
  \label{thm-FinSymCostConv} Let Assumption \ref{assum-reg} (ii) and
  (iii) hold. For any $L\ge 0$, there exists a modulus of continuity
  function $\rho_L$, which depends only on $T, d, L_q$, $C_0$, $\rho$,
  and $L$ such that, for any $t\in \dbT$,
  $\mu^N_{\vec x}\in \cP_N(\dbS)$, $\mu\in \cP_0(\dbS)$, and any
  $\a, \tilde \a\in \cA^L_{state}$, $i=1,\cds, N$,
  \begin{equation}
    \label{FinSymCostEst} \big|J_i(t, \vec x, (\a,\tilde \a)_i) - J(t,
    \mu, \a; x_i, \tilde \a)\big| + \big|v^{N, L}_i(t, \vec x, \a) -
    v(\mu^\a; t, x_i)\big|\le \rho_L(\th_N).  
  \end{equation}
\end{thm}
\begin{proof} Clearly the uniform estimates for $J$ implies that for
$v$, so we shall only prove the former one. Recall \reff{Ji},
\reff{J}, and the notations $\dbP^N$, $\dbP^i$ in the proof of Theorem
\ref{thm-FinSymMeasureConv}. Then \beaa &&\dis \Big|J_i(t, \vec x,
(\a,\tilde \a)_i) - J(t, \mu, \a; x_i, \tilde \a)\Big| \le I_T +
\sum_{s=t}^{T-1} I_s,\q\mbox{where}\\ &&\dis I_T:=
\Big|\dbE^{\dbP^N}\big[G(X^i_T, \mu^N_T)\big] -
\dbE^{\dbP^i}\big[G(X_T, \mu^\a_T)\big]\Big|;\\ &&\dis I_s:=
\Big|\dbE^{\dbP^N}\big[ F(s, X^i_s, \mu^N_s, \tilde\a(s, X^i_s,
\mu^N_s))\big] - \dbE^{\dbP^i}\big[ F(s, X_s, \mu^\a_s, \tilde\a(s,
X_s, \mu^\a_s))\big]\Big|,\q s<T.  \eeaa Note that, for $s<T$, by
\reff{FinSymMeasureEst2}, \beaa I_s &\le& \Big|\dbE^{\dbP^N}\big[ F(s,
X^i_s, \mu^N_s, \tilde\a(s, X^i_s, \mu^N_s))\big] - \dbE^{\dbP^N}\big[
F(s, X^i_s, \mu^\a_s, \tilde\a(s, X^i_s, \mu^\a_s))\big]\Big|\\ &&+
\Big| \dbE^{\dbP^N}\big[ F(s, X^i_s, \mu^\a_s, \tilde\a(s, X^i_s,
\mu^\a_s))\big]- \dbE^{\dbP^i}\big[ F(s, X_s, \mu^\a_s, \tilde\a(s,
X_s, \mu^\a_s))\big]\Big|\\ &\le&
\dbE^{\dbP^N}\big[\rho\big(C_LW_1(\mu^N_s, \mu^\a_s)\big)\big] +
\sum_{x\in \dbS} \big|F(s, x, \mu^\a_s, \tilde\a(s, x,
\mu^\a_s))\big|\big|\dbP^N(X^i_s=x) - \dbP^i(X_s=x)\big|\\
&\le&\dbE^{\dbP^N}\big[\rho\big(C_LW_1(\mu^N_s, \mu^\a_s)\big)\big]+
C_L\th_N.  \eeaa Similarly we have the estimate for $I_T$, and thus
\beaa \Big|J_i(t, \vec x, (\a,\tilde \a)_i) - J(t, \mu, \a; x_i,
\tilde \a)\Big| \le
\sum_{s=t}^T\dbE^{\dbP^N}\big[\rho\big(C_LW_1(\mu^N_s,
\mu^\a_s)\big)\big]+ C_L\th_N.  \eeaa This, together with
\reff{FinSymMeasureEst1}, implies \reff{FinSymCostEst} for some
appropriately defined modulus of continuity function $\rho_L$.
\end{proof}

Our main result of this section is the following convergence of the
set values. Recall, for a sequence of sets $\{E_N\}_{N\ge 1}$, $\dis
\limsup_{N\to\infty} E_N := \bigcap_{n \ge 1} \bigcup_{N\ge n} E_N$,
$\dis\liminf_{N\to\infty} E_N := \bigcup_{n \ge 1} \bigcap_{N\ge n}
E_N$.

\begin{thm}
\label{thm-FinSymConv} 
Let Assumption \ref{assum-reg} (ii), (iii) hold
and $\mu^N_{\vec x}\in \cP_N(\dbS)\to \mu\in \cP_0(\dbS)$. Then
 \bea
 \label{FinSymConv} \bigcap_{\e>0}\bigcup_{L\ge 0}\limsup_{N\to
\infty} \dbV^{N,\e,L}_{state}(t,\mu^N_{\vec x}) \subset
\dbV_{state}(t,\mu) \subset \bigcap_{\e>0}
\liminf_{N\to\infty}\dbV^{N, \e,0}_{state}(t,\mu^N_{\vec x}) 
\eea 
In particular, since $\dis \liminf_{N\to\infty}\dbV^{N,
\e,0}_{state}(t,\mu^N_{\vec x}) \subset \bigcup_{L\ge 0}\limsup_{N\to
\infty} \dbV^{N, \e,L}_{state}(t,\mu^N_{\vec x})$, actually equalities
hold.
\end{thm} Note that $\vec x\in \dbS_0^N$ obviously depends on $N$, so
more rigorously we should write $\vec x^N$ in the above
statements. For notational simplicity we omit this $N$ here. We also
remark that at above we are not able to switch the order of
$\limsup_{N\to\infty}$ and $\bigcap_{\e>0}\bigcup_{L\ge 0}$ in the
left side, or the order of $\liminf_{N\to\infty}$ and $\bigcap_{\e>0}$
in the right side.

\begin{proof}
  (i) We first prove the right inclusion in \reff{FinSymConv}. Fix
  $\f\in \dbV_{state}(t,\mu)$, $\e>0$, and set $\e_1 := {\e\over
    2}$. Note that $\cA_{state} = \cA^0_{state}$.  By \reff{Vtmu},
  there exists $\a^* \in \cM^{\e_1}_{state}(t,\mu)$ such that
  $ \|\f - J(t,\mu, \a^*; \cd, \a^*)\|_\infty \le \e_1.  $ Recall
  \reff{MFEe}, we have \beaa J(t, \mu, \a^*; x, \a^*) \le
  v(\mu^{\a^*};t, x)+\e_1,\q\mbox{for all} ~x\in \dbS.  \eeaa For any
  $\a\in \cA^0_{state}=\cA_{state}$, by Theorem
  \ref{thm-FinSymCostConv} we have \beaa &\dis J_i(t, \vec x, \a^*)
  \le J(t, \mu, \a^*; x_i, \a^*)+ \rho_0( \th_N) \\ &\dis \le
  v(\mu^{\a^*};t, x)+\e_1 + \rho_0(\th_N) \le v^{N, L}_i(t, \vec x,
  \a^*)+\e_1+ 2\rho_0(\th_N).  \eeaa Choose $N$ large enough such that
  $\rho_0(\th_N)\le {\e\over 4}$, then
  $J_i(t, \vec x, \a^*) \le v^{N, L}_i(t, \vec x, \a^*)+\e$.  This
  implies that $\a^* \in \cM^N_{\e,0}(t, \mu^N_{\vec x})$. Moreover,
  \beaa \|\f - J^N(t, \cd, \mu^N_{\vec x}, \a^*)\|_\infty &\le& \e_1 +
  \sup_{i} \Big|J_i(t, \vec x, \a^*) - J(t, \mu, \a^*; x_i,
  \a^*)\Big|\\ &\le&\e_1 + \rho_0(\th_N) \le \e_1 + {\e\over 4} \le
  \e.  \eeaa Then $\f\in \dbV^{N,\e,0}_{state}(t, \mu^N_{\vec x})$ for
  all $N$ large enough. That is,
  $\f\in \liminf_{N\to\infty} \dbV^{N,\e,0}_{state}(t, \mu^N_{\vec
    x})$.  Since $\f\in \dbV_{state}(t,\mu)$ and $\e>0$ are arbitrary,
  we obtain the right inclusion in \reff{FinSymConv}.

  (ii) We next show the left inclusion in \reff{FinSymConv}. Fix
  $\dis \f\in \bigcap_{\e>0}\bigcup_{L\ge 0}\limsup_{N\to \infty}
  \dbV^{N,\e,L}_{state}(t,\mu^N_{\vec x})$ and $\e>0$. Then, for
  $\e_1:={\e\over 2}>0$, there exist $L_\e>0$ and an infinite sequence
  $\{N_k\}_{k\ge 1}$ such that
  $\f\in \dbV^{N_k, \e_1,L_\e}_{state}(t,\mu^{N_k}_{\vec x})$ for all
  $k\ge 1$. Recall \reff{VtmuN}, for each $k\ge 1$ there exists
  $\a^k\in \cM^{N_k, \e_1,L_\e}_{state}(t,\mu^{N_k}_{\vec x})$ such
  that
  $\|\f - J^N(t, \cd, \mu^{N_k}_{\vec x}, \a^k)\|_\infty \le \e_1$. By
  Definition \ref{defn-NE}, we have
  $ J_i(t, \vec x,\a^k) \le v^{N_k, L_\e}_i(t, \vec x, \a^k)+\e_1.  $
  Similar to (i), by Theorem \ref{thm-FinSymCostConv} we have \beaa
  J(t, \mu, \a^k; x_i, \a^k) \le v(\mu^{\a^k}; t, x_i)+\e_1 +
  2\rho_{L_\e}(\th_{N_k}) \le v(\mu^{\a^k}; t, x_i)+\e, \eeaa for $k$
  large enough.  That is, $\a^k\in \cM^\e_{state}(t, \mu)$. Similar to
  (i) again, for $k$ large enough we have
  $\|\f - J(t,\mu, \a^k; \cd, \a^k)\|_\infty \le \e$. Then
  $\f \in \dbV_{state}^\e(t, \mu)$. Since $\e>0$ is arbitrary, we
  obtain $\f\in \dbV_{state}(t,\mu)$, and hence derive the left
  inclusion in \reff{FinSymConv}.
\end{proof}

\begin{rem}
  \label{rem-conv}
  (i) From Theorem \ref{thm-FinSymConv} (i) we see that, for any
  $\a^* \in \cM^{\e\over 2}_{state}(t,\mu)$, we have
  $\a^* \in \cM^{N,\e,0}_{state}(t, \mu^N_{\vec x})$ when $N$ is large
  enough. Moreover, by \reff{FinSymMeasureEst1} we have the desired
  estimate for the approximate equilibrium measure
  $\dbE^{\dbP^{t, \vec x, \a^*}}\big[W_1(\mu^N_s, \mu^{\a^*}_s)\big]
  \le C_L \th_N$. This verifies the standard result in the literature
  that an approximate MFE is an approximate equilibrium of the
  $N$-player game.

  (ii) From Theorem \ref{thm-FinSymConv} (ii) we see that, for any
  $\a^k\in \cM^{N_k,{\e\over 2},L_\e}_{state}(t,\mu^{N_k}_{\vec x})$,
  we have $\a^k\in \cM^\e_{state}(t, \mu)$ when $k$ is large enough,
  and we again have the estimate for the approximate equilibrium
  measure
  $\dbE^{\dbP^{t, \vec x, \a^k}}\big[W_1(\mu^{N_k}_s,
  \mu^{\a^k}_s)\big] \le C_L \th_{N_k}$. This is in the spirit that
  any limit point of the $N$-player equilibrium measures is an MFE
  measure.
\end{rem}

\begin{rem}
  \label{rem-regularity}
  (i) We should point out that the key to obtain the convergence here
  is to consider homogeneous equilibria for the $N$-player games. If
  we use heterogeneous equilibria for the $N$-player games, it turns
  out that we will have the desired convergence when we consider
  relaxed controls for the MFG, as we will do in the next two
  sections.

  (ii) Another feature of our convergence result is the uniform Lipschitz
  continuity requirement on the admissible controls.  Indeed, the left inclusion in \reff{FinSymConv} would fail in general if we replace $\dis\bigcap_{\e>0}\bigcup_{L\ge 0}\limsup_{N\to
\infty} \dbV^{N,\e,L}_{state}(t,\mu^N_{\vec x})$ with $\dis\bigcap_{\e>0}\limsup_{N\to
\infty} \dbV^{N,\e,\infty}_{state}(t,\mu^N_{\vec x})$ or with $\dis\bigcap_{\e>0}\limsup_{N\to
\infty} \dbV^{N,\e}_{state}(t,\mu^N_{\vec x})$, where $\dbV^{N,\e,\infty}_{state}$ is defined in Remark \reff{rem-L} and $\dbV^{N,\e}_{state}$ is defined similarly, by requiring $\a^*, \tilde \a: \dbT \times \dbS \times
  \cP(\dbS)\to \dbA$ in \reff{NE} to be measurable only. See Example \ref{eg-Lipschitz} below. We refer to \cite{Lacker2, DanielLuc, Djete} for some related convergence analysis without  such regularity requirement.
  
(iii) We note that the above regularity requirement on the admissible controls is also crucial for numerical computations of set values, as well as for practical implementation of the equilibria, although these issues are not studied in the present paper.  
\end{rem}

\section{Mean field games on finite space with relaxed controls}
\label{sect-relax} \setcounter{equation}{0} In this section we study
MFG with relaxed controls, or say mixed strategies. Besides its
independent interest, our main motivation is to characterize the limit
of $N$-player games with heterogeneous equilibria. We shall still
consider the finite space in Section \ref{sect-state}, however, for
the purpose of generality in this section we consider path dependent
setting.

 \subsection{The relaxed set value with path dependent controls} 
 We
start with some notations for the path dependent setting. For $\bx =
(\bx_t)_{0\le t\le T}\in \dbX$, denote by $\bx_{t\wedge \cd} =
(\bx_0,\cds, \bx_t, \bx_t,\cds, \bx_t)$ the path stopping at $t$ and
$\dbX_t := \{\bx_{t\wedge \cd}: \bx\in \dbX\}\subset \dbX$. For $\bx,
\tilde \bx \in \dbX$, we say $\bx =_t \tilde \bx$ if $\bx_{t\wedge
\cd}=\tilde \bx_{t\wedge \cd}$. Denote $\dbX^{t,\bx} := \{\tilde
\bx\in \dbX: \tilde \bx =_t \bx\}$ and $\dbX^{t,\bx}_s:= \dbX^{t,\bx}
\cap \dbX_s$, for $s\ge t$.  Introduce the concatenation $\bx\oplus_t
\tilde \bx\in \dbX$ by \beaa (\bx\oplus_t \tilde \bx)_s := \bx_s
\1_{\{s\le t\}} + \tilde \bx_s \1_{\{s>t\}},\q\mbox{and}\q
(\bx\oplus_t x)_s := \bx_s \1_{\{s\le t\}} + x \1_{\{s>t\}}, ~x\in
\dbS.  \eeaa For each $t\in \dbT$, let $\cP(\dbX_t)$ denote the set of
probability measures on $(\O, \cF^X_t)$, equipped with \beaa W_1(\mu,
\nu):= \sum_{\bx\in \dbX_t} |\mu(\bx)-\nu(\bx)|,\q\forall \mu, \nu\in
\cP(\dbX_t), \eeaa and $\cP_0(\dbX_t)$ the subset of $\mu\in
\cP(\dbX_t)$ with full support $\dbX_t$. Again this is just for
convenience of presentation.  For a measure $\mu\in
\cP(\dbX)=\cP(\dbX_T)$, denote $\mu_{t\wedge \cd} := \mu \circ
X_{t\wedge \cd}^{-1} \in \cP(\dbX_t)$. We remark that, by abusing the
notation $\mu$, here $\mu_{t\wedge \cd}$ denote the joint law of the
stopped process $X_{t\wedge \cd}$, while in Section \ref{sect-state}
$\{\mu_\cd\}$ denote the family of marginal laws.
 
For a path dependent function $\f$ on $\dbT\times \dbX\times
\cP(\dbX)$, we say $\f$ is adapted if $\f(t, \bx, \mu) = \f(t,
\bx_{t\wedge \cd}, \mu_{t\wedge \cd})$. Throughout this section, all
the path dependent functions are required to be adapted. In
particular, the data of the game $q: \dbT \times \dbX \times \cP(\dbX)
\times \dbA\times \dbS\to (0,1)$, $F: \dbT \times \dbX\times \cP(\dbX)
\times \dbA \to \dbR$, and $G: \dbX\times \cP(\dbX)\to \dbR$ are path
dependent with $q, F$ adapted. By adapting to the path dependent
setting, we shall still assume Assumption \ref{assum-reg}.

Let $\cA_{relax}$ denote the set of path dependent adapted relaxed
controls $\g: \dbT \times \dbX \to \cP(\dbA)$.  Given $t\in \dbT$,
$\mu\in \cP(\dbX_t)$, $\g\in \cA_{relax}$, and $\bx \in \dbX_t$,
$\tilde \bx\in \dbX^{t,\bx}$, $\tilde \g\in \cA_{relax}$, we
introduce:
\begin{equation}
  \begin{aligned}
    \label{relaxJ}
    &\dbP^{t,\mu,\g} \circ X_{t\wedge\cd}^{-1} = \mu,\q
    \dbP^{t,\mu,\g}(X_{s+1}=\tilde x| X =_s
    \bx) = \int_{\dbA} q(s, \bx, \mu^\g, a; \tilde x) \g(s, \bx; da);
    \\&\hspace{8.3em} \mbox{where}\q
    \mu^\g_{s\wedge \cd}:= \dbP^{t,\mu,\g} \circ
    X_{s\wedge \cd}^{-1},\q s\ge t;
    \\&\dbP^{\mu^\g; t, \bx,
      \tilde \g}(X=_t\bx) = 1,\q \dbP^{\mu^\g; t, \bx, \tilde
      \g}(X_{s+1}= \bar x| X =_s \tilde \bx) =\int_{\dbA} q(s, \tilde
    \bx, \mu^\g, a; \bar x) \tilde \g(s, \tilde \bx; da);
    \\& J(\mu^\g; s, \tilde \bx, \tilde \g) := \dbE^{\dbP^{\mu^\g; t,
        \bx,\tilde \g}}\Big[G(X, \mu^\g) + \sum_{r=s}^{T-1}
    \int_{\dbA} F(r, X, \mu^\g, a) \tilde\g(r, X, da)\Big|X =_s \tilde
    \bx\Big];
    \\& J(t, \mu, \g; \bx, \tilde \g) := J(\mu^\g; t,
    \bx, \tilde\g),\q v(\mu^\g; s, \tilde\bx) := \inf_{\tilde \g\in
      \cA_{relax}} J(\mu^\g; s, \tilde \bx, \tilde \g).
  \end{aligned}
\end{equation}

\begin{defn}
  \label{defn-rMFEe}
  (i) For any $t\in \dbT$, $\mu\in \cP_0(\dbX_t)$, and $\e>0$, let
  $\cM^\e_{relax}(t,\mu)$ denote the set of relaxed $\e$-MFE
  $\g^*\in \cA_{relax}$ such that \bea
  \label{rMFEe}
  J(t, \mu, \g^*; \bx, \g^*) \le v(\mu^{\g^*}; t, \bx)+\e,\q\mbox{for
    all} ~\bx\in \dbX_t.  \eea

  (ii) The relaxed set value of the MFG at $(t,\mu)$ is defined as:
  \bea
  \label{rVtmu}
  &&\dis \qq\dbV_{relax}(t,\mu) := \bigcap_{\e>0}  \dbV_{relax}^\e(t,\mu),\q\mbox{where}~\|\f\|_{\dbX_t} := \sup_{\bx\in \dbX_t}|\f(\bx)|,~\mbox{and}\\
  &&\dis\!\!\!\!\!\!\!\!\!\!\!\!  \dbV_{relax}^\e(t,\mu) :=
  \Big\{\f\in \dbL^0(\dbX_t; \dbR): \exists \g^*\in
  \cM^\e_{relax}(t,\mu)~\mbox{s.t.}~ \|\f - J(t,\mu, \g^*; \cd,
  \g^*)\|_{\dbX_t} \le \e\Big\}.\nonumber 
  \eea
\end{defn}

Similarly, given $T_0$ and $\psi: \dbX_{T_0}\times \cP(\dbX_{T_0})\to
\dbR$, as in \reff{JT0} define
\begin{equation}
  \label{rJT0} J(T_0,\psi; t,\mu, \g; \bx, \tilde \g):=
  \dbE^{\dbP^{\mu^\g; t, \bx, \tilde\g}}\Big[\psi(X_{T_0\wedge \cd},
  \mu^\g_{T_0\wedge \cd}) \!+\! \sum_{s=t}^{T_0-1}\! \int_{\dbA}\! F(s,
  X, \mu^\g, a) \tilde \g(s, X, da)\Big],
\end{equation}
and let $\cM^\e_{relax}(T_0,\psi; t,\mu)$ denote the set of
$\g^*\in \cA_{relax}$ such that, $\forall \bx\in \dbX_t$,
\begin{equation}
  \label{rMFEe2}
  J(T_0,\psi; t, \mu, \g^*; \bx, \g^*) \le
  v(T,\psi;\mu^\g; s, \bx):=\inf_{\g\in \cA_{relax}}J(T_0,\psi; t,
  \mu, \g^*; \bx, \g)+\e.  
\end{equation}
Note that the tower property in \reff{tower} remains true for relaxed
controls: \bea
\label{rtower} \dis J(t, \mu, \g; \bx, \tilde\g) = J(T_0, \psi; t,
\mu, \g; \bx, \tilde\g),\q \mbox{where}\q \psi(\by, \nu) := J(T_0,
\nu, \g; \by, \tilde\g).  \eea The DPP for $\dbV_{relax}$ takes the
following form.

\begin{thm}
\label{thm-DPP2} Under Assumption \ref{assum-reg} (i), for any $t\in
\dbT$, $T_0\in \dbT_t$, and $\mu\in \cP_0(\dbX_t)$, \bea
\label{DPP2} \left.\ba{c} \dis \dbV_{relax}(t,\mu) = \bigcap_{\e>0}
\Big\{\f\in \dbL^0(\dbX_t; \dbR): \|\f - J(T_0,\psi; t,\mu, \g^*; \cd,
\g^*)\|_{\dbX_t} \le \e\\ \dis \mbox{for some}~\psi \in
\dbL^0(\dbX_{T_0}\times \cP_0(\dbX_{T_0}); \dbR)~\mbox{and}~\g^* \in
\cA_{relax} ~\mbox{such that}\\ \dis ~\psi(\cd, \mu^{\g^*}_{T_0\wedge
\cd}) \in \dbV_{relax}^\e(T_0, \mu^{\g^*}_{T_0\wedge\cd}), ~\g^*\in
\cM^\e_{relax}(T_0, \psi; t,\mu)\Big\}.  \ea\right.  \eea
\end{thm}
\begin{proof}
  We shall follow the arguments in Theorem \ref{thm-DPP0}, in
  particular, we shall extend Proposition \ref{prop-DPP}. Let
  $\tilde \dbV_{relax}(t,\mu)= \bigcap_{\e>0}\tilde
  \dbV_{relax}^\e(t,\mu)$ denote the right side of \reff{DPP2}.

  (i) We first prove
  $\tilde \dbV_{relax}(t,\mu)\subset \dbV_{relax}(t,\mu)$. Fix
  $\f\in \tilde \dbV_{relax}(t,\mu)$, $\e>0$, and set
  ${\e_1}:={ \e\over 4}$. Since
  $\f\in \tilde \dbV_{relax}^{\e_1}(t, \mu)$, then \beaa \|\f-J(T_0,
  \psi; t, \mu, \g^*; \cd, \g^*)\|_{\dbX_t} \le {\e_1}\q\mbox{ for
    some desirable $\psi, \g^*$ as in \reff{DPP2}}.  \eeaa Since
  $\psi(\cd, \mu^{\g^*}_{T_0\wedge \cd}) \in \dbV_{relax}^{\e_1}(T_0,
  \mu^{\g^*}_{T_0\wedge \cd})$, there exists
  $\tilde \g^*\in \cM_{relax}^{\e_1}(T_0, \mu^{\g^*}_{T_0\wedge \cd})$
  such that \beaa \|\psi(\cd, \mu^{\g^*}_{T_0\wedge \cd}) - J(T_0,
  \mu^{\g^*}_{T_0\wedge \cd}, \tilde \g^*; \cd, \tilde \g^*)\|_{
    \dbX_{T_0}}\le {\e_1}.  \eeaa As in \reff{Oplus} denote
  $\hat \g^*:= \g^*\oplus_{T_0} \tilde \g^* := \g^* \1_{\{s<T_0\}} +
  \tilde \g^* \1_{\{s\ge T_0\}}\in \cA_{relax}$. Then, for any
  $\bx\in \dbX_t$ and $\g\in \cA_{relax}$, similarly to Proposition
  \ref{prop-DPP} (i) we have \beaa &&J(t, \mu, \hat\g^*; \bx, \g) \\
  &&= \dbE^{\dbP^{\mu^{\g^*};t, \bx, \g}}\Big[ J(T_0,
  \mu^{\g^*}_{T_0\wedge \cd}, \tilde \g^*; X_{T_0\wedge \cd}, \g) +
  \sum_{s=t}^{T_0-1} \int_{\dbA}F(s, X, \mu^{\g^*},a) \g(s, X,
  da)\Big]\\ &&\ge \dbE^{\dbP^{\mu^{\g^*}; t, \bx, \g}}\Big[ J(T_0,
  \mu^{\g^*}_{T_0\wedge \cd}, \tilde \g^*; X_{T_0\wedge \cd}, \tilde
  \g^*) + \sum_{s=t}^{T_0-1} \int_{\dbA}F(s, X, \mu^{\g^*}, a)\g(s, X,
  da)\Big]-{\e_1}\\ &&\ge \dbE^{\dbP^{\mu^{\g^*};t, \bx, \g}}\Big[
  \psi(X_{T_0\wedge \cd}, \mu^{\g^*}_{T_0\wedge \cd}) +
  \sum_{s=t}^{T_0-1} \int_{\dbA}F(s, X, \mu^{\g^*}, a)\g(s, X,
  da)\Big] - 2{\e_1}\\ &&= J(T_0, \psi; t, \mu, \g^*; \bx, \g)
  -2{\e_1} \ge J(T_0, \psi; t, \mu, \g^*; \bx, \g^*) -3{\e_1} \\ &&=
  \dbE^{\dbP^{\mu^{\g^*}; t, \bx, \g^*}}\Big[ \psi(X_{T_0\wedge \cd},
  \mu^{\g^*}_{T_0\wedge \cd}) + \sum_{s=t}^{T_0-1} \int_{\dbA}F(s, X,
  \mu^{\g^*}, a)\g^*(s, X, da)\Big] - 3{\e_1}\\ &&\ge
  \dbE^{\dbP^{\mu^{\g^*};t, \bx, \g^*}}\Big[ J(T_0,
  \mu^{\g^*}_{T_0\wedge \cd}, \tilde \g^*; X_{T_0\wedge \cd}, \tilde
  \g^*) + \sum_{s=t}^{T_0-1} \int_{\dbA}F(s, X, \mu^{\g^*}, a)\g^*(s,
  X, da)\Big] - 4{\e_1}\\ &&= J(t, \mu, \hat\g^*; \bx, \hat \g^*) -
  4\e_1 = J(t, \mu, \hat\g^*; \bx, \hat \g^*) - \e.  \eeaa That is,
  $\hat \g^* \in \cM_{relax}^\e(t,\mu)$. Moreover, note that, by
  \reff{rtower}, \beaa &&\|\f - J(t, \mu, \hat\g^*; \cd,
  \hat\g^*)\|_{\dbX_t} \le {\e_1} + \|J(T_0, \psi; t, \mu, \g^*; \cd,
  \g^*)-J(t, \mu, \hat\g^*; \cd, \hat\g^*)\|_{\dbX_t}\\ &&={\e_1} +
  \sup_{\bx\in \dbX_t} \Big|\dbE^{\dbP^{\mu^{\g^*}; t, \bx,
      \g^*}}\big[ \psi(X_{T_0\wedge \cd}, \mu^{\g^*}_{T_0\wedge \cd})
  - J(T_0, \mu^{\g^*}_{T_0\wedge \cd}, \tilde \g^*; X_{T_0\wedge \cd},
  \tilde \g^*)\big]\Big| \le 2{\e_1} <\e.  \eeaa Then
  $\f\in \dbV_{relax}^\e(t, \mu)$. Since $\e>0$ is arbitrary, we
  obtain $\f\in \dbV_{relax}(t,\mu)$.

  (ii) We now prove the opposite inclusion. Fix
  $\f\in \dbV_{relax}(t,\mu)$ and $\e>0$.  Let ${\e_2}>0$ be a small
  number which will be specified later. Since
  $\f\in \dbV_{relax}^{\e_2}(t, \mu)$, then \beaa \|\f-J(t, \mu, \g^*;
  \cd, \g^*)\|_{\dbX_t} \le {\e_2}\q\mbox{for some}~ \g^*\in
  \cM_{relax}^{\e_2}(t,\mu).  \eeaa Introduce
  $\psi(\by, \nu) := J(T_0, \nu, \g^*; \by, \g^*)$ and recall
  \reff{rtower}. Then \beaa \|\f-J(T_0,\psi; t, \mu, \g^*; \cd,
  \g^*)\|_{\dbX_t} = \|\f(\bx)-J(t, \mu, \g^*; \bx, \g^*)\|_{\dbX_t}
  \le {\e_2}.  \eeaa Moreover, since
  $\g^*\in \cM_{relax}^{\e_2}(t,\mu)$, for any $\g\in \cA_{relax}$ and
  $\bx\in \dbX_t$, we have \beaa &&J(T_0,\psi; t, \mu, \g^*; \bx,
  \g^*) = J(t, \mu, \g^*; \bx, \g^*) \\ &&\le J(t, \mu, \g^*; \bx,
  \g\oplus_{T_0} \g^*) +{\e_2} = J(T_0,\psi; t, \mu, \g^*; \bx,
  \g)+{\e_2}.  \eeaa This implies that
  $\g^*\in \cM_{relax}^{\e_2}(T_0, \psi; t,\mu)$. We claim further
  that \bea
  \label{DPP2-claim} \psi(\cd, \mu^{\g^*}_{T_0\wedge\cd}) \in
  \dbV_{relax}^{C{\e_2}}(T_0, \mu^{\g^*}_{T_0\wedge \cd}), \eea for
  some constant $C\ge 1$. Then by \reff{DPP2} we see that
  $\f\in \tilde \dbV_{relax}^{C{\e_2}}(t,\mu)\subset \tilde
  \dbV_{relax}^\e(t,\mu)$ by setting ${\e_2}\le {\e\over C}$. Since
  $\e>0$ is arbitrary, we obtain $\f\in \tilde \dbV_{relax}(t,\mu)$.

  To see \reff{DPP2-claim}, recalling \reff{relaxJ}, for any
  $\g\in \cA_{relax}$ we have \beaa &&\dis \dbE^{\dbP^{\mu^{\g^*}; t,
      \bx, \g^*}}\Big[J(T_0, \mu^{\g^*}_{T_0\wedge \cd}, \g^*;
  X_{T_0\wedge \cd}, \g^*) \Big] - \dbE^{\dbP^{\mu^{\g^*}; t, \bx,
      \g^*}}\Big[J(T_0,
  \mu^{\g^*}_{T_0\wedge \cd}, \g^*; X_{T_0\wedge \cd}, \g) \Big]\\
  &&\dis = J(t, \mu, \g^*; \bx, \g^*) - J(t,\mu,\g^*; \bx,
  \g^*\oplus_{T_0} \g) \le \e_2.  \eeaa Then, by taking infimum over
  $\g\in \cA_{relax}$, it follows from the standard control theory
  that \beaa \dbE^{\dbP^{\mu^{\g^*}; \bx, \g^*}}\Big[J(T_0,
  \mu^{\g^*}_{T_0\wedge \cd}, \g^*; X_{T_0\wedge \cd}, \g^*) \Big] \le
  \dbE^{\dbP^{\mu^{\g^*};t, \bx, \g^*}}\Big[v(\mu^{\g^*}; T_0,
  X_{T_0\wedge \cd}) \Big]+{\e_2},\q\forall \bx\in \dbX_t.  \eeaa On
  the other hand, it is obvious that
  $v(\mu^{\g^*}; T_0, \tilde \bx) \le J(T_0, \mu^{\g^*}_{T_0\wedge
    \cd}, \g^*; \tilde\bx, \g^*)$ for all $\tilde\bx\in
  \dbX_{T_0}$. Moreover, since $q\ge c_q$, clearly
  $\dbP^{\mu^{\g^*}; t, \bx, \g^*}(X=_{T_0}\tilde \bx) \ge
  c_q^{T_0-t}$, for any $\tilde\bx\in \dbX^{t,\bx}_{T_0}$. Thus, \beaa
  0&\le& J(T_0, \mu^{\g^*}_{T_0\wedge\cd}, \g^*; \tilde \bx, \g^*) -
  v(\mu^{\g^*}; T_0, \tilde\bx)\\ &\le& C\dbE^{\dbP^{\mu^{\g^*};t,
      \bx, \g^*}}\Big[\big[J(T_0, \mu^{\g^*}_{T_0\wedge \cd}, \g^*;
  X_{T_0\wedge \cd}, \g^*) - v(\mu^{\g^*}; T_0, X_{T_0\wedge
    \cd})\big]\1_{\{X=_{T_0} \tilde\bx\}} \Big]\\
  &\le&C\dbE^{\dbP^{\mu^{\g^*};t,\bx, \g^*}}\Big[J(T_0,
  \mu^{\g^*}_{T_0\wedge \cd}, \g^*; X_{T_0\wedge \cd}, \g^*) -
  v(\mu^{\g^*}; T_0, X_{T_0\wedge \cd}) \Big] \le C{\e_2}, \eeaa where
  $C:= c_q^{t-T_0}$.  This implies that
  $\g^*\in \cM_{relax}^{C{\e_2}}(T_0,
  \mu^{\g^*}_{T_0\wedge\cd})$. Then \reff{DPP2-claim} follows directly
  from
  $\psi(\cd, \mu^{\g^*}_{T_0\wedge \cd}) = J(T_0,
  \mu^{\g^*}_{T_0\wedge\cd}, \g^*; \cd, \g^*)$, and hence
  $\f\in \tilde \dbV_{relax}(t,\mu)$.
\end{proof}

\begin{rem}
  \label{rem-statepath} 
  Consider the setting that $q, F, G$ are state
  dependent, as in Section \ref{sect-state}. There is a very subtle
  issue between state dependence and path dependence of the controls.
  
  (i) For a standard non-zero sum game problems where the players
    may have different cost functions $F_i, G_i$, if one uses state
    dependent controls, in general the set value does not satisfy
    DPP. See a counterexample in \cite{FRZ}. However, with path
    dependent controls the set value of the game satisfies the DPP.

 (ii) In Section \ref{sect-state}, since all players have the same
    cost function, as we saw the set value with state dependent
    controls satisfies DPP. If we consider path dependent controls
    $\a\in \cA_{path}$, the set value will also satisfy DPP. However,
    the set values in these two settings are in general not equal, see
    Example \ref{eg-statepath} in Appendix for a counterexample.

 (iii) For relaxed controls, again restricting to
    state dependent $q, F, G$, it turns out that state dependent and
    path dependent controls lead to the same set value, see Theorem
    \ref{thm-relaxstate} in Appendix. The main reason is that the
    convex combination of relaxed controls remains a relaxed control,
    while the controls $\a$ in Section \ref{sect-state} does not share
    this property.
  \end{rem}

\subsection{An alternative formulation of the relaxed mean field game}
\label{sect-global} In this subsection we provide an alternative
formulation for the MFG with relaxed controls. This new formulation is
motivated from the heterogenous controls for the $N$-player games, and
thus is crucial for the convergence result in the next section.

Let $\cA_{path}$ denote the set of adapted path dependent controls
$\a: \dbT\times \dbX\to \dbA$, and for each $t\in \dbT$, $\cA^t_{path}
= \big\{(\a(t,\cd), \cds, \a(T-1,\cd)): \a\in \cA_{path}\big\}$.
Denote $\Xi_t := \cP(\dbX_t \times \cA^t_{path})$, and for each $\L\in
\Xi_t$, define recursively: for $s\ge t$, $\bx\in \dbX_t$, and $\tilde
\bx\in \dbX^{t,\bx}$,
\begin{equation}
  \label{Lamda} \mu^\L_{t\wedge\cd}(\bx) := \L(\bx,
    \cA^t_{path}),\q \mu^\L_{s\wedge \cd}(\tilde\bx) :=
    \int_{\cA^t_{path}} \prod_{r=t}^{s-1} q(r, \tilde\bx, \mu^\L, \a(r,
    \tilde\bx); \tilde\bx_{r+1}) \L(\bx, d\a).  
\end{equation}
Here, noting that $\a\in \cA^t_{path}$ can be equivalently expressed
as
$\{\a(s, \tilde\bx): t \le s\le T-1, \tilde \bx \in \dbX^{t,\bx}_s\}$,
we are using the following interpretation on $d\a$: for any
$\f: \cA^t_{path}\to \dbR$, \bea
\label{da} \int_{\cA^t_{path}}\f(\a) d\a := \int_\dbA\cds\int_\dbA
\f\big(\{\a(s, \tilde\bx)\}\big) \prod_{s=t}^{T-1} \prod_{\tilde
\bx\in \dbX^{t,\bx}_s} d\a(s, \tilde \bx).  \eea

Next, for $\mu\in \cP_0(\dbX_t)$, denote $\Xi_t(\mu):= \{\L\in \Xi_t:
\mu^\L_{t\wedge \cd} = \mu\}$. Moreover, recall \reff{relaxJ},
\begin{equation}
  \label{LamdaJ} J(t,\L; \bx, \a) := J(\mu^\L; t, \bx, \a),\q v(t,\L;
  \bx) := v(\mu^\L; t, \bx),\q \bx\in \dbX_t, \a\in \cA^t_{path}.  
\end{equation}
To simplify the notations, we introduce: \bea
\label{Q} \left.\ba{c} \dis Q^t_s(\{\mu_\cd\}; \tilde\bx, \a) :=
\prod_{r=t}^{s-1} q(r, \tilde\bx, \mu, \a(r, \tilde\bx);
\tilde\bx_{r+1}).  \ea\right.  \eea In particular, $Q^t_t(\{\mu_\cd\};
\bx, \a) =1$. Then we have, for any $\tilde \bx\in \dbX^{t,\bx}$,
\begin{equation}
  \label{LQ} \mu^\L_s(\tilde\bx) :=
    \int_{\cA^t_{path}} Q^t_s(\mu^\L; \tilde\bx, \a) \L(\bx, d\a),\q
    \dbP^{\mu^\L;t, \bx, \a} (X=_s \tilde \bx)= Q^t_s(\mu^\L; \tilde \bx,
    \a).  
\end{equation}

\begin{defn}
\label{defn-LamdaMFE} For any $t\in \dbT$, $\mu\in \cP_0(\dbX_t)$, and
$\e>0$, we call $\L^* \in \Xi_t(\mu)$ a global $\e$-MFE at $(t,\mu)$,
denoted as $\L^*\in \cM^\e_{global}(t,\mu)$, if \bea
\label{LamdaMFE} \int_{\cA^t_{path}}[J(t, \L^*; \bx, \a) - v(t, \L^*;
\bx)] \L^*(\bx, d\a) \le \e,\q\forall \bx\in \dbX_t.  \eea
\end{defn} Note that the above $\a$ is global in time, so we call
$\L^*$ a global equilibrium. Moreover, since there are infinitely many
$\a\in \cA^t_{path}$, it is hard to require $ J(t, \L^*; \bx, \a) -
v(t, \L^*; \bx) \le \e$ for each $\a\in \cA^t_{path}$, we thus use the
above $\dbL^1$-type of optimality condition. For the $\bx$ part,
however, since there are only finitely many $\bx$ and each of them has
positive probability, we may require the optimality for each $\bx$.

The main result of this subsection is the following equivalence
result.

\begin{thm}
\label{thm-equivalence} For any $t\in \dbT$ and $\mu\in
\cP_0(\dbX_t)$, we have \bea
\label{equivalence} \left.\ba{c} \dis \dbV_{relax}(t, \mu) =
\dbV_{global}(t, \mu) := \bigcap_{\e>0} \dbV^\e_{global}(t,
\mu),\q\mbox{where}\\ \dis \dbV^\e_{global}(t, \mu):=\Big\{\f\in
\dbL^0(\dbX_t, \dbR): \exists \L^*\in \cM^\e_{global}(t,\mu)
~\mbox{s.t.}~\|\f - v(t, \L^*; \cd)\|_{\dbX_t}\le \e\Big\}.
\ea\right.  \eea
\end{thm}

We shall prove the mutual inclusion of the two sides
separately. First, given $(t, \L)$, we construct a relaxed control as
follows: for any $t\in \dbT$, $\bx\in \dbX_t$, and $s\ge t$, $\tilde
\bx\in \dbX^{t,\bx}_s$, \bea
\label{gL} \g^\L(s, \tilde \bx, da) := {1\over \mu^\L_{s\wedge
\cd}(\tilde \bx)} \int_{\cA^t_{path}} Q^t_s(\mu^\L; \tilde \bx; \a)
\d_{\a(s, \tilde \bx)}(da) \L(\bx, d\a).  \eea On the opposite
direction, given $t\in \dbT$, $\mu\in \cP_0(\dbX_t)$, $\g\in
\cA_{relax}$, recalling \reff{da} we construct \bea
\label{Lg} \L^\g(\bx, d\a) := \mu(\bx) \prod_{s=t}^{T-1} \prod_{\tilde
\bx\in \dbX^{t, \bx}_s} \g(s, \tilde \bx, d\a(s, \tilde \bx)),\q
\forall \bx\in \dbX_t, \a\in \cA^t_{path}.  \eea In particular, the
following calculation implies $\L^\g \in \Xi_t(\mu)$: \beaa \L^\g(\bx,
\cA^t_{path}) &=& \mu(\bx) \prod_{s=t}^{T-1} \prod_{\tilde \bx\in
\dbX^{t, \bx}_s} \g(s, \tilde \bx, \dbA) =\mu(\bx) \prod_{s=t}^{T-1}
\prod_{\tilde \bx\in \dbX^{t, \bx}_s} 1 = \mu(\bx).  \eeaa

\begin{lem}
\label{lem-equivalence} For any $t\in \dbT$, $\mu\in \cP_0(\dbX_t)$,
and $\L\in \Xi_t(\mu)$, $\g\in \cA_{relax}$, we have $\mu^{\g^\L} =
\mu^\L$ and $\mu^{\L^\g} = \mu^\g$. Moreover, \bea
\label{JgL} \dis J(t, \mu, \g^\L; \bx, \g^\L) = {1\over
\mu(\bx)}\int_{\cA^t_{path}} J(t, \L; \bx, \a) \L(\bx, d\a),\q\forall
\bx\in \dbX_t.  \eea
\end{lem}
\begin{proof}
  We first prove $\mu^{\g^\L}_{s\wedge \cd} = \mu^\L_{s\wedge \cd}$ by
  induction.  The case $s=t$ follows from the definitions. Assume it
  holds for all $r\le s$. For $s+1$ and
  $\tilde \bx\in \dbX^{t,\bx}_{s+1}$, by Fubini Theorem we have \beaa
  &&\dis {\mu^{\g^\L}_{(s+1)\wedge \cd}(\tilde \bx)\over
    \mu^{\g^\L}_{s\wedge \cd}(\tilde \bx_{s\wedge \cd})} = \int_\dbA
  q(s,
  \tilde\bx, \mu^{\g^\L}, a; \tilde\bx_{s+1}) \g^\L(s, \tilde\bx, da) \\
  &&\dis = \int_\dbA q(s, \tilde\bx, \mu^{\g^\L}, a;
  \tilde\bx_{s+1}){1\over \mu^\L_{s\wedge \cd}(\tilde \bx)}
  \int_{\cA^t_{path}} Q^t_s(\mu^\L; \tilde \bx; \a) \d_{\a(s, \tilde
    \bx)}(da) \L(\bx, d\a)\\ &&\dis = {1\over \mu^\L_{s\wedge
      \cd}(\tilde \bx)} \int_{\cA^t_{path}} q(s, \tilde\bx, \mu^\L,
  \a(s, \tilde \bx); \tilde\bx_{s+1})Q^t_s(\mu^\L; \tilde \bx; \a)
  \L(\bx, d\a)\\ &&\dis = {1\over \mu^\L_{s\wedge \cd}(\tilde \bx)}
  \int_{\cA^t_{path}} Q^t_{s+1}(\mu^\L; \tilde \bx; \a) \L(\bx, d\a) =
  { \mu^\L_{(s+1)\wedge \cd}(\tilde \bx)\over \mu^\L_{s\wedge
      \cd}(\tilde \bx)}.  \eeaa Then
  $\mu^{\g^\L}_{(s+1)\wedge \cd}= \mu^\L_{(s+1)\wedge \cd}$, and we
  complete the induction argument.

  We next prove $\mu^{\L^\g}_{s\wedge \cd} = \mu^\g_{s\wedge \cd}$ by
  induction. Again the case $s=t$ is obvious.  Assume it holds for all
  $r<s$. Now for $s$, recalling \reff{da} we have \beaa \dis
  \mu^{\L^\g}_{s\wedge \cd}(\tilde \bx) &=& \int_{\cA^t_{path}}
  \big[\prod_{r=t}^{s-1} q(r, \tilde\bx, \mu^\g, \a(r, \tilde\bx);
  \tilde\bx_{r+1})\big] \big[\mu(\bx) \prod_{r=t}^{T-1} \prod_{\bar
    \bx\in \dbX^{t, \bx}_r} \g(r, \bar \bx, d\a(r, \bar \bx))\big]\\
  \dis &=& \mu(\bx) \Big[ \prod_{r=t}^{s-1} \int_{\dbA} q(r,
  \tilde\bx, \mu^\g, \a(r, \tilde\bx); \tilde\bx_{r+1}) \g(r, \tilde
  \bx, d\a(r, \bar \bx))\Big]\times\\ &&\dis \Big[ \prod_{r=t}^{s-1}
  \prod_{\bar\bx \in \dbX^{t,\bx}_r\backslash \{\tilde \bx\}} \g(r,
  \bar \bx, \dbA)\Big]\times \Big[\prod_{r=s}^{T-1} \prod_{\bar \bx\in
    \dbX^{t, \bx}_r} \g(r, \bar \bx, \dbA)\Big]\\ &=& \mu(\bx)
  \prod_{r=t}^{s-1} \int_{\dbA} q(r, \tilde\bx, \mu^\g, a;
  \tilde\bx_{r+1})\g(r, \tilde \bx, da) = \mu^\g_{s\wedge \cd}(\tilde
  \bx).  \eeaa
  
  We finally prove \reff{JgL}.  For each $s\ge t$ and
  $\tilde\bx\in \dbX^{t,\bx}_s$, by Fubini Theorem again we have \beaa
  &\dis \int_\dbA F(s, \tilde \bx, \mu^\L, a) \g^\L(s, \tilde \bx, da)
  = \int_\dbA {F(s, \tilde \bx, \mu^\L, a)\over \mu^\L_{s\wedge
      \cd}(\tilde \bx)} \int_{\cA^t_{path}} Q^t_s(\mu^\L; \tilde \bx;
  \a) \d_{\a(s, \tilde \bx)}(da) \L(\bx, d\a)\\ &\dis = {1\over
    \mu^\L_{s\wedge \cd}(\tilde \bx)} \int_{\cA^t_{path}}F(s, \tilde
  \bx, \mu^\L, \a(s, \tilde \bx)) Q^t_s(\mu^\L; \tilde \bx; \a)
  \L(\bx, d\a) \eeaa By \reff{relaxJ} we have
  $\dbP^{\mu^\L;t, \bx, \g^\L} (X=_s \tilde \bx) = {\mu^\L_{s\wedge
      \cd}(\tilde\bx) \over \mu(\bx)}$. Thus \beaa &&\dis J(t, \mu,
  \g^\L; \bx, \g^\L) \\ &&\dis = {1\over \mu(\bx)}\Big[ \sum_{\tilde
    \bx \in \dbX^{t, \bx}} G(\tilde\bx, \mu^\L) \mu^\L_{T\wedge
    \cd}(\tilde \bx) + \sum_{s=t}^{T-1} \sum_{\tilde \bx \in \dbX^{t,
      \bx}_s} \mu^\L_{s\wedge\cd}(\tilde \bx)\int_{\dbA} F(s, \tilde
  \bx, \mu^\L, a) \g^\L(s, \tilde \bx, da)\Big]\\ &&\dis = {1\over
    \mu(\bx)} \int_{\cA^t_{path}}\Big[ \sum_{\tilde \bx \in \dbX^{t,
      \bx}} G(\tilde\bx, \mu^\L)Q^t_T(\mu^\L; \tilde \bx; \a) \\
  &&\dis\qq + \sum_{s=t}^{T-1} \sum_{\tilde \bx \in \dbX^{t, \bx}_s}
  F(s, \tilde \bx, \mu^\L, \a(s, \tilde \bx)) Q^t_s(\mu^\L; \tilde
  \bx; \a) \Big] \L(\bx, d\a).  \eeaa This implies \reff{JgL}
  immediately.
\end{proof}

\begin{rem}
\label{rem-Lginverse} We can actually show that $\g^{(\L^\g)} = \g$
for all $\g\in \cA_{relax}$, see Appendix. However, it is not clear
that we would have $\L^{(\g^\L)} = \L$ for all $\L\in \Xi_t(\mu)$.
\end{rem}

\no
{\bf Proof of Theorem \ref{thm-equivalence}.}
  Since $\mu\in \cP_0(\dbX_t)$ has full support, then
  $\dis c_\mu := \inf_{\bx\in \dbX_t} \mu(\bx)>0$.

  (i) We first prove $\dbV_{global}(t, \mu) \subset \dbV_{relax}(t,
  \mu)$. Fix $\f\in \dbV_{global}(t, \mu)$ and $\e>0$. Let $\e_1>0$ be a
  small number which will be specified later. Since $\f\in
  \dbV^{\e_1}_{global}(t, \mu)$, there exists $\L^* \in
  \cM^{\e_1}_{global}(t,\mu)$ such that $\|\f - v(t, \L^*;
  \cd)\|_{\dbX_t}\le \e_1$. Set $\g^* := \g^{\L^*}$.  For any $\bx\in
  \dbX_t$, since $\mu^{\g^*} = \mu^{\L^*}$, by \reff{relaxJ},
  \reff{LamdaJ} we have $v(\mu^{\g^*};t, \bx, \g^*) = v(t, \L^*; \bx)$,
  and, by \reff{JgL}, \reff{LamdaMFE}, \beaa J(t, \mu, \g^*; \bx, \g^*)
  - v(t, \L^*; \bx) = {1\over \mu(\bx)}\int_{\cA^t_{path}}\!\!\!\! [J(t,
  \L^*; \bx, \a) - v(t, \L^*; \bx)] \L^*(\bx, d\a)\le {\e_1\over c_\mu}
  \le \e, \eeaa provided $\e_1>0$ is small enough. This implies $\g^*
  \in \cM^\e_{relax}(t, \mu)$.

  Moreover, it is clear now that, for any $\bx\in \dbX_t$ and for a
  possibly smaller $\e_1$, \beaa \big|\f(\bx) - J(t, \mu, \g^*; \bx,
  \g^*)\big| \le \e_1+ \big|v(t, \L^*; \bx) - J(t, \mu, \g^*; \bx,
  \g^*)\big| \le \e_1 + {\e_1\over c_\mu} \le \e, \eeaa Then $\f\in
  \dbV^\e_{relax}(t,\mu)$, and since $\e>0$ is arbitrary, we obtain
  $\f\in \dbV_{relax}(t,\mu)$.

  (ii) We next prove $\dbV_{relax}(t, \mu) \subset \dbV_{global}(t,
  \mu)$. Fix $\f\in \dbV_{relax}(t, \mu)$, $\e>0$, and set $\e_2 :=
  {\e\over 2}$.  Since $\f\in \dbV^{\e_2}_{relax}(t, \mu)$, there exists
  $\g^* \in \cM^{\e_2}_{relax}(t,\mu)$ such that $\|\f - J(t, \mu, \g^*;
  \cd, \g^*)\|_{\dbX_t}\le \e_2$. Set $\L^* := \L^{\g^*}$, then
  $\mu^{\L^*} = \mu^{\g^*}$. Since $\g^* \in \cM^{\e_2}_{relax}(t,\mu)$,
  we have \beaa |\f(\bx) - v(t,\L^*; \bx)| = |\f(\bx) - v(\mu^{\g^*};t,
  \bx)| \le 2\e_2 \le \e,\q\forall \bx\in \dbX_t.  \eeaa Moreover, note
  that, by \reff{JgL} again, \bea
  \label{gLMFE} \left.\ba{c} \dis \int_{\cA^t_{path}} [J(t, \L^*; \bx,
    \a) - v(t, \L^*; \bx)] \L^*(\bx, d\a) \\ \dis =\mu(\bx)[ J(t, \mu,
    \g^*; \bx, \g^*) - v(t, \L^*; \bx) ] \le \mu(\bx) \e_2 \le \e_2 \le
    \e.  \ea\right.  \eea This implies $\f\in \dbV^\e_{global}(t, \mu)$,
  and hence by the arbitrariness of $\e$, $\f\in \dbV_{global}(t, \mu)$.
\qed

 \section{The $N$-player game with heterogeneous equilibria}
 \label{sect-N2} \setcounter{equation}{0} In this section we drop the
requirement $\a^1=\cds=\a^N$ for the $N$-player game, and show that
the corresponding set value converges to $\dbV_{relax}$, which in
general is strictly larger than $\dbV_{state}$. We note that we shall
still use the pure strategies, rather than mixed strategies, for the
$N$-player game. Moreover, since we used path dependent controls in
Section \ref{sect-relax}, we shall also use path dependent controls
here.

\subsection{The $N$-player game} Let $\O^N$ and $\vec X$ be as in
Section \ref{sect-N1}, and denote
\begin{equation}
  \label{muN2} \mu^N_{t\wedge \cd} := \mu^N_{t, \vec X_{t\wedge
      \cd}},\hspace{0.5em}
  \mbox{where}\q \mu^N_{t, \vec \bx}:= {1\over N} \sum_{i=1}^N
  \d_{\bx^i} \in \cP(\dbX_t),~ \vec \bx =(\bx^1,\cds, \bx^N)\in
  \dbX^N_t.  
\end{equation}
Similarly to \reff{SN0}, for the convenience of the presentation we
introduce \bea
\label{XN0} \dbX^N_{0,t}:= \Big\{\vec \bx\in \dbX^N_t: \supp(\mu^N_{t,
\vec \bx}) = \dbX_t\Big\},\q \cP_N(\dbX_t) := \Big\{\mu^N_{t, \vec
\bx}: \vec \bx\in \dbX^N_{0,t}\Big\}.  \eea We shall consider path
dependent symmetric controls: $\cA^{t,\infty}_{path}:= \bigcup_{L\ge
0} \cA^{t,L}_{path}$, where \beaa &\dis \cA^{t,L}_{path}:=\Big\{ \a:
\{t,\cds, T-1\} \times \dbX \times \cP(\dbX)\to \dbA \Big|~ \mbox{$\a$
is adapted and}\\ &\dis \mbox{uniformly Lipschitz continuous in $\mu$
(under $W_1$) with Lipschitz constant $L$}\Big\}.  \eeaa Given $t\in
\dbT$, $\vec \bx\in \dbX^N_{0,t}$, and $\vec \a =(\a^1,\cds,\a^N)\in
(\cA^{t,\infty}_{path})^N$, introduce, for $s\ge t$, \bea
\label{PtxaPath} \left.\ba{c} \dis \dbP^{t,{\vec \bx}, {\vec \a}}(\vec
X =_t \vec \bx) = 1,~ \dbP^{t,{\vec \bx}, {\vec \a}}(\vec X_{s+1}=\vec
x'' | \vec X =_s \vec \bx') = \prod_{i=1}^N q(s, \bx'^i, \mu^N,
\a^i(s, \bx'^i, \mu^N); x''_i),\\ \dis J_i(t, \vec \bx, \vec \a) :=
\dbE^{\dbP^{t,{\vec \bx}, {\vec \a}}}\Big[G(X^i, \mu^N) +
\sum_{s=t}^{T-1} F(s, X^i, \mu^N, \a^i(s, X^i, \mu^N))\Big];\\ \dis
v^{N,L}_i(t, \vec\bx, \vec\a) := \inf_{\tilde\a\in \cA^{t,L}_{path}}
J_i(t, \vec \bx, \vec \a^{-i}, \tilde \a),\q ~ i=1,\cds,N.  \ea\right.
\eea Here $(\vec \a^{-i}, \tilde \a)$ is the vector obtained by
replacing $\a^i$ in $\vec \a$ with $\tilde \a$.

\begin{defn}
\label{defn-NEPath} For any $\e>0, L\ge 0$, we say $\vec\a \in
(\cA^{t,L}_{path})^N$ is an $(\e, L)$-equilibrium of the $N$-player
game at $(t, \vec \bx)$, denoted as $\vec\a\in \cM^{N,\e,
L}_{hetero}(t,\vec\bx)$, if: \bea
\label{NEPath} {1\over N} \sum_{i=1}^N \big[J_i(t, \vec \bx, \vec \a)
-v^{N,L}_i(t, \vec\bx, \vec\a)\big] \le \e.  \eea
\end{defn} Here, since there are $N$ players and we will send $N\to
\infty$, similar to \reff{LamdaMFE} we do not require the optimality
for each player.  In fact, by \reff{NEPath} one can easily show that
\bea
\label{Carmona} {1\over N}\Big|\big\{i=1,\cds, N: J_i(t, \vec \bx,
\vec \a) -v^{N,L}_i(t, \vec\bx, \vec\a) \ge \sqrt{\e}\big\}\Big|\le
\sqrt{\e}.  \eea This is exactly the $(\sqrt{\e},
\sqrt{\e})$-equilibrium in \cite{Carmona}.

We then define the set value of the $N$-player game with heterogeneous
equilibria:
\bea
\label{VNtmuG} \left.\ba{c} \dis\dbV^N_{hetero}(t,\vec\bx) :=
\bigcap_{\e>0} \dbV^{N, \e}_{hetero}(t,\vec\bx) := \bigcap_{\e>0}
\bigcup_{L\ge 0} \dbV^{N, \e, L}_{hetero}(t,\vec\bx),\\ \dis
\mbox{where} \q \dbV^{N, \e, L}_{hetero}(t,\vec\bx) := \Big\{\f\in
\dbL^0(\dbX_t; \dbR): \exists \vec\a\in \cM^{N, \e,
L}_{hetero}(t,\vec\bx)~\mbox{such that}\\ \dis
\max_{\bx\in\dbX_t}\min_{\{i: ~\bx^i=\bx\}} \big|\f(\bx) -
v^{N,L}_i(t, \vec\bx, \vec\a)\big| \le \e\Big\}.  \ea\right.  \eea

\begin{rem}
  (i) An alternative definition of
  $\dbV^{N, \e, L}_{hetero}(t,\vec\bx)$ is to require $\f$ satisfying
  \bea
  \label{VNtmuGalternative}
  \max_{i=1,\cds, N} \big|\f(\bx^i) - v^{N,L}_i(t, \vec\bx,
  \vec\a)\big| = \max_{\bx\in\dbX_t}\max_{\{i:~ \bx^i=\bx\}}
  \big|\f(\bx) - v^{N,L}_i(t, \vec\bx, \vec\a)\big| \le \e.  \eea
  Indeed, the convergence result Theorem \ref{thm-FinNonSymConv} below
  remains true if we use \reff{VNtmuGalternative}. However, in general
  it is possible that $\bx^i = \bx^j$ but
  $v^{N,L}_i(t, \vec\bx, \vec\a) \neq v^{N,L}_j(t, \vec\bx,
  \vec\a)$. Then, by fixing $N$ and sending $\e\to 0$, under
  \reff{VNtmuGalternative} we would have
  $\dbV^N_{hetero}(t,\vec\bx) := \bigcap_{\e>0} \dbV^{N,
    \e}_{hetero}(t,\vec\bx) = \emptyset$.

  (ii) In the homogeneous case,
  $v^{N,L}_i(t, \vec\bx, \vec\a) = v^{N,L}_j(t, \vec\bx, \vec\a)$
  whenever $\bx^i=\bx^j$, so we don't have this issue in \reff{VtmuN}.

  (iii) Note that $\mu_{t,\vec\bx}^N=\mu_{t,\vec\bx'}^N$ if and only
  if $\vec\bx$ is a permutation of $\vec\bx'$, and one can easily
  verify that
  $v^{N,L}_i(t, \vec\bx, \vec\a) = v^{N,L}_{\pi(i)}(t,
  (\bx_{\pi(1)},\cds, \bx_{\pi(N)}), (\a_{\pi(1)},\cds, \a_{\pi(N)}))$
  for any permutation $\pi$ on $\{1,\cds, N\}$, . Then, similar to the
  homogenous case, $\dbV^{N, \e, L}_{hetero}(t,\vec\bx)$ is invariant
  in $\mu_{t,\vec\bx}^N$ and we will denote is as
  $\dbV^{N, \e, L}_{hetero}(t,\mu_{t,\vec\bx}^N)$.
\end{rem}

The following convergence result of the set value is in the same
spirit of Theorem \ref{thm-FinSymConv}.
\begin{thm}
\label{thm-FinNonSymConv} Let Assumption \ref{assum-reg} hold and
$\mu^N_{t,\vec \bx}\in \cP_N(\dbX_t)\to \mu\in \cP_0(\dbX_t)$ under
$W_1$. Then \bea
 \label{FinNonSymConv} \bigcap_{\e>0}\bigcup_{L\ge 0}\limsup_{N\to
\infty} \dbV^{N, \e, L}_{hetero}(t, \mu^N_{t,\vec \bx}) \subset
\dbV_{relax}(t,\mu) \subset \bigcap_{\e>0}
\liminf_{N\to\infty}\dbV^{N, \e, 0}_{hetero}(t,\mu^N_{t,\vec \bx}).
\eea In particular, since $\dis \liminf_{N\to\infty}\dbV^{N, \e,
0}_{hetero}(t,\mu^N_{t,\vec \bx}) \subset \bigcup_{L\ge
0}\limsup_{N\to \infty} \dbV^{N, \e, L}_{hetero}(t,\mu^N_{t,\vec
\bx})$, actually equalities hold.
\end{thm}

Unlike Theorem \ref{thm-FinSymConv}, here the $N$-player game and the
MFG take different types of controls $\vec \a$ and $\g$,
respectively. The key for the convergence is the global formulation in
Subsection \ref{sect-global} for MFG. Indeed, given $t\in \dbT$,
$\vec\bx\in \dbX^N_{0,t}$, and $\vec\a\in (\cA^{t, L}_{path})^N$, the
$N$-player game is naturally related to the following $\L^N\in
\cP(\dbX_t \times \cA^{t,L}_{path})$:
\begin{equation}
  \label{LamdaN} \L^N(\bx, d\a) := {1\over N} \sum_{i\in I(\bx)}
  \d_{\a_i}(d\a),~ \mbox{where}~ I(\bx) := \big\{i=1,\cds, N:
  \bx^i=\bx\big\},~ \bx\in \dbX_t.  
\end{equation}
By the symmetry of the problem, there exists a function $J^N$,
independent of $i$, such that \bea
\label{JiN} J_i(t, \vec \bx, \vec\a) = J^N(\L^N; t, \bx^i, \a_i),\q
i=1,\cds, N.  \eea We shall use this and Theorem \ref{thm-equivalence}
to prove Theorem \ref{thm-FinNonSymConv} in the rest of this
section. We also make the following obvious observation: \bea
\label{I} \L^N(\bx, \cA^t_{path}) = {|I(\bx)|\over N} =
\mu^N_{t,\vec\bx}(\bx),\q\forall \bx\in \dbX_t.  \eea

\begin{rem}
  \label{rem-fullweak}
  (i) In this section we are using symmetric controls and we obtain
  the convergence in Theorem \ref{thm-FinNonSymConv}. If we use full
  information controls $\a_i(t, \vec X)$, as observed in
  \cite{Lacker2} in terms of the equilibrium measure, one may expect
  the limit set value will be strictly larger than $\dbV_{relax}$. It
  will be interesting to find an appropriate notion of MFE so that the
  corresponding MFG set value will be equal to the above limit, in the
  sense of Theorem \ref{thm-FinNonSymConv}.

  (ii) While the convergence in Theorem \ref{thm-FinNonSymConv} is
  about set values, the proofs in the rest of this section confirm the
  convergence of the approximate equilibria as well, exactly in the
  same manner as in Remark \ref{rem-conv}.
\end{rem}

\subsection{From $N$-player games to mean field games} In this
subsection we prove the left inclusion in \reff{FinNonSymConv}.
Notice that the $\L^N$ in \reff{LamdaN} is defined on $\cA^{t,
L}_{path}$, rather than $\cA^t_{path} = \cA^{t, 0}_{path}$. For this
purpose, recall \reff{Q} and introduce \bea
\label{nuN} \left.\ba{c} \dis \nu^N_{t\wedge \cd}(\bx):= \mu^N_{t,
\vec\bx}(\bx),~ \nu^N_{s\wedge \cd}(\tilde\bx) := {1\over N}
\sum_{i\in I(\bx)} Q^t_s(\nu^N; \tilde \bx, \a_i(\cd,\cd, \nu^N)),~
\bx\in \dbX_t, \tilde \bx\in \dbX^{t, \bx}_s, s\ge t;\\ \dis \bar
\L^N(\bx, d\a) := {\mu(\bx)\over |I(\bx)|} \sum_{i\in I(\bx)} \d_{\bar
\a_i}(d\a),\q\mbox{where}\q \bar \a_i(s, \tilde \bx) := \a_i(s, \tilde
\bx, \nu^N).  \ea\right.  \eea Then it is obvious that $\bar \a_i
\in\cA^t_{path}$ and $\bar \L^N \in \Xi_t(\mu)$.  Moreover, when
$\mu=\mu^N_{t, \vec\bx}$, by \reff{LQ} and \reff{I} it is
straightforward to verify by induction that $\mu^{\bar \L^N} = \nu^N$.

\begin{thm}
  \label{thm-NonSymMeasureConv1} Let Assumption \ref{assum-reg} (ii)
  hold. Then, for any $L\ge 0$, there exists a constant $C_L$,
  depending only on $T, d, L_q$, and $L$ such that, for any
  $t\in \dbT$, $\vec \bx\in \dbX_{0,t}^N$, $\mu\in \cP_0(\dbX_t)$,
  $\vec\a\in (\cA_{path}^{t,L})^N, \tilde \a\in \cA^{t,L}_{path}$, and
  for the $\nu^N, \bar\L^N$ defined in \reff{nuN}, we have
  \begin{equation}
    \begin{aligned}
      \label{NonSymMeasureEst1}
      &\max_{1\le i\le N}\max_{t\le s\le
        T}\dbE^{\dbP^{t, \vec \bx, (\vec\a^{-i}, \tilde
          \a)}}\big[\cW_1(\mu^N_{s\wedge\cd}, \mu^{\bar\L^N}_{s\wedge\cd})\big]
      \le C_L\th_N, ~~\th_N:= W_1(\mu^N_{t,\vec\bx}, \mu) +
      {1\over \sqrt{N}}.  
    \end{aligned}
  \end{equation}
\end{thm}
\begin{proof}
  Fix $i$ and denote $\tilde \a_j := \a_j$ for $j\neq i$, and
  $\tilde \a_i:= \tilde \a_i$.  We first show that \bea
  \label{muNnuN} \k_s := \dbE^{\dbP^N}\big[\cW_1(\mu^N_{s\wedge\cd},
  \nu^N_{s\wedge\cd})\big]\le {C_L\over \sqrt{N}},\q\mbox{where}\q
  \dbP^N:= \dbP^{t, \vec x, (\vec\a^{-i}, \tilde \a)}.  \eea Indeed,
  for $s\ge t$, by the conditional independence of
  $\{X^j_{s+1}\}_{1\le j\le N}$ under $\dbP^N$, conditional on
  $\cF_s$, it follows from the same arguments as in \reff{muNaest1}
  that \beaa \k_{s+1} &=& \dbE^{\dbP^{N}}\Big[
  \dbE^{\dbP^{N}}_{\cF_s}\big[\cW_1(\mu^N_{(s+1)\wedge \cd},
  \nu^N_{(s+1)\wedge \cd})\big]\Big] \\ &\le& {C\over \sqrt{N}} +
  C\sum_{\bx\in \dbX_{s+1}}\dbE^{\dbP^N}\Big[ \Big|{1\over
    N}\sum_{j=1}^N \dbP^{N}(X^j=_{s+1}\bx |\cF_s) -\nu^N_{(s+1)\wedge
    \cd}(\bx)\Big| \Big].  \eeaa Note that, \beaa &&\dis\Big|{1\over
    N}\sum_{j=1}^N \dbP^{N}(X^j=_{s+1}\bx |\cF_s) -{1\over
    N}\sum_{j=1}^N\1_{\{X^j=_s \bx\}} q(s, \bx, \nu^N, \a_j(s, \bx,
  \nu^N); \bx_{s+1}) \Big|\\ &&\dis =\Big| {1\over
    N}\sum_{j=1}^N\1_{\{X^j=_s \bx\}} \big[q(s, \bx, \mu^N,
  \tilde\a_j(s, \bx, \mu^N); \bx_{s+1}) - q(s, \bx, \nu^N, \a_j(s,
  \bx, \nu^N); \bx_{s+1})\big] \Big|\\ &&\dis \le
  C_LW_1(\mu^N_{s\wedge \cd}, \nu^N_{s\wedge\cd}) + {1\over N} = C_L
  \k_s +{1\over N}, \eeaa where in the last inequality, the first term
  is due to the sum over all $j\neq i$. Then \beaa &&\dis \k_{s+1} \le
  C_L \k_s + {C\over \sqrt{N}} + \dbE^{\dbP^N}\Big[ \sum_{\bx\in
    \dbX_{s+1}} \Big|{1\over N}\sum_{j=1}^N\1_{\{X^j=_s \bx\}} q(s,
  \bx, \nu^N, \a_j(s, \bx, \nu^N); \bx_{s+1}) \\ &&\dis\qq - {1\over
    N} \sum_{j\in I(\bx_{t\wedge\cd})} Q^t_s(\nu^N; \bx, \bar
  \a_j)q(s, \bx, \nu^N, \a_j(s, \bx, \nu^N); \bx_{s+1})\Big|\Big]\\
  &&\dis= C_L \k_s + {C\over \sqrt{N}} + \dbE^{\dbP^N}\Big[
  \sum_{\bx\in \dbX_{s}} \Big|{1\over N}\sum_{j=1}^N\1_{\{X^j=_s
    \bx\}} - {1\over N} \sum_{j\in I(\bx_{t\wedge\cd})} Q^t_s(\nu^N;
  \bx, \bar \a_j)\Big|\Big]\\ &&= C_L \k_s + {C\over \sqrt{N}} +
  \dbE^{\dbP^N}\Big[ \sum_{\bx\in \dbX_{s}} \big|\mu^N_{s\wedge
    \cd}(\bx)- \nu^N_{s\wedge \cd}(\bx)\big|\Big]\le C_L \k_s+{C\over
    \sqrt{N}}.  \eeaa It is obvious that $\k_t =0$.  Then by induction
  we obtain \reff{muNnuN}.

  Next, denote
  $\bar\k_s := W_1(\nu^N_{s\wedge \cd}, \mu^{\bar\L^N}_{s\wedge
    \cd})$. For $s\ge t$, by \reff{nuN}, \reff{LQ}, and \reff{Q}, we
  have \beaa &&\dis \bar \k_{s+1} = \sum_{\bx\in \dbX_t} \sum_{\tilde
    \bx\in \dbX^{t,\bx}_{s+1}}\big| \nu^N_{(s+1)\wedge \cd}(\tilde
  \bx) - \mu^{\bar\L^N}_{(s+1)\wedge \cd}(\tilde \bx)\big|\\ &&\dis =
  \sum_{\bx\in \dbX_t} \sum_{\tilde \bx\in \dbX^{t,\bx}_{s+1}}\big|
  {1\over N} \sum_{j\in I(\bx)} Q^t_{s+1}(\nu^N; \tilde \bx, \bar\a_j)
  -{\mu(\bx)\over |I(\bx)|} \sum_{j\in I(\bx)}
  Q^t_{s+1}(\mu^{\bar\L^N}; \tilde \bx, \bar\a_j)\big|\\ &&\dis =
  \sum_{\bx\in \dbX_t} \sum_{\tilde \bx\in \dbX^{t,\bx}_{s+1}}\Big[
  {1\over N} \sum_{j\in I(\bx)} \big|Q^t_{s+1}(\nu^N; \tilde \bx,
  \bar\a_j) - Q^t_{s+1}(\mu^{\bar\L^N}; \tilde \bx, \bar\a_j)\big|\\
  &&\dis \qq +\Big|{1\over N} - {\mu(\bx)\over |I(\bx)|}\Big|
  \sum_{j\in I(\bx)} Q^t_{s+1}(\mu^{\bar\L^N}; \tilde \bx,
  \bar\a_j)\Big]\\ &&\dis \le C \sum_{\bx\in \dbX_t} \sum_{\tilde
    \bx\in \dbX^{t,\bx}_{s+1}}\Big[ {1\over N} \sum_{j\in I(\bx)}
  \sum_{r=t}^s W_1(\nu^N_{r\wedge \cd}, \mu^{\bar\L^N}_{r\wedge \cd})
  +\Big|{1\over N} - {\mu(\bx)\over |I(\bx)|}\Big| |I(\bx)|\Big]\\
  &&\dis \le C \sum_{r=t}^s \bar\k_r +C\sum_{\bx\in
    \dbX_t}\big|\mu^N_{t, \vec\bx}(\bx) - \mu(\bx)\big| \le C
  \sum_{r=t}^s \bar\k_r.  \eeaa Obviously
  $\bar k_t = W_1(\mu^N_{t,\vec\bx}, \mu)$. Then by induction we have
  $\dis\sup_{t\le s\le T}\bar\k_s \le CW_1(\mu^N_{t,\vec\bx},
  \mu)$. This, together with \reff{muNnuN}, implies
  \reff{NonSymMeasureEst1} immediately.
\end{proof}

\begin{thm}
\label{thm-NonSymCostEst1} For the setting in Theorem
\ref{thm-NonSymMeasureConv1} and assuming further Assumption
\ref{assum-reg} (iii), there exists a modulus of continuity function
$\rho_L$, depending on $T, d, L_q$, $C_0$, $\rho$, $L$, s.t.
\begin{equation}
  \label{FinNonSymCostEst1} \Big|J_i(t, \vec x, (\vec\a^{-i},\tilde \a))
  - J(t,\bar\L^N; \bx^i, \tilde \a(\cd,\nu^N) )\Big| + \big|v^{N,L}_i(t,
  \vec\bx, \vec\a) - v(\mu^{\bar\L^N}; t,\bx^i)\big|\le \rho_L(\th_N).  
\end{equation}
Moreover, assume $\vec \a\in \cM^{N,\e_1, L}_{hetero}(t,\vec\bx)$ for
some $\e_1>0$, then
\begin{equation}
  \label{FinNonSymMFE1} \int_{\cA^t_{path}}[J(t, \bar\L^N; \bx, \a) -
  v(t, \bar\L^N; \bx)] \bar\L^N(\bx, d\a) \le \e_1+
  2\rho_L(\th_N),\q\forall \bx\in \dbX_t.  
\end{equation}
In particular, if $\e_1+ 2\rho_L(\th_N) \le \e$, then
$\bar \L^N \in \cM^\e_{global}(t,\mu)$.
\end{thm}
\proof  First, given Theorem \ref{thm-NonSymMeasureConv1},
  \reff{FinNonSymCostEst1} follows from the arguments in Theorem
  \ref{thm-FinSymCostConv}. Then, for
  $\vec \a\in \cM^{N,\e_1, L}_{hetero}(t,\vec\bx)$ and
  $\bx\in \dbX_t$, by \reff{NEPath} we have \beaa &&\dis
  \int_{\cA^t_{path}}[J(t, \bar\L^N; \bx, \a) - v(t, \bar\L^N; \bx)]
  \bar\L^N(\bx, d\a) = {1\over N} \sum_{i\in I(\bx)} \big[J(t,
  \bar\L^N; \bx, \bar\a_i) - v(t, \bar\L^N; \bx)\big]\\ &&\dis \le
  {1\over N} \sum_{i\in I(\bx)} \Big[\big|J(t, \bar\L^N; \bx^i,
  \bar\a_i) - J_i(t, \vec \bx, \vec\a)\big| + \big[J_i(t, \vec \bx,
  \vec\a) - v^{N,L}_i(t, \vec\bx, \vec\a)\big] \\ &&\dis\qq\qq +
  \big|v^{N,L}_i(t, \vec\bx, \vec\a) - v( \mu^{\bar\L^N}; t,
  \bx^i)\big|\Big]\\ &&\dis \le \rho_L(\th_N) + \e_1 + \rho_L(\th_N) =
  \e_1 + 2 \rho_L(\th_N).
  \eeaa 
  
  \vspace{-9.5mm}
  \qed

\no{\bf Proof of Theorem \ref{thm-FinNonSymConv}: the left
  inclusion}.
  We first fix an arbitrary function
  $\f\in \bigcap_{\e>0}\bigcup_{L\ge 0}\limsup_{N\to \infty} \dbV^{N,
    \e, L}_{hetero}(t, \mu_{t,\vec \bx}^N)$, $\e>0$, and set
  $\e_1:= {\e\over 2}$. Then there exists $L_{\e}\ge 0$ and and a
  sequence $N_k\to \infty$ (possibly depending on $\e$) such that
  $\f\in \dbV^{N_k, \e_1, L_{\e_1}}_{hetero}(t,\mu_{t,\vec
    \bx}^{N_k})$, for all $k\ge 1$. Now choose $k$ large enough so
  that $2\rho_{L_\e}(\th_{N_k})\le \e_1$. By \reff{VNtmuG} there
  exists $\vec\a\in \cM^{N_k, \e_1, L_\e}_{hetero}(t,\vec\bx)$ such
  that
  $\max_{\bx\in\dbX_t}\min_{i\in I(\bx)} |\f(\bx) -
  v_i^{N,L}(t,\vec\bx,\vec\a)| \leq \e_1$. By Theorem
  \ref{thm-NonSymCostEst1} we see that
  $\bar \L^{N_k} \in \cM^\e_{global}(t,\mu)$ and, by
  \reff{FinNonSymCostEst1}, \beaa \|\f - v( \mu^{\bar\L^N}; t,
  \cdot)\|_{\dbX_t} \!\!&\le&\!\!  \max_{\bx\in\dbX_t}\min_{i\in
    I(\bx)}\Big[\big|\f(\bx) - v^{N,L}_i(t, \vec\bx, \vec\a)\big|+
  \big|v^{N,L}_i(t, \vec\bx, \vec\a)- v(\mu^{\bar\L^N};t, \bx)\big|
  \Big]\\ \!\!&\le&\!\! \e_1 + \rho_{L_\e}(\th_N) \le \e.  
  \eeaa 
  Then $\f\in \dbV^\e_{global}(t,\mu)$.
Since $\e>0$ is arbitrary, by Theorem \ref{thm-equivalence} we
  get $\f\in \dbV_{relax}(t,\mu)$.
\qed

\subsection{From mean field games to $N$-player games}

We now turn to the right inclusion in \reff{FinNonSymConv}. Fix
$t\in \dbT$, $\vec \bx\in \dbX^N_{0,t}$, $\mu\in \cP_0(\dbX_t)$, and
$\g\in \cA_{relax}$. Our goal is to construct a desired
$\vec \a\in (\cA^{t,0}_{path})^N$. However, since $\vec \a$, or
equivalently the corresponding $\L^N$, is discrete, we need to
discretize $\g$ first. We note that it is slightly easier to
discretize $\g$ than a general $\L\in \Xi_t(\mu)$.

First, given $\e>0$, there exists a partition
$\dbA = \cup_{k=0}^{n_\e} A_k$ with $n_\e$ depending on $\e$ (and
$\g$) such that, for some arbitrarily fixed $a_k\in A_k$,
$k=0,\cds, n_\e$,
\begin{equation}
  \label{Ae}\g(s, \bx, A_0) \le\e, \forall s\in
    \dbT_t, \bx\in \dbX_s,
    \q\mbox{and}\q |a-a_k|\le \e, \forall a\in A_k,~
    k=1,\cds, n_\e.  
\end{equation}
Denote by $\cA^{t, \e}_{path}$ the subset of $\a\in \cA^{t,0}_{path}$
taking values in $\dbA_\e := \{a_k: k=0,\cds, n_\e\}$.  Define \bea
\label{ge} \g^\e(s, \bx, da) := \sum_{k=0}^{n_\e} \g(s, \bx, A_k)
\d_{a_k}(da).  \eea Recall \reff{Lg}, we see that
$\supp(\L^{\g^\e}(\bx, d\a)) = \cA^{t,\e}_{path} \subset
\cA^{t,0}_{path}$ for all $\bx\in \dbX_t$.

Next, recall \reff{I} that $N\mu^N_{t,\vec \bx}(\bx) = |I(\bx)|$ is a
positive integer for all $\bx\in \dbX_t$. Let $\L^\e_{t,\vec \bx}\in
\cP(\dbX_t\times\cA^{t,\e}_{path})$ be a modification of $\L^{\g^\e}$
such that, \bea
\label{LdN} \left.\ba{c} \dis \L^\e_{t,\vec \bx}(\bx,
\cA^{t,\e}_{path}) = \mu^N_{t, \vec\bx} (\bx)~ \mbox{and
$N\L^\e_{t,\vec \bx}(\bx, \a)$ is an integer};\\ \dis |\L^\e_{t,\vec
\bx}(\bx, \a) - \L^{\g^\e}(\bx, \a)|\le {1\over N} + |\mu^N_{t,
\vec\bx} (\bx) - \mu(\bx)|; \ea\right. ~ \forall (\bx, \a)\in
\dbX_t\times\cA^{t,\e}_{path}.  \eea Note that, since $
\cA^{t,\e}_{path}$ is finite, such a construction is easy.

We now construct $\vec\a\in (\cA^{t,\e}_{path})^N$, which relies on
$\g^\e$ and hence on $\e$. Note that \beaa \sum_{\a\in
\cA^{t,\e}_{path}} [N\L^\e_{t,\vec \bx}(\bx, \a)] = N \L^\e_{t,\vec
\bx}(\bx, \cA^{t,\e}_{path}) = N\mu^N_{t, \vec\bx} (\bx) = |I(\bx)|,
\eeaa and each $N\L^\e_{t,\vec \bx}(\bx, \a)$ is an integer.  Let
$I(\bx) = \cup_{\a\in \cA^{t,\e}_{path}} I(\bx, \a)$ be a partition of
$I(\bx)$ such that $|I(\bx, \a)| = N\L^\e_{t,\vec \bx}(\bx, \a)$. We
then set \bea
 \label{veca} \a_i := \a,\q i\in I(\bx, \a),\q (\bx,\a)\in
\dbX_t\times\cA^{t,\e}_{path}.  \eea Let $\L^N$ be the one defined by
\reff{LamdaN} corresponding to this $\vec\a$. It is clear that $\L^N =
\L^\e_{t, \vec\bx}$.
 \begin{thm}
\label{thm-NonSymConv2} (i) Let Assumption \ref{assum-reg} (ii)
hold. Then there exists a constant $C$, depending only on $T, d, L_q$,
such that, for any $t\in \dbT$, $\vec \bx\in \dbX_{0,t}^N$, $\mu\in
\cP_0(\dbX_t)$, $\g\in \cA_{relax}$, $\e>0$, and for the $\vec\a\in
(\cA_{path}^{t,\e})^N$ constructed above, we have, for the $\th_N$ in
\reff{NonSymMeasureEst1} and for any $\tilde \a\in \cA^{t,0}_{path}$,
\bea
\label{NonSymMeasureEst2} \max_{1\le i\le N}\max_{t\le s\le
T}\dbE^{\dbP^{t, \vec \bx, (\vec\a^{-i}, \tilde
\a)}}\big[W_1(\mu^N_{s\wedge\cd}, \mu^\g_{s\wedge\cd})\big] \le
C\e+C_\e\th_N, \eea where $C_\e$ may depend on $\e$ as well.

(ii) Assume further Assumption \ref{assum-reg} (iii), then there
exists a modulus of continuity function $\rho_0$, depending only on
$T, d, L_q$, $C_0$, and $\rho$, such that,
\begin{equation}
  \label{FinNonSymCostEst2} \Big|J_i(t, \vec x, (\vec\a^{-i},\tilde \a))
  - J(\mu^\g; t, \bx^i, \tilde \a )\Big| + \big|v^{N,0}_i(t, \vec\bx,
  \vec\a) - v(\mu^\g; t, \bx^i)\big|\le \rho_0\big(C\e+C_\e\th_N\big).  
\end{equation}
Moreover, assume $\g\in \cM^{\e}_{relax}(t,\mu)$, then \bea
\label{FinNonSymMFE2} {1\over N} \sum_{i=1}^N \big[J_i(t, \vec \bx,
\vec \a) -v^{N,0}_i(t, \vec\bx, \vec\a)\big] \le \e+ 2
\rho_0\big(C\e+C_\e\th_N\big),\q\forall \bx\in \dbX_t.  \eea In
particular, this means that $\vec\a \in \cM^{N, \tilde \e,
0}_{hetero}(t, \vec \bx)$ with $\tilde \e := \e+ 2
\rho_0\big(C\e+C_\e\th_N\big)$.
\end{thm}
\begin{proof}
  (i) We first show by induction that \bea
  \label{NonSymMeasureEst2-1} \k_s:= W_1\big(\mu^\g_{s\wedge \cd},
  \mu^{\g^{\e}}_{s\wedge \cd}\big) \le C\e,\q s=t,\cds, T.  \eea
  Indeed, it is obvious that $\k_t = 0$. For $s\ge t$, by
  \reff{relaxJ}, \reff{Ae}, and \reff{ge}, we have \beaa &&\dis
  \k_{s+1} = \sum_{\bx\in \dbX_{s+1}} \big|\mu^\g_{(s+1)\wedge
    \cd}(\bx) - \mu^{\g^\e}_{(s+1)\wedge \cd}(\bx) \big|\\ &&\dis =
  \sum_{\bx\in \dbX_s, x\in \dbS}\Big|\mu^\g_{s\wedge \cd}(\bx)
  \int_{\dbA}\!\! q(s, \bx, \mu^\g, a; x) \g(s, \bx, da) -
  \mu^{\g^\e}_{s\wedge \cd}(\bx)
  \int_{\dbA}\!\! q(s, \bx, \mu^{\g^\e}, a; x) \g^\e(s, \bx, da) \Big|\\
  &&\dis \le \sum_{\bx\in \dbX_s, x\in \dbS}\Big[ \big|\mu^\g_{s\wedge
    \cd}(\bx) - \mu^{\g^\e}_{s\wedge \cd}(\bx)\big|+ \sum_{k=1}^{n_\e}
  \int_{A_k}\!\!\big|q(s, \bx, \mu^\g, a; x) - q(s, \bx, \mu^{\g^\e},
  a_k; x)\big| \g(s, \bx, da)\\ &&\dis \qq + \int_{A_0}q(s, \bx,
  \mu^\g, a; x) \g(s, \bx, da) + \int_{A_0}q(s, \bx, \mu^{\g^\e}, a;
  x) \g^\e(s, \bx, da) \\ &&\dis\le C\k_s + C\e.  \eeaa Then by
  induction we have \reff{NonSymMeasureEst2-1}.

  We next show by induction that, recalling \reff{nuN}, \bea
  \label{NonSymMeasureEst2-2} \bar\k_s:= W_1\big(\nu^N_{s\wedge\cd},
  \mu^{\g^\e}_{s\wedge\cd}\big) \le C_\e\th_N,\q s=t, \cds, T.  \eea
  Indeed, $\bar \k_t = W_1(\mu^N_{t,\vec \bx}, \mu)$. For $s\ge t$,
  noting that $\a_i \in \cA^{t,\e}_{path}\subset \cA^{t,0}_{path}$ and
  recalling from Lemma \ref{lem-equivalence} that
  $\mu^{\L^{\g^\e}} = \mu^{\g^\e}$, then by \reff{nuN} and \reff{LQ}
  that \beaa &&\dis \bar\k_{s+1} = W_1\big(\nu^N_{{s+1}\wedge\cd},
  \mu^{\L^{\g^\e}}_{(s+1)\wedge\cd}\big) \\ &&\dis = \sum_{\bx\in
    \dbX_t} \sum_{\tilde \bx\in \dbX^{t,\bx}_{s+1}} \Big|{1\over N}
  \sum_{\a\in \cA^{t, \e}_{path}}\sum_{i\in I(\bx,\a)}
  Q^t_{s+1}(\nu^N; \tilde \bx, \a) - \int_{\cA^t_{path}}
  Q^t_{s+1}(\mu^{\g^\e}; \tilde\bx, \a) \L^{\g^\e}(\bx, d\a) \Big|\\
  &&\dis = \sum_{\bx\in \dbX_t} \sum_{\tilde \bx\in
    \dbX^{t,\bx}_{s+1}} \Big| \sum_{\a\in \cA^{t,
      \e}_{path}}\big[\L^\e_{t,\vec\bx}(\bx, \a) Q^t_{s+1}(\nu^N;
  \tilde \bx, \a) - \L^{\g^\e}(\bx, \a) Q^t_{s+1}(\mu^{\g^\e}; \tilde
  \bx, \a) \big] \Big|\\ &&\dis \le \sum_{\bx\in \dbX_t} \sum_{\tilde
    \bx\in \dbX^{t,\bx}_{s+1}} \sum_{\a\in \cA^{t,
      \e}_{path}}\Big[\big|\L^\e_{t,\vec\bx}(\bx, \a) -
  \L^{\g^\e}(\bx, \a) \big|Q^t_{s+1}(\nu^N; \tilde \bx, \a) \\
  &&\dis\qq + \L^{\g^\e}(\bx, \a)\big| Q^t_{s+1}(\nu^N; \tilde \bx,
  \a) -Q^t_{s+1}(\mu^{\g^\e}; \tilde \bx, \a) \big| \Big].  \eeaa
  Then, by \reff{LdN} and noting that $C_\e := |\cA^{t, \e}_{path}|$
  is independent of $N$, we have \beaa \bar\k_{s+1} &\le& \sum_{\bx\in
    \dbX_t} \sum_{\tilde \bx\in \dbX^{t,\bx}_{s+1}} \sum_{\a\in
    \cA^{t, \e}_{path}}\Big[\th_N Q^t_{s+1}(\nu^N; \tilde \bx, \a) +
  C\L^{\g^\e}(\bx, \a) \sum_{r=t}^s
  W_1\big(\nu^N_{r\wedge\cd}, \mu^{\g^\e}_{r\wedge\cd}\big) \Big]\\
  &\le& C_\e \th_N + C\sum_{r=t}^s \bar \k_r.  \eeaa This implies
  \reff{NonSymMeasureEst2-2} immediately.
  
  Finally, combining \reff{NonSymMeasureEst2-1},
  \reff{NonSymMeasureEst2-2}, and \reff{NonSymMeasureEst1}, we obtain
  \reff{NonSymMeasureEst2}.
  
  (ii) First, similar to \reff{FinNonSymCostEst1}, by
  \reff{NonSymMeasureEst2} we have \reff{FinNonSymCostEst2} following
  from the arguments in Theorem \ref{thm-FinSymCostConv}. Next, for
  $\g\in \cM^{\e}_{relax}(t,\mu)$, by \reff{gLMFE} we have
  $\L^\g \in \cM^{\e}_{global}(t,\mu)$. Then \reff{FinNonSymMFE2}
  follows from similar arguments as those for \reff{FinNonSymMFE1}.
\end{proof}
\no{\bf Proof of Theorem \ref{thm-FinNonSymConv}: the right
  inclusion}.
  Fix $\f\in \dbV_{relax}(t,\mu)$ and $\e>0$. Let $\e_1>0$ be a small
  number which will be specified later. There exists
  $\g\in \cM^{\e_1}_{relax}(t,\mu)$ such that
  $\|\f - J(t,\mu, \g; \cd, \g)\|_{\dbX_t} \le \e_1$. Let $\g^{\e_1}$
  and $\vec \a$ be constructed as above. By \reff{FinNonSymMFE2} we
  have \beaa {1\over N} \sum_{i=1}^N \big[J_i(t, \vec \bx, \vec \a)
  -v^{N,0}_i(t, \vec\bx, \vec\a)\big] \le \e_1+ 2
  \rho_0\big(C\e_1+C_{\e_1}\th_N\big),\q\forall \bx\in \dbX_t.  \eeaa
  Choose $\e_1$ small enough such that
  $\e_1 + 2\rho_0(C\e_1 + \e_1) < \e$. Then, for all $N$ large enough
  such that $\th_N \le {\e_1\over C_{\e_1}}$, we have
  ${1\over N} \sum_{i=1}^N \big[J_i(t, \vec \bx, \vec \a)
  -v^{N,0}_i(t, \vec\bx, \vec\a)\big] \le \e$. That is,
  $\vec\a\in \dbV^{N, \e, 0}_{hetero}(t,\mu^N_{t,\vec \bx})$ for all
  $N$ large enough. Then, following the same arguments as those in the
  proof for the left inclusion, we can easily get
  $\f \in \dbV^{N, \e, 0}_{hetero}(t,\mu^N_{t,\vec \bx})$ for all $N$
  large enough, and thus
  $\f \in \liminf_{N\to\infty}\dbV^{N, \e, 0}_{hetero}(t,\mu^N_{t,\vec
    \bx})$. Since $\e>0$ is arbitrary, we get the desired inclusion.
\qed

\section{A continuous time model with controlled diffusions}
\label{sect-Diffusion} 
\setcounter{equation}{0} 
In this section we study a  continuous time model where the state process is a controlled diffusion with closed loop drift controls. In this case the
laws of the controlled state process are all equivalent. The
volatility control case involves mutually singular measures
(corresponding to degenerate $q$ in the discrete setting) and is much
more challenging. We shall leave that for future research.  To ensure
the convergence, we consider state dependent homogeneous controls
for the $N$-player games, as we did in Section 3.

\subsection{The mean field game and the dynamic programming principle}
\label{sect-contDPP} 
Let $T>0$ be a fixed terminal time, $(\O, \cF,
\dbF=\{\cF_t\}_{0\le t\le T}, \dbP)$ a filtered probability space
where $\cF_0$ is atomless; $B$ a $d$-dimensional Brownian motion; and
the set $\dbA\subset \dbR^{d_0}$ a Borel measurable set. The state
process $X$ will also take values in $\dbR^d$. Its law lies in the
space $\cP_2:=\cP_2(\dbR^d)$ equipped with the $2$-Wasserstein
distance $W_2$. We remark that in the finite state space case $W_1$
and $W_2$ are equivalent, while in continuous models they are not. In
fact, at below we shall require $W_1$-regularity, which is stronger
than the $W_2$-regularity, and obtain $W_1$-convergence, which is
weaker than the $W_2$-convergence. This is not surprising in the mean
field literature, see, e.g. \cite{MZ}. The main advantage of the
$W_1$-distance is the following well known representation, see
e.g. \cite{CD1}: for any $\mu, \tilde\mu\in \cP_1(\dbR^d)$,
\begin{equation}
  \label{contW1} W_1(\mu, \tilde \mu) =
  \sup\Big\{\int_{\dbR^d}\!\!\f(x)[\mu(dx) - \tilde \mu(dx)]: \f\in
  C_{Lip}(\dbR^d)~\mbox{s.t.}~ |\f(x)-\f(\tilde x)|\le |x-\tilde
  x|\Big\}.   
\end{equation}
Here $C_{Lip}(\dbR^d)$ denote the set of uniformly Lipschitz
continuous functions $\f: \dbR^d\to \dbR$. Moreover, for each
$(t, \mu) \in [0,T]\times \cP_2$, let $\dbL^2(t,\mu)$ denote the set
of $\cF_t$-measurable random variables $\xi$ whose law (under $\dbP$)
$\cL_\xi = \mu$.

We consider coefficients $(b, f): [0, T]\times \dbR^d \times \cP_2
\times \dbA\to (\dbR^d, \dbR)$ and $g: \dbR^d \times \cP_2\to \dbR$.
Throughout this section, the following assumptions will always be in
force.

\begin{assum}
  \label{assum-diffusion}
   (i) $b, f, g$ are Borel measurable in $t$ and bounded by $C_0$ (for
  simplicity);

  (ii) $b, f, g$ are uniformly Lipschitz continuous in $(x, \mu, a)$
  with a Lipschitz constant $L_0$, where the Lipschitz continuity in
  $\mu$ is under $W_1$.
\end{assum}

Let $\cA_{cont}$ denote the set of admissible controls $\a: [0,
T]\times \dbR^d \to \dbA$ which is measurable in $t$ and Lipschitz
continuous in $x$, with the Lipschitz constant $L_\a$ possibly
depending on $\a$. Given $(t,\mu)\in [0, T]\times \cP_2$, $\xi\in
\dbL^2(t,\mu)$, and $\a\in \cA_{cont}$, consider the McKean-Vlasov
SDE: \bea
 \label{Xmu} X^{t,\xi,\a}_s = \xi + \int_t^s b(r, X^{t,\xi,\a}_r,
\mu^\a_r, \a(r, X^{t,\xi,\a}_r)) dr + B_s - B_t,\q \mu^\a_s :=
\cL_{X^{t,\xi,\a}_s}.  \eea By the required Lipschitz continuity, the
above SDE is wellposed, and it is obvious that $\mu^\a_t = \mu$ and
$\mu^\a_s$ does not depend on the choice of $\xi\in
\dbL^2(t,\mu)$. Then, when only the law is involved, by abusing the
notations we may also denote $X^{t,\xi,\a}$ as $X^{t,\mu, \a}$.
 
 Next, for any $x\in \dbR^d$, and $\tilde \a\in \cA_{cont}$, we
introduce \bea
 \label{Xmua} \left.\ba{c} \dis J(t,\mu,\a; x, \tilde \a) := J(\mu^\a;
t, x, \tilde \a),\q v(\mu^\a; s, x) := \inf_{\tilde \a\in \cA_{cont}}
J(\mu^\a; s, x, \tilde \a), s\ge t,\q \mbox{where}\\ \dis X^{\mu^\a;
s, x, \tilde \a}_r = x + \int_s^r b(l, X^{\mu^\a; s, x, \tilde \a}_l,
\mu^\a_l, \tilde\a(l, X^{\mu^\a; s, x, \tilde \a}_l)) dl + B_r - B_s,
~r\ge s;\\ \dis J(\mu^\a; s, x, \tilde \a) := \dbE\Big[g(X^{\mu^\a; s,
x, \tilde \a}_T, \mu^\a_T) + \int_s^T f(r, X^{\mu^\a; s, x, \tilde
\a}_r, \mu^\a_r, \tilde\a(r, X^{\mu^\a; s, x, \tilde \a}_r)) dr\Big].
\ea\right.  \eea Here we abuse the notations by using the same
notations as in the discrete setting. Clearly $u(s, x) := J(\mu^\a; s,
x, \tilde \a)$ and $v(s, x):= v(\mu^\a; s,x)$ satisfy the following
linear PDE and standard HJB equation on $[t, T]\times \dbR^d$,
respectively, with parameter $\mu^\a$:
\begin{equation}
  \begin{aligned}
    \label{HJB}\pa_s u (s,x) + {1\over 2}
    \tr\big(\pa_{xx} u(s,x)\big) + b(s, x, \mu^\a_s, \tilde \a(s, x))
    \cd \pa_x u(s,x) + f(s, x, \mu^\a_s, \tilde\a(s,x))&=0;
    \\ \pa_t v(s, x) + {1\over 2} \tr\big(\pa_{xx} v(s, x)\big) +
    \inf_{a\in \dbA} \big[
    b(s, x, \mu^\a_s, a) \cd \pa_x v(s, x)+ f(s, x, \mu^\a_s, a)
    \big]&=0;
    \\ u(T,x) = v(T, x) = g(x, \mu^\a_T).\hspace{10em}&
  \end{aligned}
\end{equation}

\begin{defn}
\label{defn-contMeL} Fix $(t,\mu)\in [0, T]\times \cP_2$. For any
$\e>0$, we say $\a^*\in \cA_{cont}$ is an $\e$-MFE at $(t,\mu)$,
denoted as $\a^*\in \cM^\e_{cont}(t,\mu)$, if 
\bea
\label{contMeL} 
\int_{\dbR^d} \big[J(t, \mu,\a^*; x, \a^*) -
v(\mu^{\a^*}; t, x)\big] \mu(dx) \le \e. 
 \eea
\end{defn}

\begin{rem}
\label{rem-contMeL}
Similar to \reff{NEPath} and \reff{Carmona}, here we do not require
  $\a^*$ to be optimal for every player $x$. In fact, alternatively, we may replace \reff{contMeL} with
\bea
\label{contMeL-equivalent}
    \mu\Big\{x \ :\ |
    J(t, \mu,\a^*; x, \a^*) - v(\mu^{\a^*}; t, x)| > \e \Big\} < \e.
    \eea
The intuition is that, since there are infinitely many players, we shall tolerate that a small portion of
  players may not be happy for the $\a^*$, as in \cite{Carmona}, and
  their possible deviation from $\a^*$ won't change the equilibrium
  measure $\mu^{\a^*}$ significantly.  We note that, although \reff{contMeL-equivalent} and \reff{contMeL} are not equivalent for fixed $\e$, they define the same set value in \reff{contV}  below, and the proofs are slightly easier by using \reff{contMeL}. 
   
 However,  if we require the $\e$-optimality for $\mu$-a.e. $x$, namely the probability in the left side of \reff{contMeL-equivalent} becomes $0$, then the set value will be different and may not satisfy the DPP. Such difference would disappear in the discrete model though.
\end{rem}

To define the set value, we need the following simple but crucial
regularity result, whose proof is postponed to Appendix.
\begin{lem}
\label{lem-vreg} Let Assumption \ref{assum-diffusion} hold. There
exists a constant $C>0$, depending only on $T, d, C_0, L_0$, such
that, for any $t, \mu, \a, \tilde \a$ and $s\ge t$,
\begin{equation}
  \label{vLip} \big|J(\mu^\a; \tilde \a, s,
    x)-J(\mu^\a; \tilde \a, s, \tilde x)\big| + \big|v(\mu^\a; s, x) -
    v(\mu^\a; s, \tilde x)\big|\le
    C |x-\tilde x|,\q\forall x, \tilde x.
  \end{equation}
\end{lem}

We then define the set value of the mean field game: \bea
\label{contV} \left.\ba{c} \dis \dbV_{cont}(t,\mu) := \bigcap_{\e>0}
\dbV^\e_{cont}(t, \mu),\q\mbox{where}\\ \dis \dbV^\e_{cont}(t,\mu) :=
\Big\{ \f\in C_{Lip}(\dbR^d): ~\mbox{there exists $\a^*\in
\cM^\e_{cont}(t,\mu)$ such that}\ms\\ \dis \int_{\dbR^d}
\big|\f(x)-J(t, \mu,\a^*; x, \a^*)\big| \mu(dx)\le \e\Big\}.
\ea\right.  \eea In particular, since $J(t, \mu,\a^*; x, \a^*) \ge
v(\mu^{\a^*}; t, x)$, then by \reff{vLip} and \reff{contMeL} we see
that both $J(t, \mu,\a^*; \cd, \a^*) $ and $v(\mu^{\a^*}; t, \cd)$
belong to $\dbV_{cont}(t,\mu)$. Moreover, again due to \reff{contMeL},
we may replace the inequality in the last line of \reff{contV} with $
\int_{\dbR^d} \big|\f(x)-v(\mu^{\a^*}; t, x)\big| \mu(dx)\le \e$.

Similarly, given $T_0$ and $\psi\in C_{Lip}(\dbR^d)$, we may define
the functions $J(T_0, \psi; t, \mu, \a; x,\tilde \a)$, $J(T_0, \psi;
\mu^\a; s, x, \tilde \a)$, $v(T_0, \psi; \mu^\a; s, x)$, as well as
the sets $\cM^\e_{cont}(T_0, \psi; t, \mu)$, $\dbV^\e_{cont}(T_0,\psi;
t, \mu)$, $\dbV_{cont}(T_0, \psi; t, \mu)$ in the obvious sense. In
particular, we have the following tower property:
\begin{equation}
  \begin{aligned}
    \label{tower2}
    J(t, \mu, \a; x, \tilde \a) &=
    J(T_0, \psi;
    t, \mu, \a; x, \tilde \a),\q\mbox{where}\q \psi(x):= J(T_0,
    \mu^\a_{T_0}, \a; x, \tilde \a);\\ v(\mu^{\a}; t, x) &= v(T_0, \tilde
    \psi; \mu^{\a}; t, x),\hspace{2.7em}\mbox{where}\q \tilde\psi(x) :=
    v(\mu^\a; T_0, x).  
  \end{aligned}
\end{equation}

We now establish the DPP for $\dbV_{cont}(t,\mu)$.

\begin{thm}
\label{thm-contDPP} Let Assumption \ref{assum-diffusion} hold. For any
$0\le t\le T_0\le T$ and $\mu\in \cP_2$, it holds \bea
\label{contDPP} \left.\ba{c} \dis \dbV_{cont}(t, \mu) =\tilde
\dbV_{cont}(t, \mu) := \bigcap_{\e>0} \tilde \dbV^\e_{cont}(t,
\mu),\q\mbox{where}\\ \dis \tilde \dbV^\e_{cont}(t, \mu) :=
\Big\{\f\in C_{Lip}(\dbR^d): \int_{\dbR^d} |\f(x)-J(T_0,\psi; t, \mu,
\a^*; x,\a^*)|\mu(dx)\le \e,\\ \dis \mbox{for some $(\psi,\a^*)$
satisfying:} ~\psi \in \dbV^\e_{cont}(T_0, \mu^{\a^*}_{T_0}), \a^*\in
\cM^\e_{cont}(T_0, \psi; t,\mu)\Big\}.  \ea\right.  \eea
\end{thm}
\begin{proof}
  (i) We first prove
  $\dbV_{cont}(t,\mu) \subset \tilde \dbV_{cont}(t,\mu)$. Fix
  $\f\in \dbV_{cont}(t,\mu)$, $\e>0$, and set $\e_1:= {\e\over
    2}$. Since $\f\in \dbV^{\e_1}_{cont}(t, \mu)$, there exists
  $\a^*\in \cM^{\e_1}_{cont}(t, \mu)$ satisfying \reff{contV} for
  $\e_1$.  Denote \beaa \psi(x):= J(T_0, \mu^{\a^*}_{T_0}, \a^*; x,
  \a^*),\q \tilde \psi(x) := v(\mu^{\a^*}; T_0, x).  \eeaa By
  \reff{tower2} we have
  $J(T_0, \psi; t, \mu, \a^*; x, \a^*) = J(t, \mu, \a^*; x, \a^*)$ and
  thus \beaa \int_{\dbR^d} \big|\f(x)-J(T_0, \psi; t, \mu,\a^*; x,
  \a^*)\big| \mu(dx)\le \e_1\le \e.  \eeaa We shall show that
  $\psi \in \dbV^\e_{cont}(T_0, \mu^{\a^*}_{T_0})$ and
  $\a^*\in \cM^\e_{cont}(T_0, \psi; t,\mu)$. Then
  $\f\in \tilde \dbV^\e_{cont}(t, \mu)$, and therefore, since $\e>0$
  is arbitrary, we have $\f\in \tilde \dbV(t,\mu)$.

  {\it Step 1.} In this step we show that
  \begin{equation}
    \label{contDPPest1} \int_{\dbR^d}\big[J(T_0, \mu^{\a^*}_{T_0}, \a^*;
    x, \a^*)- v(\mu^{\a^*}; T_0, x)\big] \mu^{\a^*}_{T_0}(dx) =
    \int_{\dbR^d}[\psi( x)- \tilde \psi(x)] \mu^{\a^*}_{T_0}(dx)\le \e_1.  
  \end{equation}
  Then $\a^*\in \cM^\e_{cont}(T_0, \mu^{\a^*}_{T_0})$, which, together
  with the regularity of $\psi$ from Lemma \ref{lem-vreg}, implies
  immediately that $\psi \in \dbV^\e_{cont}(T_0, \mu^{\a^*}_{T_0})$.

  To see this, we recall \reff{Xmu} with $\xi \in \dbL^2(t,
  \mu)$. Since $\a^*\in \cM^{\e_1}_{cont}(t, \mu)$, by \reff{tower2}
  we have \beaa &\dis \e_1 \ge \dbE\Big[J(t, \mu, \a^*; \xi, \a^*) -
  v(\mu^{\a^*}; t, \xi)\Big]= \dbE\Big[J(T_0, \psi; t, \mu, \a^*; \xi,
  \a^*) - v(T_0, \tilde \psi; \mu^{\a^*}; t, \xi)\Big]\\ &\dis \ge
  \dbE\Big[J(T_0, \psi; t, \mu, \a^*; \xi, \a^*) - J(T_0, \tilde \psi;
  t, \mu, \a^*; \xi, \a^*)\Big]=\dbE\Big[ \psi(X^{t, \xi, \a^*}_{T_0})
  - \tilde \psi(X^{t, \xi, \a^*}_{T_0})\Big].  \eeaa Note that
  $\cL_{X^{t, \xi, \a^*}_{T_0}} = \mu^{\a^*}_{T_0}$, then this is
  exactly \reff{contDPPest1}.

  {\it Step 2.} It remains to show that
  $ \a^*\in \cM^\e_{cont}(T_0, \psi; t,\mu)$. By the definition of $v$
  and its regularity from Lemma \ref{lem-vreg}, there exists
  $\tilde \a^*\in \cA_{cont}$ such that \beaa J(T_0, \psi; t, \mu,
  \a^*; x, \tilde \a^*) \le v(T_0,\psi; \mu^{\a^*}; t, x) +
  \e_1,\q\forall x\in \dbR^d.  \eeaa Then, denoting
  $\hat\a^* := \tilde \a^* \oplus_{T_0} \a^*\in \cA_{cont}$, by
  \reff{tower2} again we have \beaa &&\dis \dbE\Big[J(T_0, \psi; t,
  \mu, \a^*; \xi, \a^*) - v(T_0,\psi; \mu^{\a^*}; t, \xi) \Big]\\
  &&\dis \le \dbE\Big[J(T_0, \psi; t, \mu, \a^*; \xi, \a^*) - J(T_0,
  \psi; t, \mu, \a^*; \xi, \tilde\a^*)\Big] + \e_1 \\ &&\dis =
  \dbE\Big[J(t, \mu,
  \a^*; \xi, \a^*) - J(t, \mu, \a^*; \xi, \hat \a^*)\Big] + \e_1\\
  &&\dis \le \dbE\Big[J(t, \mu, \a^*; \xi, \a^*) - v(\mu^{\a^*}; t,
  \xi)\Big] +\e_1\le \e_1 + \e_1 = \e, \eeaa This means
  $ \a^*\in \cM^\e_{cont}(T_0, \psi; t,\mu)$.

  (ii) We next prove
  $\tilde \dbV_{cont}(t,\mu) \subset \dbV_{cont}(t,\mu)$. Fix
  $\f\in \tilde\dbV_{cont}(t,\mu)$, $\e>0$, and set
  $\e_1:= {\e\over 4}$. Since
  $\f\in \tilde\dbV^{\e_1}_{cont}(t, \mu)$, there exist $(\psi, \a^*)$
  satisfying the desired properties in \reff{contDPP} for $\e_1$. In
  particular, since
  $\psi\in \dbV^{\e_1}_{cont}(T_0, \mu^{\a^*}_{T_0})$, there exists
  desired $\tilde \a^*\in \cM^{\e_1}_{cont}(T_0, \mu^{\a^*}_{T_0})$
  required in \reff{contV} for $\e_1$. Denote
  $\hat \a^* := \a^* \oplus_{T_0} \tilde \a^*\in \cA_{cont}$ and \beaa
  \hat\psi(x) := J(T_0, \mu^{\a^*}_{T_0}, \tilde \a^*; x, \tilde
  \a^*),\q \tilde \psi(x):= v(\mu^{\hat\a^*}; T_0, x).  \eeaa By
  \reff{contDPP}, \bea
  \label{Jdif} &&\dis\qq\qq\qq \dbE\Big[\big|J(T_0, \psi; t, \mu,
  \a^*; \xi, \a^*) - J(T_0,\hat\psi; t, \mu, \a^*; \xi,
  \a^*)\big|\Big] \\&&\dis\!\!\!\!\!\!\!\!\!\!=
  \dbE\Big[\big|\psi(X_{T_0}^{\mu^{\a^*};t,\xi,\a^*}) -
  \hat\psi(X_{T_0}^{\mu^{\a^*};t,\xi,\a^*})\big|\Big]=\int_{\dbR^d}\big|\psi(x)
  - J(T_0, \mu^{\a^*}_{T_0}, \tilde \a^*; x, \tilde
  \a^*)\big|\mu_{T_0}^{\a^*}(dx)\leq \e_1 \nonumber \eea Then, since
  $\f\in \tilde\dbV^{\e_1}_{cont}(t, \mu)$ with corresponding
  $(\psi, \a^*)$, by \reff{tower2} and \reff{Jdif} we have \beaa
  \dbE\Big[\big|\f(\xi) - J(t, \mu, \hat\a^*; \xi, \hat\a^*)\big|\Big]
  \leq \dbE\Big[\big|\f(\xi)-J(T_0, \psi; t, \mu, \a^*; \xi, \a^*)
  \big|\Big] + \e_1 \le 2\e_1\le \e, \eeaa where
  $\xi\in \dbL^2(t,\mu)$. We claim further that
  $\hat \a^*\in \cM^\e_{cont}(t, \mu)$. Then
  $\f\in \dbV^\e_{cont}(t, \mu)$, and thus $\f\in \dbV_{cont}(t,\mu)$,
  since $\e>0$ is arbitrary.

  To see the claim, since
  $\a^* \in \cM^{\e_1}_{cont}(T_0, \psi; t,\mu)$,
  $\tilde \a^*\in \cM^{\e_1}_{cont}(T_0, \mu^{\a^*}_{T_0})$, by
  \reff{tower2} we have \beaa &&\dis \dbE\Big[J(t, \mu, \hat\a^*; \xi,
  \hat\a^*) - v(\mu^{\hat\a^*}; t, \xi) \Big]\\ &&\dis =
  \dbE\Big[J(T_0, \hat\psi; t, \mu, \a^*; \xi, \a^*) - v(T_0, \tilde
  \psi; \mu^{\a^*}; t, \xi) \Big]\\ &&\dis \leq \dbE\Big[J(T_0, \psi;
  t, \mu, \a^*; \xi, \a^*) - v(T_0, \tilde \psi; \mu^{\a^*}; t, \xi)
  \Big] + \e_1\\ &&\dis \le \dbE\Big[ v(T_0, \psi; \mu^{\a^*}; t, \xi)
  - v(T_0, \tilde \psi; \mu^{\a^*}; t, \xi) \Big] + 2\e_1\\ &&\dis \le
  \sup_{\tilde\a\in \cA_{cont}}\dbE\Big[J(T_0, \psi; t, \mu, \a^*;
  \xi, \tilde \a) -
  J(T_0, \tilde\psi; t, \mu, \a^*; \xi, \tilde \a) \Big] + 2\e_1\\
  &&\dis = \dbE\big[\psi(X^{t,\xi, \a^*}_{T_0}) - \tilde\psi(X^{t,\xi,
    \a^*}_{T_0})\big] + 2\e_1 \le \dbE\big[\hat\psi(X^{t,\xi,
    \a^*}_{T_0}) - \tilde\psi(X^{t,\xi, \a^*}_{T_0})\big] + 3\e_1 \leq
  \e_1 + 3\e_1 = \e.  \eeaa This means
  $\hat \a^*\in \cM^\e_{cont}(t, \mu)$, and hence completes the proof.
\end{proof}

\begin{rem}
\label{rem-contV}
(i) Our set value $\dbV_{cont}(t, \mu)$ is defined for each $(t, \mu)$ with elements in $C_{Lip}(\dbR^d)$, instead of $\dbV(t, x, \mu)\subset \dbR$ for each $(t,x,\mu)$. This is consistent with \reff{V0}  in the discrete model, and is due to the fact that an $\e$-MFE $\a^*$ in Definition \ref{defn-contMeL} depends on $(t, \mu)$, but is common for all initial states $x$. Indeed, if we define $\dbV_{cont}(t,x,\mu)$ in an obvious manner, it will not satisfy the DPP.  

(ii) The above observation is also consistent with the fact that the following master equation is local in $(t, \mu)$, but non-local in $x$ due to the term $ \pa_x V(t, \tilde x, \mu)$:
\bea
\label{master}
\left.\ba{c}
\dis \pa_t V(t,x,\mu) + {1\over 2}\tr(\pa_{xx} V) + H(x,\mu, \pa_x V) \\
\dis + \int_{\dbR^d} \big[{1\over 2}\tr( \pa_{\tilde x\mu} V(t,x,\mu, \tilde x)) + \pa_p H(\tilde x, \mu, \pa_x V(t, \tilde x, \mu)) \pa_\mu V(t, x, \mu, \tilde x)\big] \mu(d\tilde x)=0. 
\ea\right.
\eea
Under appropriate conditions, in particular under certain monotonicity conditions, the above master equation has a unique solution and we have $\dbV_{cont}(t, \mu) = \{\cV(t, \mu)\}$ is a singleton, where  $\cV(t,\mu)(x) := V(t,x,\mu)$ is a function of $x$. In this way, we may also view \reff{master} as a first order ODE on the space $C^2(\dbR^d)$ (the regularity in $x$ is  a lot easier to obtain):
\bea
\label{master2}
\left.\ba{c}
\dis \pa_t \cV(t,\mu) + \cH(\mu,  \cV(t, \mu)) + \cM(\mu,  \cV(t, \mu), \pa_\mu\cV(t, \mu))=0, \\
\dis  \mbox{where}\q \cH(\mu, v(\cd))(x) :=  {1\over 2}\tr(\pa_{xx} v(x)) + H(x,\mu, \pa_x v(x)),\\
\dis \cM(\mu, v(\cd), \tilde v(\cd, \cd))(x) := \int_{\dbR^d} \big[{1\over 2}\tr( \pa_{\tilde x} \tilde v(x,\tilde x)) + \pa_p H(\tilde x, \mu, \pa_x v(\tilde x)) \tilde v(x, \tilde x)\big] \mu(d\tilde x). 
\ea\right.
\eea
It could be interesting to explore master equations from this perspective as well.
\end{rem}

\subsection{Convergence of the $N$-player game} By enlarging the
filtered probability space $(\O, \cF, \dbF, \dbP)$, if necessary, we
let $B^1, \cds, B^N$ be independent $d$-dimensional Brownian motions
on it. Set $\dis\cA^\infty_{cont} := \cup_{L\ge 0} \cA^L_{cont}$,
where, for each $L\ge 0$, $\cA^L_{cont}$ denotes the set of admissible
controls $\a: [0, T]\times \dbR^d \times \cP_2 \to \dbA$ such that
\beaa |\a(t, x, \mu) - \a(t, \tilde x, \tilde \mu)| \le L_\a|x-\tilde
x| + LW_1(\mu, \tilde \mu).  \eeaa Here the Lipschitz constant $L_\a$
may depend on $\a$, hence the Lipschitz continuity in $x$ is not
uniform in $\a$. We emphasize that the Lipschitz continuity in $\mu$
is under $W_1$, rather than $W_2$, so that we can use the
representation \reff{contW1}.  Note that $\cA_{cont} = \cA^0_{cont}$,
and by Remark \ref{rem-cAL} (i), all the results in the previous
subsection remain true if we replace $\cA_{cont}$ with
$\cA^\infty_{cont}$.

Given $t\in [0, T]$, $\vec x =(x_1,\cds, x_N) \in \dbR^{dN}$ and $\vec
\a = (\a_1,\cds, \a_N)\in (\cA^L_{cont})^N$, consider \bea
\label{contXi} \left.\ba{c} \dis X^{t, \vec x, \vec \a; i}_s = x_i +
\int_t^s \!\! b\big(r, X_r^{t, \vec x, \vec \a; i}, \mu^{t, \vec x,
\vec\a}_r, \a_i(r, X_r^{t, \vec x, \vec \a; i}, \mu^{t, \vec x,
\vec\a}_r)\big) dr + B^i_s-B^i_t, i=1,\cds, N;\\ \dis\mbox{where}\q
\mu^{t,\vec x, \vec \a}_s := {1\over N}\sum_{i=1}^N\d_{X^{t, \vec x,
\vec \a; i}_s};\\ \dis J_i(t, \vec x, \vec \a) := \dbE\Big[g(X^{t,
\vec x, \vec \a; i}_T, \mu^{t, \vec x, \vec\a}_T) + \int_t^T\!\!
f\big(s, X_s^{t, \vec x, \vec \a; i}, \mu^{t, \vec x, \vec\a}_s,
\a_i(s, X_s^{t, \vec x, \vec \a; i}, \mu^{t, \vec x, \vec\a}_s)\big)
ds\Big],\\ \dis v^{N,L}_i(t, \vec x, \vec \a) := \inf_{\tilde \a\in
\cA^L_{cont}} J_i(t, \vec x, (\vec \a^{-i}, \tilde \a)).  \ea\right.
\eea

In light of Lemma \ref{lem-vreg}, the following regularity result is
interesting in its own right. However, since it will not be used for
our main result, we postpone its proof to Appendix.
\begin{prop}
\label{prop-contNreg} Let Assumption \ref{assum-diffusion} hold. For
any $L\ge 0$, there exists a constant $C_L>0$, depending only on $T,
d, C_0, L_0$, and $L$, such that, for any $(t, \vec x)\in [0, T]\times
\dbR^{dN}$, $\bar x, \tilde x\in \dbR^d$, and $\vec \a \in
(\cA^L_{cont})^N$, we have \bea
\label{vLireg} \big|v^{N,L}_i\big(t, (\vec x^{-i}, \bar x), \vec
\a\big)- v^{N,L}_i\big(t, (\vec x^{-i}, \tilde x), \vec
\a\big)\big|\le C_L|\bar x-\tilde x|,\q i=1,\cds, N.  \eea
\end{prop}

Given $\a\in \cA^L_{cont}$, by viewing it as the homogeneous control
$(\a,\cds, \a)$, we may use the simplified notations $X^{t, \vec x,
\a; i}$, $\mu^{t, \vec x, \a}$, $J_i(t, \vec x, \a)$, and $
v^{N,L}_i(t, \vec x, \a)$ in the obvious sense.

\begin{defn}
  \label{defn-contNcM} (i) For $(t, \vec x) \in [0, T] \times \dbR^{dN}$\!, $\e> 0$,  $L\ge 0$, we
  call $\a^*\in \cA^L_{cont}$ a homogeneous $(\e, L)$-equilibrium of
  the $N$-player game at $(t, \vec x)$, denoted as
  $ \a^* \! \in\!  \cM^{N,\e,L}_{cont}(t, \vec x)$, if \bea
\label{contNcM} {1\over N}\sum_{i=1}^N \big[J_i(t,\vec x, \a^*) -
v^{N,L}_i(t, \vec x, \a^*)\big] \le \e.  \eea

(ii) The set value for the $N$-player game is defined as: \bea
\label{contVN} &&\dis\qq\qq \dbV^N_{cont}(t,\vec x) := \bigcap_{\e>0}
\dbV^{N,\e}_{cont}(t, \vec x):= \bigcap_{\e>0}\bigcup_{L\ge 0}
\dbV^{N,\e,L}_{cont}(t, \vec x),\q\mbox{where}\\ &&\dis \!\!\!\!\!\!
\dbV^{N,\e,L}_{cont}(t,\vec x) := \Big\{ \f\in C_{Lip}(\dbR^d):
\exists \a^*\in \cM^{N,\e,L}_{cont}(t,\vec x)~\mbox{s.t.}~{1\over
N}\sum_{i=1}^N |\f(x_i)-J_i(t,\vec x, \a^*) | \le \e\Big\}.\nonumber
\eea
\end{defn}

We remark that, although $\dbV^{N,\e,L}_{cont}(t,\vec x)$ involves
only the values $\{\f(x_i)\}_{1\le i\le N}$, for the convenience of
the convergence analysis we consider its elements as $\f\in
C_{Lip}(\dbR^d)$.

\begin{rem}
  \label{rem-contNJN}
  (i) Recall \reff{muN}. By the required symmetry, obviously there
  exist functions
  $J^N, v^{N,L}: [0, T] \times \cP_2 \times \cA^L_{cont} \times
  \dbR^d\to \dbR$ such that
  \begin{equation}
    \label{contJiN} J_i(t, \vec x, \a) = J^N(t, \mu^N_{\vec x}, \a;
    x_i),\ \ v^{N,L}_i(t, \vec x, \a) = v^{N,L}(t, \mu^N_{\vec x}, \a;
    x_i),\ \ i=1,\cds, N. 
  \end{equation}
  Moreover, $\dbV^N_{cont}(t,\vec x)$ is invariant in $\mu^N_{\vec x}$
  and thus can be denoted as $\dbV^N_{cont}(t, \mu^N_{\vec x})$.

  (ii) The required inequalities in Definition \ref{defn-contNcM} are
  equivalent to: \beaa \int_{\dbR^d} [J^N- v^{N,L}](t, \mu^N_{\vec x},
  \a^*; x) \mu^N_{\vec x}(dx)\le \e,\q \dis\int_{\dbR^d} \big[\f(x) -
  J^N(t, \mu^N_{\vec x}, \a^*; x) \big]\mu^N_{\vec x}(dx)\le \e.
  \eeaa
\end{rem}

We now turn to the convergence, starting with the convergence of the
equilibrium measures.  Recall the vector $(\a, \tilde \a)_i$
introduced in \reff{NE}.

\begin{thm}
\label{thm-contmuConv} Let Assumption \ref{assum-diffusion} hold. For
any $L\ge 0$, there exists a constant $C_L>0$, depending only on $T,
d, C_0, L_0$, and $L$, such that, for any $t\in [0, T]$, $\vec x\in
\dbR^{dN}$, $\mu\in \cP_2$, $\a, \tilde \a\in \cA^L_{cont}$, and
$i=1,\cds, N$, \bea
   \label{contmuConv} &\dis \sup_{t\le s\le T} \dbE\Big[W_1(\mu^{t,
\vec x, (\a,\tilde\a)_i}_s,\mu_s^\a)\Big] \le C_L\th_N,\\ &\dis
\mbox{where}\q \th_N:= W_1(\mu^N_{\vec x}, \mu)+ N^{-{1\over {d\vee
3}}} \|\vec x\|_2 + N^{-1},\q \|\vec x\|_2^2:= {1\over N}\sum_{i=1}^N
|x_i|^2.\nonumber \eea
\end{thm}
\proof  Recall \reff{contXi} and introduce, for $j=1,\cds, N$, \bea
  \label{contXi2} \left.\ba{c} \dis \tilde X^j_s = x_j+ \int_t^s b(r,
    \tilde X_r^j, \mu^\a_r, \a(r, \tilde X_r^j, \mu^\a_r)) dr +
    B^j_s-B^j_t,\ \ \tilde \mu^N_s := {1\over N} \sum_{j=1}^N
    \d_{\tilde X^j_s};\\ \dis \tilde X_s = \tilde \xi+ \int_t^s b(r,
    \tilde X_r, \mu^\a_r, \a(r, \tilde X_r, \mu^\a_r)) dr + B_s-B_t, \
    \mbox{where}\ \tilde\xi\in \dbL^2(\cF_0; \mu^N_{\vec x}).
    \ea\right.  \eea Note that $\tilde X^1,\cds, \tilde X^N$ are
  independent. We proceed the rest of the proof in two steps.

  {\it Step 1.} In this step we estimate
  $\dbE\big[W_1(\tilde \mu^N_s,\mu_s^\a)\big]$.  First, by \cite[Lemma
  8.4]{MZ} we have \beaa \dbE\big[W_1(\tilde \mu^N_s, \cL_{\tilde
    X_s})\big] \le CN^{-{1\over {d\vee 3}}} \|\vec x\|_2.  \eeaa Next,
  fix an $\f$ in \reff{contW1} and let $u=u_\f$ denote the solution to
  the following PDE on $[t, s]$: \bea
  \label{HJB2} \pa_r u + {1\over 2} \tr\big(\pa_{xx} u\big) + b(r, x,
  \mu^\a_s, \a(r, x, \mu^\a_r)) \cd \pa_x u =0,\q u(s, x) = \f(x).
  \eea Applying Lemma \ref{lem-vreg} with
  $\tilde \a(r, x) := \a(r, x, \mu^\a_r)$ and $f=0$, we see that $u$
  is uniformly Lipschitz continuous in $x$, with a Lipschitz constant
  $C$ independent of $\f$ and $L$. Thus, \beaa \dbE\big[ \f(\tilde
  X_s) - \f(X^\a_s)\big] = \dbE\big[ u(t, \tilde \xi) - u(t, \xi)\big]
  \le C\dbE[ |\tilde \xi-\xi|].  \eeaa Since $\cF_0$ is atomless, we
  may choose $\xi, \tilde \xi$ such that
  $\dbE[ |\tilde \xi-\xi|] = W_1(\mu^N_{\vec x}, \mu)$, then
  \reff{contW1} implies
  $ W_1(\cL_{\tilde X_s}, \mu^\a_s) \le CW_1(\mu^N_{\vec x}, \mu).  $
  Put together, we have \bea
  \label{contmuConv2} \dbE\big[W_1(\tilde \mu^N_s, \mu^\a_s)\big] \le
  CW_1(\mu^N_{\vec x}, \mu)+ CN^{-{1\over {d\vee 3}}} \|\vec x\|_2\le
  C\th_N,\q t\le s\le T.  \eea

  {\it Step 2.} We next estimate
  $\dbE\big[W_1(\mu^{t, \vec x, (\a, \tilde \a)_i}_s, \mu^\a_s)\big]$.
  Denote $\a_i:=\tilde \a$, $\a_j:= \a$ for $j\neq i$, and \beaa
  &\b^j_s:= b(s, \tilde X_s^j, \tilde \mu^N_s, \a_j(s, \tilde X_s^j,
  \tilde \mu^N_s)) - b(s, \tilde X_s^j, \mu^\a_s, \a(s, \tilde X_s^j,
  \mu^\a_s)),\q 1\le j\le N\\ & M_s:= \prod_{j=1}^N M^j_s,\q M^j_s:=
  \exp\Big(\int_t^s \b^j_r dB^j_r - {1\over 2} \int_t^s |\b^j_r|^2
  dr\Big).  \eeaa Then, by the Girsanov theorem we have \bea
  \label{Girsanov} \dis \dbE\big[W_1(\mu^{t, \vec x, (\a, \tilde
    \a)_i}_s, \mu^\a_s)\big] &=&\dbE\big[ M_s W_1(\tilde \mu^N_s,
  \mu^\a_s)\big] = \dbE\big[ [M_s-1] W_1(\tilde \mu^N_s, \mu^\a_s)\big]
  + \dbE\big[W_1(\tilde \mu^N_s, \mu^\a_s)\big]\nonumber\\ \dis &=&
  \sum_{j=1}^N \dbE\Big[ \int_t^s M_r \b^j_r dB^j_r~ W_1(\tilde \mu^N_s,
  \mu^\a_s)\Big]+ \dbE\big[W_1(\tilde \mu^N_s, \mu^\a_s)\big].  \eea By
  the martingale representation theorem, we have \bea
  \label{mrt} W_1(\tilde \mu^N_s, \mu^\a_s) = \dbE\big[ W_1(\tilde
  \mu^N_s, \mu^\a_s)\big] + \sum_{j=1}^N \int_t^s Z^j_r dB^j_r.  \eea
  Note that $\tilde X^j$ are independent. Consider the following
  linear PDE on $[t, s]\times \dbR^{dN}$:
  \begin{equation}
    \begin{aligned}
      \label{HJB3}
      &\pa_r u(r, \vec x') + {1\over
        2}\sum_{j=1}^N \tr\big(\pa_{x_jx_j} u(r, \vec x')\big) + \sum_{j=1}^N
      b(r, x_j', \mu^\a_s, \a(r, x_j', \mu^\a_r)) \cd \pa_{x_j} u(r, \vec
      x') =0,
      \\&\hspace{10em}
      \dis u(s, \vec x') = W_1(\mu^N_{\vec x'}, \mu^\a_s). 
    \end{aligned}
  \end{equation}
  By standard BSDE theory, see e.g. \cite[Chapter 5]{Zhang}, we have
  $Z^j_r = \pa_{x_j} u(r, \vec X^{t, \vec x}_r)$, where
  $X^{t, \vec x, j}_r := x_j+B^j_r-B^j_t$. Note that the terminal
  condition $u(s, \vec x')$ is Lipschitz continuous in $x_j'$ with
  Lipschitz constant ${1\over N}$. Then, similarly to \reff{HJB2}, by
  Lemma \ref{lem-vreg} we see that
  $|Z^j|\le |\pa_{x_j}u|\le {C\over N}$ for some constant $C$
  independent of $\a$ and $L$. Thus, by \reff{Girsanov} and
  \reff{mrt}, \beaa \dbE\Big[W_1(\mu^{t, \vec x, (\a, \tilde \a)_i}_s,
  \mu^\a_s) - W_1(\tilde \mu^N_s, \mu^\a_s)\Big] =\sum_{j=1}^N
  \dbE\Big[ \int_t^s\!\! M_r \b^j_r \cd Z^j_r dr\Big] \le {C\over
    N}\sum_{j=1}^N \dbE\Big[ \int_t^s M_r |\b^j_r| dr\Big].  \eeaa
  Note that $|\b^i|\le C$ and, for $j\neq i$,
  $|\b^j_r|\le C_L W_1(\tilde \mu^N_r, \mu^\a_r)$. Then, by
  \reff{contmuConv2}, \beaa &&\dis \dbE\Big[W_1(\mu^{t, \vec x, (\a,
    \tilde \a)_i}_s, \mu^\a_s)\Big] \le \dbE\big[W_1(\tilde \mu^N_s,
  \mu^\a_s)\big] + {C\over N} \dbE\Big[ \int_t^s \!\! M_r |\b^i_r| dr
  +\sum_{j\neq i} \int_t^s \!\! M_r |\b^j_r| dr\Big]\\ &&\dis \le
  \dbE\big[W_1(\tilde \mu^N_s, \mu^\a_s)\big] + {C\over N} + {C_L\over
    N} \sum_{j\neq i} \dbE\Big[ \int_t^s M_r W_1(\tilde \mu^N_r,
  \mu^\a_r) dr\Big] = {C\over N} +C_L \th_N \le C_L\th_N.
  \eeaa

\vspace{-12mm}
\qed

\begin{thm}
  \label{thm-contCostConv} For the setting in Theorem
  \ref{thm-contmuConv}, we have
  \begin{equation}
    \label{contCostEst}
    \Big|J_i(t, \vec x, (\a,\tilde \a)_i) - J(t, \mu,
    \a; x_i, \tilde \a)\Big| + \Big|v^{N,L}_i(t, \vec x, \a) - v(\mu^\a;
    t, x_i)\Big|\le C_L\th_N^{1\over 4}.
  \end{equation}
\end{thm}
\begin{proof}
  Fix $i$. First, by taking supremum over $\tilde \a\in \cA^L_{cont}$,
  the uniform estimate for $J$ implies that for $v$ immediately. So it
  suffices to prove the former estimate.

  For this purpose, recall \reff{contXi} and denote \beaa \dis \tilde
  J_i(t, \vec x, (\a,\tilde \a)_i) := \dbE^{\dbP}\Big[g(X^{t, \vec x,
    (\a,\tilde \a)_i; i}_T, \mu^\a_T) + \int_t^T\!\! f(s, X_s^{t, \vec
    x, (\a,\tilde \a)_i; i}, \mu^\a_s, \tilde\a(s, X_s^{t, \vec x,
    (\a,\tilde \a)_i; i}, \mu^\a_s)) ds\Big].  \eeaa Then one can
  easily see that, by applying Theorem \ref{thm-contmuConv},
  \begin{equation}
    \label{contCostEst1}
    \big| J_i(t, \vec x, (\a,\tilde \a)_i)-\tilde
    J_i(t, \vec x, (\a,\tilde \a)_i)\big|\le C_L \sup_{t\le s\le T}
    \dbE\big[W_1(\mu^{t, \vec x, (\a,\tilde\a)_i}_s,\mu_s^\a)\big] \le
    C_L\th_N.
  \end{equation}

  Next, denote \beaa &X^i_s := x_i + B^i_s - B^i_t,\q \tilde
  \mu^{N,i}_s:= {1\over N} \Big[\sum_{j\neq i} \d_{X^{t, \vec x,
      (\a,\tilde \a)_i; j}_s} + \d_{X^i_s}\Big];\\ &\b_s := b(s,
  X^i_s, \mu^\a_s, \tilde \a(s, X^i_s, \mu^\a_s)),\q M_s :=
  \exp\Big(\int_t^s
  \b_r dB^i_r - {1\over 2} \int_t^s |\b_r|^2dr \Big);\\
  &\tilde \b_s := b(s, X^i_s, \tilde\mu^{N,i}_s, \tilde \a(s, X^i_s,
  \tilde \mu^{N,i}_s)),\q \tilde M_s := \exp\Big(\int_t^s \tilde \b_r
  dB^i_r - {1\over 2} \int_t^s |\tilde\b_r|^2 dr\Big).  \eeaa By
  \reff{Xmua} and \reff{contXi}, it follows from the Girsanov theorem
  again that
  \begin{equation}
    \begin{aligned}
      \label{DM} &
      \Big|\tilde J_i(t, \vec x, (\a,\tilde \a)_i) - J(t,
      \mu, \a; x_i, \tilde \a)\Big|
      \\ &\dis= \Big|\dbE\Big[\big[\tilde M_T
      - M_T\big]\big[g(X^i_T, \mu^\a_T) + \int_t^T f(s, X_s^i, \mu^\a_s,
      \tilde\a(s, X_s^i, \mu^\a_s)ds\big]\Big]\Big| \le C\dbE\big[|\tilde
      M_T - M_T|\big].
    \end{aligned}
  \end{equation}
  Denote $\D M_s:= \tilde M_s - M_s$, $\D \b_s := \tilde \b_s -
  \b_s$. Then, since $b$ is bounded, \beaa &&\dis \dbE[|\D M_s|^2] =
  \dbE\Big[\big(\int_t^s [\tilde M_r \tilde \b_r - M_r \b_r]
  dB^i_r\big)^2\Big] = \dbE\Big[\int_t^s [\tilde M_r \tilde \b_r - M_r
  \b_r]^2 dr\Big]\\ &&\dis \le C\int_t^s \dbE[|\D M_r|^2] dr + C
  \dbE\Big[\int_t^s |\tilde M_r|^2|\D \b_r|^2dr\Big] \\ &&\dis \le
  C\int_t^s \dbE[|\D M_r|^2] dr + C \dbE\Big[\int_t^s \tilde
  M_r^{3\over 2} \tilde M_r^{1\over 2}|\D \b_r|^{1\over 2} dr\Big] \\
  &&\dis \le C\int_t^s \dbE[|\D M_r|^2] dr + C \Big(\dbE\Big[\int_t^s
  \tilde M_r|\D \b_r| dr\Big]\Big)^{1\over 2}\\ &&\dis \le C\int_t^s
  \dbE[|\D M_r|^2] dr + C_L \Big(\dbE\Big[\int_t^s \tilde M_r
  W_1(\tilde \mu^{N,i}_r, \mu^\a_r) dr\Big]\Big)^{1\over 2}\\ &&\dis =
  C\int_t^s \dbE[|\D M_r|^2] dr + C_L \Big(\dbE\Big[\int_t^s
  W_1(\mu^{t, \vec x, (\a,\tilde \a)_i}_r, \mu^\a_r)
  dr\Big]\Big)^{1\over 2}\\ &&\dis \le C\int_t^s \dbE[|\D M_r|^2] dr +
  C_L\th_N^{1\over 2}, \eeaa where the last inequality thanks to
  Theorem \ref{thm-contmuConv}. Then, by the Grownwall inequality we
  obtain $ \dbE[|\D M_s|^2] \le C_L\th_N^{1\over 2}, $ and thus
  \reff{DM} implies \beaa \Big|\tilde J_i(t, \vec x, (\a,\tilde \a)_i)
  - J(t, \mu, \a; x_i, \tilde \a)\Big|\le C_L\th_N^{1\over 4}.  \eeaa
  This, together with \reff{contCostEst1}, implies the estimate for
  $J$ in \reff{contCostEst} immediately.
\end{proof}

\begin{thm}
 \label{thm-contVConv} Let Assumption \ref{assum-diffusion}
hold. Assume further that $\dis\lim_{N\to\infty}W_1(\mu^N_{\vec x},
\mu)=0$, and there exists a constant $C>0$ such that\footnote{ Note
again that $\vec x$ depends on $N$. Also, the conditions here are
slightly weaker than $\dis\lim_{N\to\infty}W_2(\mu^N_{\vec x},
\mu)=0$.} $\|\vec x\|_2 \le C$ for all $N$. Then \bea
  \label{contVConv} \bigcap_{\e > 0}\bigcup_{L\ge 0}
\limsup_{N\to\infty} \dbV^{N,\e,L}_{cont}(t,\mu^N_{\vec x}) \subset
\dbV_{cont}(t,\mu) \subset \bigcap_{\e > 0}\liminf_{N\to\infty}
\dbV^{N,\e,0}_{cont}(t,\mu^N_{\vec x}) \eea In particular, since
$\dis \liminf_{N\to\infty} \dbV^{N,\e,0}_{cont}(t,\mu^N_{\vec x})
\subset \bigcup_{L\ge 0} \limsup_{N\to\infty}
\dbV^{N,\e,L}_{cont}(t,\mu^N_{\vec x})$, actually equalities hold.
\end{thm}
\begin{proof}
  (i) We first prove the right inclusion in \reff{contVConv}. Fix
  $\f\in \dbV_{cont}(t,\mu)$, $\e>0$, and set $\e_1 := {\e\over 2}$.
  By \reff{contV} and \reff{contMeL}, there exists
  $\a^* \in \cM^{\e_1}_{cont}(t,\mu)$ such that \beaa \int_{\dbR^d}
  \!\! \big[J(t, \mu,\a^*; x, \a^*) - v(\mu^{\a^*}; t, x)\big] \mu(dx)
  \le \e_1, ~ \int_{\dbR^d}\!\! \big|\f(x)-J(t, \mu,\a^*; x,
  \a^*)\big| \mu(dx)\le \e_1.  \eeaa Recall Lemma \ref{lem-vreg} and
  note that $\f\in C_{Lip}(\dbR^d)$, then by \reff{contW1} we have
  \beaa &\dis \int_{\dbR^d} \big[J(t, \mu,\a^*; x, \a^*) -
  v(\mu^{\a^*}; t, x)\big] \mu^N_{\vec x}(dx) \le \e_1 + C
  W_1(\mu^N_{\vec x}, \mu), \\ &\dis \int_{\dbR^d} \big|\f(x)-J(t,
  \mu,\a^*; x, \a^*)\big|\mu^N_{\vec x}(dx)\le \e_1 + C_\f
  W_1(\mu^N_{\vec x}, \mu), \eeaa where $C_\f$ may depend on the
  Lipschitz constant of $\f$. Moreover, by \reff{contCostEst} we have
  \beaa &&\dis {1\over N}\sum_{i=1}^N \big[J_i(t,\vec x, \a^*) -
  v^{N,L}_i(t, \vec x, \a^*)\big] \le {1\over N}\sum_{i=1}^N \big[J(t,
  \mu,\a^*; x_i, \a^*) - v(\mu^{\a^*}; t, x_i)\big] + C_L\th_N^{1\over
    4}\\ &&\dis\q = \int_{\dbR^d} \big[J(t, \mu,\a^*; x, \a^*) -
  v(\mu^{\a^*}; t, x)\big] \mu^N_{\vec x}(dx) + C_L\th_N^{1\over 4}
  \le \e_1 + C_L\th_N^{1\over 4};\\ &&\dis {1\over N}\sum_{i=1}^N
  |\f(x_i)-J_i(t,\vec x, \a^*) | \le {1\over N}\sum_{i=1}^N
  |\f(x_i)-J(t, \mu,\a^*; x_i, \a^*)| + C_L\th_N^{1\over 4}\\ &&\dis
  \q = \int_{\dbR^d} \big|\f(x)-J(t, \mu,\a^*; x,
  \a^*)\big|\mu^N_{\vec x}(dx) + C_L\th_N^{1\over 4} \le \e_1 +
  C_{L,\f} \th_N^{1\over 4}.  \eeaa We emphasize again that
  $\|\vec x\|_2 \le C$ is independent of $N$. Then, by choosing $N$
  large enough such that $C_L\th_N^{1\over 4} \le \e_1$,
  $C_{L,\f} \th_N^{1\over 4}\le \e_1$, we obtain \beaa {1\over
    N}\sum_{i=1}^N \big[J_i(t,\vec x, \a^*) - v^{N,L}_i(t, \vec x,
  \a^*)\big] \le \e;\q {1\over N}\sum_{i=1}^N |\f(x_i)-J_i(t,\vec x,
  \a^*) | \le \e.  \eeaa This implies that
  $\a^* \in \cM^{N,\e,0}_{cont}(t,\vec x)$ and
  $\f\in \dbV^{N,\e,0}_{cont}(t,\mu^N_{\vec x})$, for all $N$ large
  enough.  That is,
  $\f\in \liminf_{N\to\infty}\dbV^{N,\e,0}_{cont}(t,\vec x)$ for any
  $\e>0$.

  (ii) We next show the left inclusion in \reff{contVConv}. Fix
  $\dis \f\in \bigcap_{\e>0}\bigcup_{L\ge 0}\limsup_{N\to \infty}
  \dbV^{N,\e,L}_{cont}(t,\mu^N_{\vec x})$, $\e>0$, and set
  $\e_1:= {\e\over 2}$. There exist $L_\e\ge 0$ and an infinite
  sequence $\{N_k\}_{k\ge 1}$ such that
  $\f\in \dbV^{N_k,\e_1,L_\e}_{cont}(t,\mu^N_{\vec x})$ for all
  $k\ge 1$. Recall \reff{contNcM} and \reff{contVN}, there exists
  $\a^k\in\cA^{L_\e}_{cont}$ such that \beaa {1\over
    N_k}\sum_{i=1}^{N_k} \big[J_i(t,\vec x, \a^k) - v^{N_k,L_\e}_i(t,
  \vec x, \a^k)\big] \le \e_1;\q {1\over N_k}\sum_{i=1}^{N_k}
  |\f(x_i)-J_i(t,\vec x, \a^k) | \le \e_1.  \eeaa Note that $L_\e$ is
  fixed, in particular it is independent of $k$. In light of Remark
  \ref{rem-cAL} (i) and denote
  $\tilde \a^k(s, x) := \a^k(s, x, \mu^{\a^k})$, then
  $\mu^{\tilde \a^k}=\mu^{\a^k}$. Similarly to (i), by
  \reff{contCostEst} we have \beaa &\dis \int_{\dbR^d} \big[J(t, \mu,
  \tilde \a^k; x, \tilde \a^k) - v(\mu^{\a^k}; t, x)\big]
  \mu^{N_k}_{\vec x}(dx) \le \e_1 + C_{L_\e} \th_{N_k}^{1\over 4}, \\
  &\dis \int_{\dbR^d} \big|\f(x)-J(t, \mu, \tilde\a^k; x,
  \tilde\a^k)\big|\mu^N_{\vec x}(dx)\le \e_1 +
  C_{L_\e}\th_{N_k}^{1\over 4}.  \eeaa Then, by Lemma \ref{lem-vreg}
  and \reff{contW1} we have \beaa &\dis \int_{\dbR^d} \big[J(t, \mu,
  \tilde \a^k; x, \tilde \a^k) - v(\mu^{\a^k}; t, x)\big] \mu(dx) \le
  \e_1 + C_{L_\e} \th_{N_k}^{1\over 4} + CW_1( \mu^{N_k}_{\vec x},
  \mu), \\ &\dis \int_{\dbR^d} \big|\f(x)-J(t, \mu, \tilde \a^k; x,
  \tilde \a^k)\big|\mu(dx)\le \e_1 + C_{L_\e}\th_{N_k}^{1\over 4} +
  C_\f W_1( \mu^{N_k}_{\vec x}, \mu).  \eeaa Now choose $k$ large
  enough (possibly depending on $\e$ and $\f$) such that \beaa
  C_{L_\e} \th_{N_k}^{1\over 4} + CW_1( \mu^{N_k}_{\vec x}, \mu) \le
  \e_1,\q C_{L_\e}\th_{N_k}^{1\over 4} + C_\f W_1( \mu^{N_k}_{\vec x},
  \mu)\le \e_1.  \eeaa Then we have \beaa \int_{\dbR^d} \!\! \big[J(t,
  \mu, \tilde \a^k; x, \tilde \a^k) - v(\mu^{\a^k}; t, x)\big] \mu(dx)
  \le \e, ~ \int_{\dbR^d} \!\! \big|\f(x)-J(t, \mu,\tilde \a^k; x,
  \tilde \a^k)\big|\mu(dx)\le \e.  \eeaa This implies that
  $\tilde\a^k \in \cM^{\e}_{cont}(t,\mu)$ and
  $\f\in \dbV^{\e}_{cont}(t,\mu)$. Since $\e>0$ is arbitrary, we
  obtain $\f\in \dbV_{cont}(t,\mu)$.
\end{proof}

\section{Appendix}
\label{sect-Appendix} 
\setcounter{equation}{0}

  \subsection{Some examples}

In this subsection we first construct  an example in discrete setting such that $\dbV_0 \subset \dbV_{state} \subset \dbV_{path}\subset\dbV_{relax}$ with all the inclusions strict, where $\dbV_{path}$ are defined in an obvious way. In particular, $\dbV_0$ is empty.

\begin{eg}
  \label{eg-statepath} Set $T=2$, $\dbS = \{\ul x, \ol x\}$,
  $\dbA = ({1\over 3}, {2\over 3})$, and 
  \beaa 
  &\dis q(0, x, \mu, a; \ul x) =q(0, x, \mu, a; \ol x)\equiv {1\over 2},\q q(1,
  x, \mu, a; \ul x) = a,\q q(1, x, \mu, a; \ol x) = 1-a;\\ 
  &\dis F(0,x,\mu, a) = 0,\q F(1,x,\mu, a) = F_1(a):= a[1-a],\q G(x, \mu) = \mu(\ul
  x).  
  \eeaa 
  Then for any $\mu \in \cP_0(\dbS)$, we have $\dbV = \{(y, y): y\in \hat \dbV\}$ for $\dbV = \dbV_0, \dbV_{state}, \dbV_{path},\dbV_{relax}$, and
  \bea
  \label{eghatV}
  \left.\ba{c}
\dis \hat \dbV_0(0,\mu) = \emptyset,\qq \hat\dbV_{state}(0, \mu) = \Big\{{5\over 9}, {13\over 18},{8\over 9}\Big\}, \ms \\
\dis  \hat\dbV_{path}(0, \mu):= \bigg\{\ul \l ~\!\mu(\ul x) + \ol \l ~\!\mu(\ol x) + {2\over 9}:\q  \ul \l, \ol \l \in \Big\{ {1\over 3}, {1\over 2}, {2\over 3}\Big\}\bigg\}, \ms \\
\dis\hat\dbV_{relax}(0, \mu):= \bigg\{\ul \l ~\!\mu(\ul x) + \ol \l~\! \mu(\ol x) + {2\over 9}:\q  \ul \l, \ol \l \in \Big[{1\over 3},  {2\over 3}\Big]\bigg\}
  \ea\right.
    \eea
\end{eg}

\proof Since $|\dbS|=2$, for any $\mu\in \cP_0(\dbS)$ clearly it suffices
  to specify $\mu(\ul x)$.

(i) We first compute $\dbV_0(0,\mu)$. For any $\a, \tilde \a\in \cA_{state}$, it is straightforward to
  compute: 
  \bea
  \label{mua}
  \left.\ba{lll}
 \dis \mu^\a_1(\ul x) = \sum_{x_0\in \dbS} \mu(x_0)
  q(0, x_0, \mu, \a(0, x_0); \ul x) = \sum_{x_0\in \dbS} \mu(x_0) {1\over 2} = {1\over 2};\\
  \dis \mu^\a_2(\ul x) = \sum_{x_1\in \dbS} \mu^\a_1(x_1) q(1, x_1,
  \mu^\a_1, \a(1, x_1); \ul x) = {1\over 2} \sum_{x_1\in \dbS}\a(1, x_1);\\ 
  \dis \dbP^{\mu^\a; 0, x_0, \tilde \a}(X_1 = \ul x) = q(0, x_0, \mu, \tilde \a(0, x_0); \ul x) = {1\over 2}.
  \ea\right.
  \eea
  Then
  \bea
  \label{J0}
 \dis J(0, \mu, \a; x_0, \tilde \a) &=& \dbE^{\dbP^{\mu^\a; 0, x_0, \tilde \a}}\Big[G(X_2, \mu^\a_2) + \sum_{t=0,1} F(t, X_t, \mu^\a_t, \tilde\a(t, X_t))\Big]\nonumber\\
  &=& \mu^\a_2(\ul x) + \dbE^{\dbP^{\mu^\a; 0, x_0, \tilde
      \a}}\Big[ F_1(\tilde\a(1, X_1))\Big]\nonumber\\
  &=& {1\over 2} \sum_{x_1\in \dbS}\a(1, x_1) + {1\over
    2} \sum_{x_1\in \dbS} F_1( \tilde\a(1, x_1)).
    \eea
  Given $\a$, we see that $\inf_{\tilde \a}  J(0, \mu, \a; x_0, \tilde \a)={1\over 2} \sum_{x_1\in \dbS}\a(1, x_1) + {2\over 9}$, and the minimum is achieved when $\tilde \a(1, x_1) = {1\over 3}, {2\over 3}$, $\forall x_1\in \dbS$, which are  not included in $\dbA$. Thus $\cM_{state}(0,\mu) = \emptyset$, and hence $\dbV_0(0,\mu) = \emptyset$.
  
  (ii) We next compute $\dbV_{state}(0, \mu)$. Fix $\e>0$ small. By \reff{MFEe} and \reff{J0} it is clear that 
  \bea
  \label{ae0}
 \a^\e\in  \cM_{state}^\e(0,\mu) \q\mbox{ if and only if}\q {1\over 2}\sum_{x_1\in \dbS} F_1(\a^\e(1, x_1))\le {2\over 9} +\e.
  \eea
 and in this case,  for any $x_0\in \dbS$, by \reff{J0} again we have
 \beaa
  J(0, \mu, \a^\e; x_0, \a^\e) = J_0(\a^\e):= {1\over 2}\sum_{x_1\in \dbS}\tilde F_1(\a^\e(1, x_1)),~\mbox{where}~ \tilde F_1(a) := a+F_1(a)=a [2- a].
  \eeaa
  In particular, this implies that $ \dbV^\e_{state}(0, \mu) = \Big\{(y, y): y\in  \hat\dbV^\e_{state}(0, \mu)\Big\}$ where
  \beaa
   \hat\dbV^\e_{state}(0, \mu):= \Big\{J_0(\a^\e): \a^\e\in  \cM_{state}^\e(0,\mu)\Big\}.
  \eeaa

 Recall again that $\dis\inf_{a\in \dbA} F_1(a) = {2\over 9}$. By \reff{ae0},  $\a^\e\in  \cM_{state}^\e(0,\mu)$ if and only if there exists a function $\chi_\e: \dbS\to \dbR$ such that $F_1(\a^\e(1, x_1))= {2\over 9} +\chi_\e(x_1)$ for all $x_1\in \dbS$, and
 \bea
 \label{ae1}
  \chi_\e(\ul x), \chi_\e(\ol x)>0,\q \chi_\e(\ul x)+\chi_\e(\ol x)\le 2\e.
  \eea
  This implies that
  \beaa
  \a^\e(1, x_1) = {1\over 3} + \hat\chi_\e(x_1) ~\mbox{or}~ {2\over 3} - \hat\chi_\e(x_1),\q\mbox{where}\q \hat \chi_\e(x_1) := {6\chi_\e(x_1)\over 1+\sqrt{1  - 36\chi_\e(x_1)}}.
  \eeaa
  Note that $\tilde F_1$ is strictly increasing for $a\in \dbA$. Then, by \reff{Vtmu} we have, for $\e>0$ small,
  \beaa
   &\dis \hat\dbV^\e_{state}(0, \mu) = \bigcup_{\chi_\e} \bigcup_{i=1}^4 (y_i-\e, y_i+\e),\\
   &\dis y_1 := {1\over 2}\Big[\tilde F_1\big({1\over 3} + \chi_\e(\ul x)\big) + \tilde F_1\big({1\over 3} + \chi_\e(\ol x)\big)\Big],~ y_2 := {1\over 2}\Big[\tilde F_1\big({1\over 3} + \chi_\e(\ul x)\big) + \tilde F_1\big({2\over 3} - \chi_\e(\ol x)\big)\Big],\\
   &\dis y_3 := {1\over 2}\Big[\tilde F_1\big({2\over 3} - \chi_\e(\ul x)\big) + \tilde F_1\big({1\over 3} + \chi_\e(\ol x)\big)\Big],~ y_4 := {1\over 2}\Big[\tilde F_1\big({2\over 3} - \chi_\e(\ul x)\big) + \tilde F_1\big({2\over 3} - \chi_\e(\ol x)\big)\Big],
   \eeaa
  where the first union is over all $\chi_\e$  satisfying \reff{ae1}. Note that $0< \chi_\e(\ul x), \chi_\e(\ol x) < 2\e$. Then by \reff{Vtmu} it is obvious that $\dbV_{state}(0, \mu) = \Big\{(y, y): y\in  \hat\dbV_{state}(0, \mu)\Big\}$ and
  \beaa
 \hat\dbV_{state}(0, \mu) = \Big\{\tilde F_1\big({1\over 3}\big),~ {1\over 2}\big[\tilde F_1\big({1\over 3}\big)+\tilde F_1\big({2\over 3}\big)\big], ~\tilde F_1\big({2\over 3}\big)\Big\} =\Big\{{5\over 9}, {13\over 18}, {8\over 9}\Big\}.
   \eeaa

  (iii) We now compute $\dbV_{path}(0, \mu)$. For any
  $\a, \tilde \a\in \cA_{path}$, we still have
  $\mu^\a_1(\ul x) = {1\over 2}$ and
  $\dbP^{\mu^\a; 0, x_0, \tilde \a}(X_1 = \ul x) ={1\over
    2}$, for all $x_0\in \dbS$. Moreover, 
    \bea
    \label{Je2}
    &&\dis \mu^\a_2(\ul x) = \sum_{x_0, x_1\in
    \dbS} \mu(x_0) q(0, x_0, \mu, \a(0, x_0); x_1) q(1, x_1, \mu^\a_1,
  \a(1, x_0, x_1); \ul x) \nonumber\\
   &&\dis \qq \q= {1\over 2} \sum_{x_0,
    x_1\in \dbS}\mu(x_0)\a(1, x_0, x_1); \nonumber\\ 
    &&\dis J(0, \mu, \a; x_0,
  \tilde \a) = \dbE^{\dbP^{\mu^\a; 0, x_0, \tilde \a}}\Big[G(X_2,
  \mu^\a_2) + F(1, X_1, \mu^\a_1, \tilde\a(1, X_0, X_1))\Big]\nonumber\\ 
  &&\dis
  \q = \mu^\a_2(\ul x) + \dbE^{\dbP^{\mu^\a; 0, x_0, \tilde \a}}\Big[F_1(
  \tilde\a(1, X_0, X_1))\Big]\nonumber\\ 
  &&\dis \q =
 \sum_{\tilde x_0\in \dbS}  \mu(\tilde x_0) \times {1\over 2}\sum_{x_1\in \dbS}\a(1, \tilde
  x_0, x_1) + {1\over 2}\sum_{x_1\in \dbS}F_1( \tilde\a(1, x_0, x_1) ).  
  \eea 
  Similarly to \reff{ae0}, 
  \beaa
  \a^\e\in \cM^\e_{path}(0,\mu) \q\mbox{if and only if}\q
  {1\over 2}\sum_{x_1\in \dbS} F_1(\a^\e(1, x_0, x_1) )\le {2\over 9}+\e,~\forall x_0\in \dbS.
  \eeaa
  Furthermore, by abusing the notation $\chi_\e$, the above is equivalent to that there exists $\chi_\e: \dbS\times \dbS\to \dbA$ such that, by denoting $\hat \chi_\e(x_0, x_1) := {6\chi_\e(x_0, x_1)\over 1+\sqrt{1  - 36\chi_\e(x_0,x_1)}}$,
 \beaa
 &\dis \chi_\e(x_0, x_1)>0, ~\forall x_0, x_1\in \dbS,\q\mbox{and}\q \chi_\e(x_0, \ul x)+\chi_\e(x_0,\ol x)\le 2\e, ~ \forall x_0\in \dbS;\\
  &\dis \a^\e(1, x_0, x_1) = {1\over 3} + \hat\chi_\e(x_0, x_1) ~~\mbox{or}~~ {2\over 3} - \hat\chi_\e(x_0, x_1). 
  \eeaa
  Following the same arguments as in (ii), we can easily see that $\dbV_{path}(0, \mu)$ consists of pairs
$\big(J(0, \mu, \a^*; \ul x, \a^*),~ J(0, \mu, \a^*; \ol x, \a^*)\big)$ for all $\a^*: \dbS^2 \to \{{1\over 3}, {2\over 3}\}$. Note that $F_1({1\over 3}) = F_1({2\over 3}) = {2\over 9}$, and ${1\over 2}\sum_{x_1\in \dbS}\a^*(1, \tilde
  x_0, x_1)$ takes $3$ possible values: ${1\over 3}, {1\over 2}, {2\over 3}$. Then by \reff{Je2} we have
\bea
\label{Je3}
J(0, \mu, \a^*; x_0,
  \a^*) = \ul \l~\! \mu(\ul x) + \ol \l ~\!\mu(\ol x) + {2\over 9},\q\mbox{where}\q \ul \l, \ol \l \in \big\{ {1\over 3}, {1\over 2}, {2\over 3}\big\}.
  \eea
 Again this is independent of $x_0$. Then  $ \dbV_{path}(0, \mu) = \Big\{(y, y): y\in  \hat\dbV_{path}(0, \mu)\Big\}$ and
\beaa
 \hat\dbV_{path}(0, \mu):= \Big\{\ul \l ~\!\mu(\ul x) + \ol \l~\! \mu(\ol x) + {2\over 9}:\q  \ul \l, \ol \l \in \big\{ {1\over 3}, {1\over 2}, {2\over 3}\big\}\Big\}.
   \eeaa
  In particular, we see that  $\hat\dbV_{state}(0, \mu)$ consists of the elements of $\hat\dbV_{path}(0, \mu)$ with $\ul \l = \ol \l$, and $\hat\dbV_{path}(0, \mu) = \hat\dbV_{state}(0, \mu)$ when $\mu(\ul x) = \mu(\ol x)$.
  
  (iv) Finally we compute $\dbV_{relax}(0,\mu)$. Fix $\g, \tilde \g\in \cA_{relax}$, it is straightforward to
  compute: 
  \beaa 
  &&\dis \mu^\g_1(\ul x) = \sum_{x_0\in \dbS} \mu(x_0)  \int_\dbA q(0, x_0,
      \mu, a; \ul x)\g(0,x_0; da)=\sum_{x_0\in \dbS} \mu(x_0)\times {1\over 2}={1\over 2};\\
 &&\dis \dbP^{\mu^\g; 0, x_0, \tilde \g}(X_1 = \ul x) =  \int_\dbA q(0, x_0,
      \mu, a; \ul x)\tilde \g(0,x_0; da)={1\over 2};\\
 &&\dis     \mu_2^\g(\ul x) = \sum_{x_0, x_1\in \dbS} \mu(x_0) \int_{\dbA^2} q(0, x_0, \mu, a_0; x_1)  q(1, x_1,
      \mu^\g_1, a_1; \ul x)\g(0, x_0; da_0)\g(1,x_0, x_1; da_1) \\
  &&\dis \qq\q    = {1\over 2} \sum_{x_0, x_1\in \dbS} \mu(x_0) \int_{\dbA}   a \g(1,x_0, x_1; da);
      \\
      &&\dis J(0, \mu, \g; x_0, \tilde \g)
      = \dbE^{\dbP^{\mu^\g; 0, x_0, \tilde \g}}\Big[
      G(X_2,\mu^\g_2) +
      \sum_{t=0,1} \int_\dbA F(t, X_t, \mu^\g_t, a) \tilde\g(t,X; da)\Big]\\
      &&\dis \qq\q = \mu^\g_2(\ul x) +
      \dbE^{\dbP^{\mu^\g; 0, x_0, \tilde \g}}
      \Big[ \int_\dbA F_1(a)\tilde\g(1, X; da)\Big]
      \\
      &&\dis \qq\q    = {1\over 2} \sum_{\tilde x_0, x_1\in \dbS} \mu(\tilde x_0) \int_{\dbA}   a \g(1,\tilde x_0, x_1; da) + {1\over 2} \sum_{x_1\in \dbS} \int_{\dbA}   F_1(a) \tilde\g(1,x_0, x_1; da). 
  \eeaa
 Similarly to \reff{ae0}, 
  \bea
  \label{ae2}
  \g^\e\in \cM^\e_{relax}(0,\mu) \q\mbox{if and only if}\q
  {1\over 2} \sum_{x_1\in \dbS} \int_{\dbA}   F_1(a) \g^\e(1,x_0, x_1; da)\le {2\over 9}+\e,~\forall x_0\in \dbS,
  \eea
  and in this case, for any $x_0\in \dbS$,
  \bea
  \label{Je4}
  J(0, \mu, \g^\e; x_0, \g^\e)
     = {1\over 2} \sum_{\tilde x_0, x_1\in \dbS} \mu(\tilde x_0) \int_{\dbA}   a \g^\e(1,\tilde x_0, x_1; da) + {1\over 2} \sum_{x_1\in \dbS} \int_{\dbA}   F_1(a) \g^\e(1,x_0, x_1; da). 
      \eea
      
  Let $\hat\cM_{relax}$ denote the set of $\g^*: \dbS^2 \to \cP(\{{1\over 3}, {2\over 3}\})$ and set
  \bea
  \label{Je5}
  \hat J(\g^*):=  {1\over 2}\sum_{x_0, x_1\in \dbS} \mu(x_0) \Big[{1\over 3} \g^*(x_0, x_1; {1\over 3}) + {2\over 3} \g^*(x_0, x_1; {2\over 3})\Big] + {2\over 9}. 
  \eea
We claim that, for any $\g^\e\in \cM^\e_{relax}(0,\mu)$, there exists $\hat \g^\e\in \hat\cM_{relax}$ such that 
\bea
\label{Je6}
\Big|J(0, \mu, \g^\e; x_0, \g^\e) - \hat J(\hat\g^\e)\Big| \le C\sqrt{\e}.
\eea
On the other hand, for any $\g^*\in \hat\cM_{relax}$, denote 
\bea
\label{Ae123}
A^\e_1:= ({1\over 3}, {1\over 3}+\sqrt{\e}],\q A^\e_2:= [{2\over 3}-\sqrt{\e}, {2\over 3}),\q A^\e_3:= \dbA \backslash (A^\e_1\cup A^\e_2),
\eea
and set $\g^\e\in \cA_{relax}$ such that 
\beaa
\g^\e(1, x_0, x_1; da) :={1\over 2\sqrt{\e}}\Big[ \g^*(x_0, x_1; {1\over 3}) \1_{A^\e_1}(a) + \g^*(x_0, x_1; {2\over 3})\1_{A^\e_2}(a)\Big]da.
\eeaa
Note that $F_1(a) \le ({1\over 3}+\sqrt{\e})({2\over 3}-\sqrt{\e})= {2\over 9} + {\sqrt{\e}\over 3}-\e$, $\g^\e(1, x_0, x_1;da)$-a.s. Then it is clear that $\g^\e\in \cM^{{\sqrt{\e}\over 3}-\e}_{relax}$. Moreover, one can easily verify that
\beaa
&&\dis \Big|J(0, \mu, \g^\e; x_0, \g^\e) - \hat J(\hat\g^\e)\Big| \\
&&\dis \le \sum_{i=1}^2 {1\over 2} \sum_{\tilde x_0, x_1\in \dbS} \mu(\tilde x_0) \g^*(x_0, x_1; {i\over 3}) \Big|{1\over \sqrt{\e}}\int_{A^\e_i} ada - {i\over 3}\Big| +{\sqrt{\e}\over 3}-\e \le C\sqrt{\e}.
\eeaa
This, together with \reff{Je6} and \reff{rVtmu}, implies that    $\dbV_{relax}(0, \mu) = \Big\{(y, y): y\in  \hat\dbV_{relax}(0, \mu)\Big\}$ and, by denoting $\ul \l:= {1\over 2}\sum_{x_1\in \dbS} \Big[{1\over 3} \g^*(\ul x, x_1; {1\over 3}) + {2\over 3} \g^*(\ul x, x_1; {2\over 3})\Big]\in  [{1\over 3}, {2\over 3}]$ and similarly for $\ol \l$,
\beaa
   \hat\dbV_{relax}(0, \mu):= \Big\{\hat J(\g^*): \g^*\in  \hat\cM_{relax}\Big\}=\Big\{\ul \l~\! \mu(\ul x) + \ol \l ~\!\mu(\ol x) + {2\over 9}:\q  \ul \l, \ol \l \in [{1\over 3}, {2\over 3}]\Big\}.
   \eeaa

 It remains to prove \reff{Je6}. Let $\g^\e$ satisfies \reff{ae2}. 
Then, for any $x_0\in \dbS$, we have 
\beaa
\e \!\!\!&\ge&\!\!\! {1\over 2} \sum_{x_1\in \dbS} \int_{\dbA}   F_1(a) \g^\e(1,x_0, x_1; da)-{2\over 9} = {1\over 2} \sum_{x_1\in \dbS} \int_{\dbA}   (a-{1\over 3})({2\over 3}-a) \g^\e(1,x_0, x_1; da)\\
\!\!\!&\ge&\!\!\!  {1\over 2} \sum_{x_1\in \dbS} \int_{A^\e_3}   (a-{1\over 3})({2\over 3}-a) \g^\e(1,x_0, x_1; da) \ge \sqrt{\e}({1\over 3}-\sqrt{\e}){1\over 2} \sum_{x_1\in \dbS} \int_{A^\e_3}  \g^\e(1,x_0, x_1; da).
  \eeaa
Thus
  \beaa
 \int_{A^\e_3}  \g^\e(1,x_0, x_1; da) \le C\sqrt{\e},\q\forall x_0, x_1\in \dbS.
   \eeaa 
   Recall \reff{Ae123} and set $\hat \g^\e\in \hat \cM_{relax}$ by: 
   \beaa
  \hat \g^\e(x_0, x_1; {1\over 3}) := {\g^\e(1,x_0, x_1; A^\e_1)\over \sum_{i=1}^2\g^\e(1,x_0, x_1; A^\e_i)},\q \hat \g^\e(x_0, x_1; {2\over 3}) := {\g^\e(1,x_0, x_1; A^\e_2)\over \sum_{i=1}^2\g^\e(1,x_0, x_1; A^\e_i)}.
  \eeaa  
Then $F_1(a) = {2\over 9}$, $\hat \g^\e(x_0, x_1;da)$-a.s., and thus
\beaa
&&\dis \Big|J(0, \mu, \g^\e; x_0, \g^\e) - \hat J(\hat\g^\e)\Big| \\
&&\dis \le \sum_{i=1}^2 {1\over 2} \sum_{\tilde x_0, x_1\in \dbS} \mu(\tilde x_0)\Big| \int_{\dbA^\e_i}   a \g^\e(1,\tilde x_0, x_1; da)-{i\over 3} \hat\g^\e(\tilde x_0, x_1; A^\e_i)\Big| \\
&&+{1\over 2} \sum_{\tilde x_0, x_1\in \dbS} \mu(\tilde x_0) \int_{\dbA^\e_3}   a \g^\e(1,\tilde x_0, x_1; da)+\Big| {1\over 2} \sum_{x_1\in \dbS} \int_{\dbA}   F_1(a) \g^\e(1,x_0, x_1; da) - {2\over 9}\Big|\\
&&\dis \le C\sum_{i=1}^2\big|\g^\e(1,\tilde x_0, x_1; A^\e_i) - \hat\g^\e(\tilde x_0, x_1; A^\e_i)\big|  +C\sqrt{\e} \\
&&\dis \le C{1-  \sum_{i=1}^2 \g^\e(1,\tilde x_0, x_1; A^\e_i)\over \sum_{i=1}^2 \g^\e(1,\tilde x_0, x_1; A^\e_i)} +C\sqrt{\e} \le {C\sqrt{\e}\over 1-C\sqrt{\e}} + C\sqrt{\e} \le C\sqrt{\e}.
\eeaa   
This proves \reff{Je6}.
\qed

Our next example shows that the left inclusion in \reff{FinSymConv} fails if we remove the $L$-Lipschitz continuity requirement, as mentioned in Remark \ref{rem-regularity} (ii). This justifies our uniform regularity requirement on the admissible controls in order to have the desired convergence as in Theorem \ref{thm-FinSymConv}. Recall $\dbV^{N,\e}$ and $\dbV^{N,\e, \infty}$ in Remark \ref{rem-regularity} (ii).

  \begin{eg}
    \label{eg-Lipschitz}
    Let $T, \dbS,  q$ be as in Example \ref{eg-statepath}, and 
    \beaa
   \dbA = [{1\over 3}, {2\over 3}],\q  F \equiv 0,\q G(\ul x, \mu) = \frac{20}{9}-5 \mu(\ul x),\q G(\ol x, \mu) =  \frac{20}{9}-3\mu(\ol x).
    \eeaa
    Then, for any $\mu\in \cP_0(\dbS)$ and $\mu^N_{\vec x}\in \cP_N(\dbS)$ with $\mu^N_{\vec x}\to \mu$,  $(0, 0)$ is in  $\dis \bigcap_{\e>0}\limsup_{N\to\infty}
      \dbV_{state}^{N,\e,\infty}(0,\mu^N_{\vec x})$ and $\dis\bigcap_{\e>0}\limsup_{N\to\infty}
      \dbV_{state}^{N,\e}(0,\mu^N_{\vec x})$, but not in $\dbV_{state}(0,\mu)$.
      \end{eg}

  \begin{proof} (i) We first compute $\dbV_{state}(0,\mu)$.  For $\a, \tilde \a\in \cA_{state}$ (which do not depend on $\mu$), similarly to \reff{mua} we have  
  \bea
  \label{mua2}
  \left.\ba{lll}
 \dis \mu^\a_1(\ul x)  = {1\over 2},\q  \mu^\a_2(\ul x) = {1\over 2} \sum_{x_1\in \dbS}\a(1, x_1),\q \dbP^{\mu^\a; 0, x_0, \tilde \a}(X_1 = \ul x)  = {1\over 2},\\
  \dis \dbP^{\mu^\a; 0, x_0, \tilde \a}(X_2 = \ul x)  = \sum_{x_1\in \dbS} \dbP^{\mu^\a; 0, x_0, \tilde \a}(X_1 = x_1) q(1, x_1,
  \mu^\a_1, \tilde\a(1, x_1); \ul x) = {1\over 2}\sum_{x_1\in \dbS} \tilde\a(1, x_1).
    \ea\right.
  \eea
  Then
  \beaa
&&   \dis J(0, \mu, \a; x_0, \tilde \a) = \dbE^{\dbP^{\mu^\a; 0, x_0, \tilde \a}}[G(X_2, \mu^\a_2)]\\
&&\dis = \frac{20}{9} -5 \dbP^{\mu^\a; 0, x_0, \tilde \a}(X_2=\ul x) \mu^\a_2(\ul x) -3 \dbP^{\mu^\a; 0, x_0, \tilde \a}(X_2=\ol x) \mu^\a_2(\ol x)  \\
  &&\dis =\frac{20}{9}-{5\over 2}\sum_{x_1\in \dbS} \tilde\a(1, x_1) \times {1\over 2} \sum_{x_1\in \dbS}\a(1, x_1) - 3 \Big[1- {1\over 2}\sum_{x_1\in \dbS} \tilde\a(1, x_1)\Big]\Big[1- {1\over 2} \sum_{x_1\in \dbS}\a(1, x_1)\Big]\\
  &&\dis =  {1\over 2} \Big[3-4\sum_{x_1\in \dbS} \a(1, x_1)\Big] \sum_{x_1\in \dbS} \tilde\a(1, x_1) + {3\over 2} \sum_{x_1\in \dbS}\a(1, x_1) - \frac{7}{9}.
    \eeaa
    Note that, when $\sum_{x_1\in \dbS} \a(1, x_1) > {3\over 4}$, $\inf_{\tilde \a\in \cA_{state}} J(0, \mu, \a; x_0, \tilde \a)$ is achieved at $\tilde \a \equiv {2\over 3}$. Since $\sum_{x_1\in \dbS} {2\over 3} = {4\over 3} > {3\over 4}$, then $\a\equiv {2\over 3}$ is an equilibrium with
    \begin{equation*}
      J(0,\mu,{2\over 3};x_0, {2\over 3}) =  
      {1\over 2} \Big[3-4\sum_{x_1\in \dbS} {2\over 3}\Big] \sum_{x_1\in \dbS} {2\over 3} + {3\over 2} \sum_{x_1\in \dbS}{2\over 3} - \frac{7}{9}
      = - \frac{1}{3},\q \forall x_0\in \dbS.
    \end{equation*}
  Similarly, when $\sum_{x_1\in \dbS} \a(1, x_1) < {3\over 4}$, $\inf_{\tilde \a\in \cA_{state}} J(0, \mu, \a; x_0, \tilde \a)$ is achieved at $\tilde \a \equiv {1\over 3}$. Since $\sum_{x_1\in \dbS} {1\over 3} = {2\over 3} < {3\over 4}$, then $\a\equiv {1\over 3}$ is also an equilibrium with
      \begin{equation*}
      J(0,\mu, {1\over 3};x_0, {1\over 3}) =  {1\over 2} \Big[3-4\sum_{x_1\in \dbS} {1\over 3}\Big] \sum_{x_1\in \dbS} {1\over 3} + {3\over 2} \sum_{x_1\in \dbS}{1\over 3} - \frac{7}{9} =  \frac{1}{3},\q \forall x_0\in \dbS.
    \end{equation*}
    Moreover, when $\sum_{x_1\in \dbS} \a(1, x_1) = {3\over 4}$, then all $\tilde \a$, including $\tilde \a=\a$, are minimizers of $J$, and thus such $\a$ is an equilibrium. In this case 
\beaa
J(0, \mu, \a; x_0, \a)= {3\over 2} \sum_{x_1\in \dbS}\a(1, x_1) - {7\over 9} = {3\over 2} \times {3\over 4}- {7\over 9} = \frac{25}{72},\q \forall x_0\in \dbS.
\eeaa
Put all cases together, we have $\dbV_{state}(0,\mu) = \big\{(-{1\over 3}, -{1\over 3}),({1\over 3}, {1\over 3}),({25\over 72}, {25\over 72})\big\}$.

\ms
(ii) We next show that $(0, 0)\in\bigcap_{\e>0}\limsup_{N\to\infty} \dbV_{state}^{N,\e}(0,\mu^N_{\vec x})$. 
Set
\beaa
\a(t,x, \mu) := \a(\mu) := {1\over 3} \1_{\{\mu(\ul x)\le {1\over 2}\}} + {2\over 3} \1_{\{\mu(\ul x)> {1\over 2}\}},~ E^N_1:= \{\mu^N_1(\ul x) \le {1\over 2}\},~ E^N_2:= \{\mu^N_1(\ul x) > {1\over 2}\},
\eeaa
where $\a$ does not depend on $(t,x)$. Then,  for any $\tilde \a: \dbT \times \dbS \times \cP(\dbS) \to \dbA$, recalling the setting in Subsection \ref{sect-stateN} and denoting $\dbP^i := \dbP^{0,\vec x, (\a, \tilde \a)_i}$, we have
  \begin{equation*}
      \begin{aligned}
        &J_i(0,\vec x,(\a, \tilde\a)_i)
            = \dbE^{\dbP^i}\big[G(X_2^i,\mu_2^N)\big] = \frac{20}{9} -\dbE^{\dbP^i}  \big[5 \mu^N_2(\ul x) \1_{\{X^i_2 = \ul x\}} + 3  \mu^N_2(\ol x) \1_{\{X^i_2 = \ol x\}}\big] \\
            &=\frac{20}{9}-{1\over N} \dbE^{\dbP^i}  \big[5  \1_{\{X^i_2 = \ul x\}} + 3  \1_{\{X^i_2 = \ol x\}}\big] - {1\over N} \sum_{j\neq i} \dbE^{\dbP^i}  \big[5  \1_{\{X^j_2=X^i_2 = \ul x\}} + 3  \1_{\{X^j_2=X^i_2 = \ol x\}}\big]\\
             &= \frac{20}{9}-{1\over N}\sum_{j\neq i}\dbE^{\dbP^i} \Big[5 \a(\mu^N_1)\tilde \a(1, X^i_1, \mu^N_1) + 3 [1-\a(\mu^N_1)][1-\tilde \a(1, X^i_1, \mu^N_1)\big]\Big]+ O\big({1\over N}\big)\\
            &= \frac{20}{9} - \dbE^{\dbP^i} \Big[ \big[2-{1\over 3} \tilde \a(1, X^i_1, \mu^N_1)\big]\1_{E^N_1} + \big[1+{7\over 3} \tilde \a(1, X^i_1, \mu^N_1)\big]\1_{E^N_2}\Big]+ O\big({1\over N}\big).
            \end{aligned}
    \end{equation*}
    Notice that, under each $\dbP^i$, $X^1_1, \cds, X^N_1$ are i.i.d. with $\dbP^i(X^j_1=\ul x) = \dbP^i(X^j_1=\ol x) = {1\over 2}$. Thus we may use a common $\bar\dbP$, under which $\vec X_1$ has the above distribution, such that
    \begin{equation}
      \begin{aligned}
        \label{eq-repcost}
        J_i(0,\vec x,(\a, \tilde\a)_i) =\frac{20}{9}- \dbE^{\bar\dbP} \Big[ \big[2-{1\over 3} \tilde \a(1, X^i_1, \mu^N_1)\big]\1_{E^N_1} + \big[1+{7\over 3} \tilde \a(1, X^i_1, \mu^N_1)\big]\1_{E^N_2}\Big]+ O\big({1\over N}\big).
      \end{aligned}
    \end{equation}
 If we ignore the term $ O\big({1\over N}\big)$, clearly $\tilde \a =\a$ is the minimizer of the above $J_i$. Then for fixed $\e>0$ and for $N$ large enough, $\a$ is an $\e$-minimizer for all $i$, and thus $\a$ is an $\e$-equilibrium. Note that  $N \mu^N_1(\ul x)= \sum_{i=1}^N \1_{\{X^i_1=\ul x\}}$ has distribution Binomial($N, {1\over 2}$) under $\bar \dbP$. Then $\bar\dbP(E^N_1) = {1\over 2}$ when $N$ is odd, and 
  \beaa
  {1\over 2} \le \bar\dbP(E^N_1) \le {1\over 2} + \bar\dbP(N \mu^N_1(\ul x) = {N\over 2}) = {1\over 2} + {1\over 2^N}\left(\ba{lll} N\\ {N\over 2}\ea\right) = {1\over 2} + O\big({1\over \sqrt{N}}\big),
  \eeaa
  when $N$ is even. Thus
  \beaa
  J_i(0,\vec x,\a)= \frac{20}{9}-{17\over 9}\bar\dbP(E^N_1) -{23\over 9} \bar\dbP(E^N_2)+ O\big({1\over N}\big) =\frac{20}{9}-{1\over 2}[{17\over 9}+{23\over 9} ]+ O\big({1\over \sqrt{N}}\big) = O\big({1\over \sqrt{N}}\big).
  \eeaa
  Since $\mu^N_{\vec x}\to \mu\in \cP_0(\dbS)$, we have $\mu^N_{\vec x}\in \cP_0(\dbS)$ for $N$ large enough. Then, in light of \reff{JN}, 
  \beaa
  J_N(0, x_0, \mu^N_{\vec x}, \a) =  O\big({1\over \sqrt{N}}\big),\q\forall x_0\in \dbS.
  \eeaa
  This implies that $(0,0)\in\bigcap_{\e>0}\limsup_{N\to\infty} \dbV_{state}^{N,\e}(0,\mu^N_{\vec x})$.
         
(iii) We finally show that $(0,0)\in\bigcap_{\e>0}\limsup_{N\to\infty} \dbV_{state}^{N,\e, \infty}(0,\mu^N_{\vec x})$. Set
\beaa
&\dis \a^N(t,x, \mu) := {1\over 3} \1_{\{\mu(\ul x)\le p_N \}} + {2\over 3} \1_{\{\mu(\ul x)\ge  q_N\}} + \big[{1\over 3} + {N\over 3}\big(\mu(\ul x) -  p_N\big)\big]\1_{\{p_N < \mu(\ul x) < q_N\}},\\
&\dis \mbox{where}\q  p_N :={1\over 2}- {1\over 2N},\q q_N :={1\over 2}+ {1\over 2N}\\ 
&\dis \tilde E^N_1:= \big\{\mu^N_1(\ul x) \le  p_N\big\},\q \tilde E^N_2:= \big\{\mu^N_1(\ul x) \ge q_N\big\},\q \tilde E^N_3:= \big\{p_N<\mu^N_1(\ul x) < q_N\big\}.
\eeaa
Then clearly $\a^N\in \cA^\infty_{state}$. For any $\tilde \a\in  \cA^\infty_{state}$, similarly to \reff{eq-repcost} we have
\beaa
      \begin{aligned}
        &J_i(0,\vec x,(\a^N, \tilde\a)_i)
           = \frac{20}{9} - \dbE^{\bar\dbP} \Big[ \big[2-{1\over 3} \tilde \a(1, X^i_1, \mu^N_1)\big]\1_{\tilde E^N_1} + \big[1+{7\over 3} \tilde \a(1, X^i_1, \mu^N_1)\big]\1_{\tilde E^N_2}\\
            &\qq - \big[5 \a(\mu^N_1)\tilde \a(1, X^i_1, \mu^N_1) + 3 [1-\a(\mu^N_1)][1-\tilde \a(1, X^i_1, \mu^N_1)]\big]\1_{\tilde E^N_3}\Big]+ O\big({1\over N}\big).
        \end{aligned}
    \eeaa
 Again, fix $\e>0$ and consider $N$ large enough. On $\tilde E^N_1 \cup \tilde E^N_2$, it is optimal to choose $\tilde \a = \a^N$, up to the error $O\big({1\over N}\big)$. Then
\beaa
J_i(0,\vec x,(\a^N, \tilde\a)_i) - J_i(0,\vec x,\a^N) \le C \bar\dbP (\tilde E^N_3)+ O\big({1\over N}\big),
\eeaa
 When $N$ is odd, $\tilde E^N_3 = \emptyset$ and thus $\bar\dbP (\tilde E^N_3)=0$. When $N$ is even,
 \beaa
 \bar\dbP (\tilde E^N_3) = \bar \dbP(\mu^N_1(\ul x) = {1\over 2}) = {1\over 2^N}\left(\ba{lll} N\\ {N\over 2}\ea\right) = O\big({1\over \sqrt{N}}\big).
 \eeaa
  So in both cases, we have
    \beaa
J_i(0,\vec x,(\a^N, \tilde\a)_i) - J_i(0,\vec x,\a^N) \le  O\big({1\over \sqrt{N}}\big),
\eeaa 
That is, $\a^N\in \cM^{N,\e,\infty}_{state}(0, \mu^N_{\vec x})$ for $N$ large enough. Thus $J_N(0, \cd, \mu^N_{\vec x}, \a^N)   \in \dbV_{state}^{N,\e, \infty}(0,\mu^N_{\vec x})$. Then by similar arguments as in (ii) we see that $(0,0)\in\bigcap_{\e>0}\limsup_{N\to\infty} \dbV_{state}^{N,\e, \infty}(0,\mu^N_{\vec x})$.
 \end{proof}

  \begin{rem}
\label{rem-Lipschitz}
Consider the setting in Example \ref{eg-Lipschitz} (ii). Denote $\dbP^\a = \dbP^{0,\vec x, \a}$, we have
\beaa
\dis \dbE^{\dbP^\a}[\mu^N_2(\ul x)] \!\!\!&=&\!\!\! {1\over N}\sum_{i=1}^N \dbP^\a(X^i_2=\ul x) = {1\over N}\sum_{i=1}^N \dbE^{\bar\dbP}\big[\a(\mu^N_1)\big] \\
\!\!\!&=&\!\!\! {1\over 3} \bar \dbP\big(\mu^N_1(\ul x) \le {1\over 2} \big) + {2\over 3} \bar \dbP\big(\mu^N_1(\ul x) > {1\over 2} \big)={1\over 2};\\
\dis \dbE^{\dbP^\a}[|\mu^N_2(\ul x)|^2] \!\!\!&=&\!\!\!   {1\over N^2}\sum_{i,j=1}^N \dbP^\a(X^i_2=X^j_2=\ul x)  = {1\over N^2}\sum_{i=1}^N \dbE^{\bar\dbP}\big[\a(\mu^N_1)\big] + {1\over N^2}\sum_{i\neq j} \dbE^{\bar\dbP}\big[|\a(\mu^N_1)|^2\big]\\
\!\!\!&=&\!\!\! {1\over 9} \bar \dbP\big(\mu^N_1(\ul x) \le {1\over 2} \big) + {4\over 9} \bar \dbP\big(\mu^N_1(\ul x) > {1\over 2} \big) +O\big({1\over N}\big)={5\over 18} + O\big({1\over N}\big);\\
\dis Var^{\dbP^\a}(\mu^N_2(\ul x))\!\!\!&=&\!\!\! {5\over 18} + O\big({1\over N}\big)- ({1\over 2})^2 = {1\over 36} +  O\big({1\over N}\big).
\eeaa
Then we see that  the random measure $\mu^N_2$ under $\dbP^\a$, which is an $O({1\over \sqrt{N}})$-equilibrium measure of the $N$-player problem,  does not converge to a deterministic measure. This explains why \cite{Lacker2} introduced the weak mean field equilibrium when considering the convergence issue for all measurable controls. However, we shall emphasize again that, as pointed out in Remark \ref{rem-regularity} (iii), measurable controls/equilibria are not desirable for numerical or practical purpose.  
\end{rem}

\subsection{The subtle path dependence issue in Remark
  \ref{rem-statepath}} 
  In this subsection we elaborate Remark
\ref{rem-statepath} (ii) and (iii). Throughout the subsection,
$q, F, G$ are state dependent as in Section \ref{sect-state}. As we always saw in Example \ref{eg-statepath}, in general $\dbV_{state}\neq \dbV_{path}$, confirming Remark
\ref{rem-statepath} (ii). 
We now turn to Remark \ref{rem-statepath} (iii) for relaxed controls. For simplicity we verify it only for raw set values. The equality for set values follow similar ideas but with more involved approximations, as we saw in Example \ref{eg-statepath} (iv). Let $\cA_{relax}$ be the path
dependent ones in Section \ref{sect-relax}, and $\cA^{state}_{relax}$
denote the subset taking the form $\g(t, x, da)$. We emphasize again
that here we are considering state dependent $q, F, G$. Fix $t=0$ and
$\mu\in \cP_0(\dbS)$.

\begin{lem}
  \label{lem-gstate} For any $\g \in \cA_{relax}$, define
  \begin{equation}
    \label{statetildeg} \tilde \g(s, x, da) :={1\over \mu^\g_s(x)}
    \sum_{\bx\in \dbX_s: \bx_s = x} \mu^\g_{s\wedge \cd}(\bx) \g(s, \bx,
    da),~\mbox{where}~ \mu^\g_s(x):= \sum_{\bx\in \dbX_s: \bx_s = x}
    \mu^\g_{s\wedge \cd}(\bx). 
  \end{equation}
  Then $\tilde \g\in \cA^{state}_{relax}$ and
  $\mu^{\tilde \g}_s = \mu^\g_s$.
\end{lem}
\begin{proof}
  First it is obvious that \beaa \tilde \g(s, x, \dbA) = {1\over
    \mu^\g_s(x)} \sum_{\bx\in \dbX_s: \bx_s = x} \mu^\g_{s\wedge
    \cd}(\bx) \g(s, \bx, \dbA) = {1\over \mu^\g_s(x)} \sum_{\bx\in
    \dbX_s: \bx_s = x} \mu^\g_{s\wedge \cd}(\bx) =1, \eeaa so
  $\tilde \g\in \cA^{state}_{relax}$. Next, by definition
  $\mu^{\tilde \g}_0 =\mu = \mu^\g_0$. Assume
  $\mu^{\tilde \g}_s = \mu^\g_s$, then for $s+1$, \beaa \mu^{\tilde
    \g}_{s+1}(x) &=& \sum_{\tilde x\in \dbS} \mu^{\tilde \g}_s(\tilde
  x) \int_\dbA q(s, \tilde x, \mu^{\tilde \g}_s, a; x) \tilde \g(s,
  \tilde x, da) \\ &=& \sum_{\tilde x\in \dbS} \mu^{\g}_s(\tilde x)
  \int_\dbA q(s, \tilde x, \mu^{\g}_s, a; x) {1\over \mu^\g_s(\tilde
    x)} \sum_{\bx\in \dbX_s: \bx_s =\tilde x} \mu^\g_{s\wedge
    \cd}(\bx) \g(s, \bx, da)\\ &=& \sum_{\bx\in \dbX_s}
  \mu^\g_{s\wedge \cd}(\bx) \int_\dbA q(s, \bx_s, \mu^{\g}_s, a; x)
  \g(s, \bx, da) = \mu^\g_{s+1}(x).  \eeaa This completes the
  induction argument.
\end{proof}

\begin{lem}
  \label{lem-stateMFE} If $\g^*\in \cA_{relax}$ is a relaxed MFE at
  $(0, \mu)$, then the corresponding
  $\tilde \g^*\in \cA^{state}_{relax}$ is a state dependent relaxed
  MFE at $(0, \mu)$. Moreover, in this case we have \bea
\label{stateJ} J(0, \mu, \g^*; x, \g^*) = J(0, \mu, \tilde \g^*; x,
\tilde \g^*).  \eea
\end{lem}
\begin{proof}
  First, by Lemma \ref{lem-gstate} it is straightforward to verify
  that \beaa \int_\dbS J(0, \mu, \g; x, \g) \mu(dx) = \int_\dbS J(0,
  \mu, \tilde \g; x, \tilde \g) \mu(dx).  \eeaa On the other hand,
  since $\g^*\in \cA_{relax}$, by the standard control theory we have
  \begin{equation}
    \label{statev} \inf_{\g\in \cA_{relax}} J(0, \mu, \g^*; x, \g) =
    v(\mu^{\g^*}; 0, x) = v(\mu^{\tilde \g^*}; 0, x) = \inf_{\g'\in
      \cA^{state}_{relax}} J(0, \mu, \tilde\g^*; x, \g').  
  \end{equation}
  Then \beaa \int_\dbS J(0, \mu, \tilde \g^*; x, \tilde \g^*) \mu(dx)=
  \int_\dbS J(0, \mu, \g^*; x, \g^*) \mu(dx) = \int_\dbS v(\mu^{\tilde
    \g^*}; 0, x)\mu(dx).  \eeaa Since
  $J(0, \mu, \tilde \g^*; x, \tilde \g^*) \ge v(\mu^{\tilde \g^*}; 0,
  x)$ and $\supp(\mu) = \dbS$, then
  $J(0, \mu, \tilde \g^*; x, \tilde \g^*) =v(\mu^{\tilde \g^*}; 0, x)$
  for all $x\in \dbS$. This implies that
  $\tilde \g^*\in \cA^{state}_{relax}$ is a state dependent relaxed
  MFE at $(0, \mu)$, and consequently \reff{statev} leads to
  \reff{stateJ}.
\end{proof}

\begin{thm}
\label{thm-relaxstate} The MFGs with state dependent relaxed controls
and path dependent relaxed controls have the same relaxed raw set
value.
\end{thm}
\begin{proof}
  By Lemma \ref{lem-stateMFE}, clearly the path dependent raw set
  value is included in the state dependent raw set value. On the other
  hand, for any state dependent relaxed control
  $\hat \g^*\in \cA^{state}_{relax}$, we may still view
  $\g^* := \hat \g^*$ as a path dependent relaxed
  control\footnote{While it is trivial that
    $\cA^{state}_{relax} \subset \cA^{path}_{relax}:=\cA_{relax}$, as
    stated here, in general it is not trivial that
    $\cM^{state}_{relax} \subset \cM^{path}_{relax}$, because for the
    latter one has to compare with other path dependent relax
    controls, which is a stronger requirement than that for
    $\cM^{state}_{relax}$. The rest of the proof is exactly to prove
    $\cM^{state}_{relax} \subset \cM^{path}_{relax}$.}, and it is
  straightforward to verify that the
  $\tilde \g^*\in \cA^{state}_{relax}$ corresponding to $\g^*$ is
  equal to $\hat \g^*$. Then, following the arguments in Lemma
  \ref{lem-stateMFE}, in particular \reff{statev}, one can easily show
  that $J(0, \mu, \g^*; x, \g^*) = v(\mu^{\g^*}; 0, x)$ and thus
  $\g^*$ is also an MFE among $\cA_{relax}$. Therefore,
  $J(0, \mu, \g^*; \cd, \g^*)$ belong to the path dependent raw set
  value as well.
\end{proof}

\subsection{Some technical proofs}

\no{\bf Proof of Theorem \ref{thm-DPP1}.}
  Let
  $\tilde \dbV_{state}(t,\mu)= \bigcap_{\e>0}\tilde
  \dbV^\e_{state}(t,\mu)$ denote the right side of \reff{DPP1} in the
  obvious sense. We shall follow the arguments in Theorem
  \ref{thm-DPP0}.

  (i) We first prove
  $\tilde \dbV_{state}(t,\mu)\subset \dbV_{state}(t,\mu)$. Fix
  $\f\in \tilde \dbV_{state}(t,\mu)$, $\e>0$, and set
  $\e_1:= {\e\over 4}$. Since
  $\f\in \tilde \dbV^{\e_1}_{state}(t, \mu)$, there exist desirable
  $\psi$ and $\a^*\in \cM^{\e_1}_{state}(T_0, \psi; t,\mu)$ as in
  \reff{DPP1}, and the property
  $\psi(\cd, \mu^{\a^*}_{T_0}) \in \dbV^{\e_1}_{state}(T_0,
  \mu^{\a^*}_{T_0})$ implies further that there exists
  $\tilde \a^*\in \cM^{\e_1}_{state}(T_0, \mu^{\a^*}_{T_0})$ such that
  \beaa \|\f-J(T_0, \psi; t, \mu, \a^*; \cd, \a^*)\|_\infty \le
  \e_1,\q \|\psi(\cd, \mu^{\a^*}_{T_0}) - J(T_0, \mu^{\a^*}_{T_0},
  \tilde \a^*; \cd, \tilde \a^*)\|_\infty\le \e_1.  \eeaa Denote
  $\hat \a^*:= \a^*\oplus_{T_0} \tilde \a^* \in \cA_{state}$. Then,
  for any $\a\in \cA_{state}$ and $x\in \dbS$, similar to the
  arguments in Proposition \ref{prop-DPP} (i), we have \beaa &&J(t,
  \mu, \hat\a^*; x, \a) = \dbE^{\dbP^{\mu^{\a^*}; t, x, \a}}\Big[
  J(T_0, \mu^{\a^*}_{T_0}, \tilde \a^*; X_{T_0}, \a) +
  \sum_{s=t}^{T_0-1} F(s, X_s, \mu^{\a^*}_s, \a(s, X_s))\Big]\\ &&\ge
  \dbE^{\dbP^{\mu^{\a^*}; t, x, \a}}\Big[ J(T_0, \mu^{\a^*}_{T_0},
  \tilde \a^*; X_{T_0}, \tilde \a^*) +
  \sum_{s=t}^{T_0-1} F(s, X_s, \mu^{\a^*}_s, \a(s, X_s))\Big]-\e_1\\
  &&\ge \dbE^{\dbP^{\mu^{\a^*}; t, x, \a}}\Big[ \psi(X_{T_0},
  \mu^{\a^*}_{T_0}) + \sum_{s=t}^{T_0-1} F(s, X_s, \mu^{\a^*}_s, \a(s,
  X_s))\Big] - 2\e_1\\ &&= J(T_0, \psi; t, \mu, \a^*; x, \a) -2\e_1
  \ge J(T_0, \psi; t, \mu, \a^*; x, \a^*) -3\e_1 \\ &&=
  \dbE^{\dbP^{\mu^{\a^*}; t, x, \a^*}}\Big[ \psi(X_{T_0},
  \mu^{\a^*}_{T_0}) + \sum_{s=t}^{T_0-1} F(s, X_s, \mu^{\a^*}_s,
  \a^*(s, X_s))\Big] - 3\e_1\\ &&\ge \dbE^{\dbP^{\mu^{\a^*}; t, x,
      \a^*}}\Big[ J(T_0, \mu^{\a^*}_{T_0}, \tilde \a^*; X_{T_0},
  \tilde \a^*) + \sum_{s=t}^{T_0-1} F(s, X_s, \mu^{\a^*}_s, \a^*(s,
  X_s))\Big] - 4\e_1\\ &&= J(t, \mu, \hat\a^*; x, \hat \a^*) - \e.
  \eeaa That is, $\hat \a^* \in \cM^\e_{state}(t,\mu)$. Moreover, note
  that \beaa &&\|\f - J(t, \mu, \hat\a^*; \cd, \hat\a^*)\|_\infty \le
  \e_1 + \|J(T_0, \psi; t, \mu, \a^*; \cd, \a^*)-J(t, \mu, \hat\a^*;
  \cd, \hat\a^*)\|_\infty\\ &&=\e_1 + \sup_{x\in \dbS}
  \Big|\dbE^{\dbP^{\mu^{\a^*}; t, x, \a^*}}\Big[ \psi(X_{T_0},
  \mu^{\a^*}_{T_0}) - J(T_0, \mu^{\a^*}_{T_0}, \tilde \a^*; X_{T_0},
  \tilde \a^*)\Big]\Big| \le 2\e_1 \le \e.  \eeaa Then
  $\f\in \dbV^\e_{state}(t,\mu)$. Since $\e>0$ is arbitrary, we obtain
  $\f\in \dbV_{state}(t,\mu)$.

  (ii) We now prove the opposite inclusion. Fix $\f\in
  \dbV_{state}(t,\mu)$ and $\e>0$.  Let $\e_1>0$ be a small number which
  will be specified later. Since $\f\in \dbV^{\e_1}_{state}(t, \mu)$,
  there exists $\a^*\in \cM^{\e_1}_{state}(t,\mu)$ such that $\|\f-J(t,
  \mu, \a^*; \cd, \a^*)\|_\infty \le \e_1$. Introduce $\psi(x,
  \nu):=J(T_0, \nu, \a^*; x, \a^*)$.  By \reff{tower} we have \beaa
  \|\f-J(T_0,\psi; t, \mu, \a^*; \cd, \a^*)\|_\infty = \|\f-J(t, \mu,
  \a^*; \cd, \a^*)\|_\infty \le \e_1.  \eeaa Moreover, since $\a^*\in
  \cM^{\e_1}_{state}(t,\mu)$, for any $\a\in \cA_{state}$ and $x\in
  \dbS$, we have \beaa &&J(T_0,\psi; t, \mu, \a^*; x, \a^*) = J(t, \mu,
  \a^*; x, \a^*) \\ &&\le J(t, \mu, \a^*; x, \a\oplus_{T_0} \a^*) +\e_1
  = J(T,\psi; t, \mu, \a^*; x, \a)+\e_1.  \eeaa This implies that
  $\a^*\in \cM^{\e_1}_{state}(T_0, \psi; t,\mu)$. We claim further that
  \bea
  \label{DPP1-claim} \psi(\cd, \mu^{\a^*}_{T_0}) \in \dbV_{C\e_1}(T_0,
  \mu^{\a^*}_{T_0}), \eea for some constant $C\ge 1$. Then by
  \reff{DPP1} we see that $\f\in \tilde
  \dbV^{C\e_1}_{state}(t,\mu)\subset \tilde \dbV^\e_{state}(t,\mu)$ by
  setting $\e_1\le {\e\over C}$. Since $\e>0$ is arbitrary, we obtain
  $\f\in \tilde \dbV_{state}(t,\mu)$.

  To show \reff{DPP1-claim}, we follow the arguments in Proposition
  \ref{prop-DPP} (ii). Recall $v$ in \reff{J} and the standard DPP
  \reff{vDPP} for $v$, for any $x\in \dbS$ we have \beaa
  \dbE^{\dbP^{\mu^{\a^*}; t, x, \a^*}}\Big[J(T_0, \mu^{\a^*}_{T_0},
  \a^*; X_{T_0}, \a^*) \Big] &\le&\!\! \inf_{\a\in\cA_{state}}\!\!
  \dbE^{\dbP^{\mu^{\a^*}; t, x, \a^*}}\Big[J(T_0, \mu^{\a^*}_{T_0},
  \a^*; X_{T_0}, \a) \Big]+\e_1\\ &=&\dbE^{\dbP^{\mu^{\a^*}; t, x,
      \a^*}}\Big[v(\mu^{\a^*}; T_0, X_{T_0}) \Big]+\e_1, \eeaa It is obvious
  that $v(\mu^{\a^*}; T_0, \cd) \le J(T_0, \mu^{\a^*}_{T_0}, \a^*; \cd,
  \a^*)$. Moreover, since $q\ge c_q$, clearly $\dbP^{\mu^{\a^*}; t, x,
    \a^*}(X_{T_0}=\tilde x) \ge c_0^{T_0-t}$, for any $\tilde x\in
  \dbS$. Thus, for $C:= c_0^{t-T_0}$, \beaa 0&\le& J(T_0,
  \mu^{\a^*}_{T_0}, \a^*; \tilde x, \a^*) - v(\mu^{\a^*}; T_0, \tilde
  x)\\ &\le& C\dbE^{\dbP^{\mu^{\a^*}; t, x, \a^*}}\Big[\big[J(T_0,
  \mu^{\a^*}_{T_0}, \a^*; X_{T_0}, \a^*) - v(\mu^{\a^*}; T_0,
  X_{T_0})\big]\1_{\{X_{T_0}=\tilde x\}} \Big]\\ &\le&C
  \dbE^{\dbP^{\mu^{\a^*}; t, x, \a^*}}\Big[\big[J(T_0, \mu^{\a^*}_{T_0},
  \a^*; X_{T_0}, \a^*) - v(\mu^{\a^*}; T_0,X_{T_0})\big] \Big] \le
  C\e_1.  \eeaa This implies that $\a^*\in \cM^{C\e_1}_{state}(T_0,
  \mu^{\a^*}_{T_0})$. Since $\psi(\cd, \mu^{\a^*}_{T_0}) = J(T_0,
  \mu^{\a^*}_{T_0}, \a^*; \cd, \a^*)$, we obtain \reff{DPP1-claim}
  immediately, and hence $\f\in \tilde \dbV_{state}(t,\mu)$.
\qed

\ms
\no{\bf Proof of the claim in Remark \ref{rem-Lginverse}}.
  By \reff{gL} and \reff{Lg} we have \beaa &&\dis \g^{(\L^\g)}(s,
  \tilde \bx, da) := {1\over \mu^\g_{s\wedge \cd}(\tilde \bx)}
  \int_{\cA^t_{path}} Q^t_s(\mu^\g; \tilde \bx; \a) \d_{\a(s, \tilde
    \bx)}(da) \L^\g(\bx, d\a)\\ &&\dis = {1\over \mu^\g_{s\wedge
      \cd}(\tilde \bx)} \int_\dbA\cds\int_\dbA \big[ \prod_{r=t}^{s-1}
  q(r, \tilde\bx, \mu^\g, \a(r, \tilde\bx); \bx_{r+1}) \big]\times
  \d_{\a(s, \tilde \bx)}(da)\times \\ &&\dis\qq \big[\mu(\bx)
  \prod_{r=t}^{T-1} \prod_{\bar \bx\in \dbX^{t, \bx}_s} \g(r, \bar
  \bx, d\a(r, \bar \bx))\big]\\ &&\dis = {\mu(\bx)\over
    \mu^\g_{s\wedge \cd}(\tilde \bx)} \int_\dbA\cds\int_\dbA \big[
  \prod_{r=t}^{s-1} q(r, \tilde\bx, \mu^\g, \a(r, \tilde\bx);
  \bx_{r+1}) \g(r, \tilde \bx, d\a(r, \tilde \bx))\big]\times \\
  &&\dis\qq \big[ \d_{\a(s, \tilde \bx)}(da) \g(s, \tilde \bx, d\a(s,
  \tilde\bx)) \prod_{\bar \bx\in \dbX^{t, \bx}_s\backslash \{\tilde
    \bx\}} \g(s, \bar \bx, d\a(s, \bar\bx))\big] \times\\ &&\dis\qq
  \big[ \prod_{r=t}^{s-1} \prod_{\bar \bx\in \dbX^{t, \bx}_s\backslash
    \{\tilde \bx\}} \g(r, \bar \bx, d\a(r, \bar \bx))\big] \big[
  \prod_{r=s}^{T-1} \prod_{\bar \bx\in \dbX^{t, \bx}_s} \g(r, \bar
  \bx, d\a(r, \bar \bx))\big]\\ &&\dis = {\mu(\bx)\over
    \mu^\g_{s\wedge \cd}(\tilde \bx)} \big[ \prod_{r=t}^{s-1}
  \int_\dbA q(r, \tilde\bx, \mu^\g, \bar a; \bx_{r+1}) \g(r, \tilde
  \bx, d\bar a)\big]\times \big[ \g(s, \tilde \bx, da)\big]\\ &&\dis =
  {\mu(\bx)\over \mu^\g_{s\wedge \cd}(\tilde \bx)} Q^t_s(\mu^\g;
  \tilde \bx, \g) \g(s, \tilde \bx, d a) = \g(s, \tilde \bx, d a).
  \eeaa That is, $\g^{(\L^\g)} = \g$.
\qed

\ms
\no{\bf Proof of Lemma \ref{lem-vreg}}.
  Clearly the uniform estimate for $J(\mu^\a;\cd)$ implies that for
  $v(\mu^\a;\cd)$, so we shall only prove the former one. Fix
  $(t, \mu)\in [0, T]\times \cP_2$ and $\a, \tilde \a\in \cA_{cont}$,
  and denote $u(s, x):= J(\mu^\a; \tilde \a, s, x)$.  By standard PDE
  theory $u$ is a classical solution to the linear PDE in \reff{HJB}
  and we have the following formula: denoting
  $X^{s,x}_r := x + B_r - B_s$, \beaa &\dis\pa_x u(s, x) =
  \dbE^\dbP\Big[[g(X^{s, x}_T, \mu^\a_T)- g(x, \mu^\a_T)]
  {B_T-B_s\over T-s} \\ &\dis+ \int_s^T \big[ b(r, X^{s,x}_t,
  \mu^\a_r, \tilde \a(r, X^{s,x}_r)) \cd \pa_x u(r, X^{s,x}_r) + f(r,
  X^{s,x}_t, \mu^\a_r, \tilde \a(r, X^{s,x}_r))\big] {B_r-B_s\over
    r-s} dr\Big].  \eeaa Then, by the Lipschitz continuity of $g$ and
  the boundedness of $b$ and $f$, \beaa |\pa_x u(s, x)| &\le&
  \dbE\Big[L_0 {|B_T-B_s|^2\over T-s} + C_0\int_s^T \big[ |\pa_x u(r,
  X^{s,x}_r)| + 1\big] {|B_r-B_s|\over r-s} dr\Big]\\ &\le& C +
  C_0\dbE\Big[ \int_s^T |\pa_x u(r, X^{s,x}_r)|{|B_r-B_s|\over
    r-s}dr\Big].  \eeaa Denote
  $K_s:= e^{\l s} \sup_x |\pa_x u(s, x)|$,
  $\bar K:= \sup_{t\le s\le T} K_s$, for some constant $\l>0$. Then
  \beaa K_s &\le& C e^{\l s} + C_0 \int_s^T {K_r e^{-\l (r-s)} \over
    \sqrt{r-s}}dr\le C e^{\l s} + C_0\bar K\int_s^T {e^{-\l
      (r-s)}\over \sqrt{r-s}} dr \\ &\le& C e^{\l s} + C_0\bar
  K\int_s^\infty {e^{-\l (r-s)}\over \sqrt{r-s}} dr =C e^{\l s} +
  C_0\bar K\int_0^\infty {e^{-\l r}\over \sqrt{r}} dr =C e^{\l s} +
  {C_0\over \sqrt{\pi \l}} \bar K.  \eeaa Thus
  $\bar K\le {C_0\over \sqrt{\pi \l}} \bar K + C e^{\l T}$.  Set
  $\l:= {4C_0^2\over \pi}$ so that
  ${C_0\over \sqrt{\pi \l}} = {1\over 2}$, we obtain
  $\bar K\le C_1:= 2Ce^{\l T}$, which implies the desired estimate
  immediately.
\qed

\ms
\no{\bf Proof of Proposition \ref{prop-contNreg}}.
  Fix $(t, \vec x, \vec \a, \bar x, \tilde x)$ and $i$. For any
  $\tilde \a\in \cA^L_{cont}$, introduce
  $\bar \a(s, x, \mu) := \tilde \a(s, x-\bar x + \tilde x, \mu)$, and
  denote \beaa &\dis \bar X^{i}_s := \bar x + B^i_s - B^i_t,\q X^j_s
  := x_j + B^j_s - B^j_t,\q j\neq i,\q ;\\ &\dis \bar \mu^{N}_s :=
  {1\over N} \big[\d_{\bar X^{i}_s} + \sum_{j\neq i} \d_{X^j_s}\big],~
  \bar M^{j}_s := \exp\Big(\int_t^s \bar b^{j}_r dB^j_r - {1\over
    2}\int_t^s |\bar b^{j}_r|^2 dr\Big), j\ge 1, ~ \mbox{where}\\
  &\dis \bar b^{i}_s:= b(s, \bar X^{i}_s, \bar \mu^{N}_s, \bar \a(s,
  \bar X^{i}_s, \bar \mu^{N}_s)),\q \bar b^{j}_s:= b(s, X^j_s, \bar
  \mu^{N}_s, \a_j(s, X^j_s, \bar \mu^{N}_s)), j\neq i.  \eeaa By the
  Girsanov Theorem we have \beaa J_i(t, (\vec x^{-i},\bar x), (\vec
  \a^{-i}, \bar \a)) = \dbE\Big[\big[\prod_{j=1}^N \bar M^{j}_T\big]
  \big[g(\bar X^{i}_T, \bar\mu^{N}_T) + \int_t^T f(s, \bar X^{i}_s,
  \bar \mu^{N}_s, \bar \a(s, \bar X^{i}_s, \bar\mu^{N}_s)) \big]
  ds\Big].  \eeaa Similarly define $\tilde X^i$, $\tilde \mu^N$,
  $\tilde M^j$, $\tilde b^i$, $\tilde b^j$ corresponding to
  $(\tilde x, \tilde \a)$ in the obvious sense. Then we have a similar
  expression as above and
  $\bar \a(s, \bar X^{i}_s,\mu) = \tilde \a(s, \tilde X^{i}_s,
  \mu)$. Therefore, \bea
  \label{vLiJi} \left.\ba{c} \dis v^{N,L}_i\big(t, (\vec x^{-i}, \bar
    x), \vec \a\big) - J_i(t, (\vec x^{-i}, \tilde x), (\vec \a^{-i},
    \tilde \a)) \\ \dis \le J_i(t, (\vec x^{-i}, \bar x), (\vec
    \a^{-i}, \bar \a))-J_i(t, (\vec x^{-i}, \tilde x), (\vec \a^{-i},
    \tilde \a)) \le C\sum_{j=1}^N K^j_T + K_0, \ea\right.  \eea where
  \beaa K^j_s &:=& \dbE\Big[\big[\prod_{k< j} \bar M^{k}_s\big]
  \big[\prod_{k> j} \tilde M^{k}_s\big]\big|\bar M^{j}_s -\tilde
  M^{j}_s|\Big],\q j\ge 1;\\ K_0 &:=& \dbE\Big[\prod_{j=1}^N \bar
  M^{j}_T \big[ |g(\bar X^{i}_T, \bar \mu^{N}_T) - g(\tilde X^{i}_T,
  \tilde\mu^{N}_T)|\\ &&+\int_t^T | f(s, \bar X^{i}_s, \bar\mu^{N}_s,
  \bar\a(s, \bar X^{i}_s, \bar\mu^{N}_s)) - f(s, \tilde X^{i}_s,
  \tilde\mu^{N}_s, \tilde \a(s, \tilde X^{i}_s,
  \tilde\mu^{N}_s))|ds\big]\Big].  \eeaa Denote
  $\D x := \bar x-\tilde x$. Note that \bea
  \label{DXest} \left.\ba{c} \dis\bar X^{i}_s -\tilde X^{i}_s = \D
    x,\q W_1(\bar\mu^{N}_s, \tilde \mu^{N}_s) \le {|\D x|\over N},\\
    \dis \big|\bar\a(s, \bar X^{i}_s, \bar\mu^{N}_s) - \tilde \a(s,
    \tilde X^{i}_s, \tilde\mu^{N}_s)\big| = \big|\tilde \a(s, \tilde
    X^{i}_s, \bar \mu^{N}_s) - \tilde \a(s, \tilde X^{i}_s, \tilde
    \mu^{N}_s)\big| \le {L\over N} |\D x|.  \ea\right.  \eea By the
  required Lipschitz continuity, we have \bea
  \label{K0} K_0 \le C\dbE^\dbP\Big[\prod_{j=1}^N \bar M^{j}_T \big[
  [1+{1\over N}]|\D x| +\int_t^T [1+{L\over N}]|\D x| ds\big]\Big]\le
  C|\D x|.  \eea

  Next, introduce \beaa \G^j_s := \dbE\Big[\big[\prod_{k< j} \bar
  M^{k}_s\big] \big[\prod_{k> j} \tilde M^{k}_s\big] \big|\bar
  M^{j}_s|^2\Big],\q \D \G^j_s := \dbE\Big[\big[\prod_{k< j} \bar
  M^{k}_s\big] \big[\prod_{k> j} \tilde M^{k}_s\big] \big|\bar
  M^{j}_s- \tilde M^j_s|^2\Big].  \eeaa Note that $B^1,\cds, B^N$ are
  independent. By applying the It\^o formula, we have \beaa \G^j_s =
  1+ \int_t^s \dbE\Big[\big[ \prod_{k< j} \bar M^{k}_r\big]
  \big[\prod_{k> j} \tilde M^{k}_r\big] \big|\bar M^{j}_r \bar
  b^{j}_r|^2 \Big]dr \le 1+C\int_t^s \G^j_r dr, \eeaa Then
  $\G^j_s\le C$. Thus, by applying the It\^o formula again we have
  \beaa \D \G^j_s &=&\int_t^s \dbE\Big[\big[\prod_{k< j} \bar
  M^{k}_r\big] \big[\prod_{k> j} \tilde M^{k}_r\big] \big[\bar M^{j}_r
  \bar b^{j}_r -\tilde M^{j}_r\tilde b^{j}_r]^2\Big] dr\\ &\le&
  C\int_t^s \dbE\Big[\big[\prod_{k< j} \bar M^{k}_r\big]
  \big[\prod_{k> j} \tilde M^{k}_r\big] \big[|\bar M^{j}_r-\tilde
  M^{j}_r| +\bar M^{j}_r | \bar b^{j}_r-\tilde b^{j}_r|\big]^2\Big]
  dr\\ &\le&C\int_t^s \D \G^j_r dr + C\int_t^s \dbE\Big[\big[\prod_{k<
    j} \bar M^{k}_r\big] \big[\prod_{k> j} \tilde M^{k}_r\big]
  \big[\bar M^{j}_r | \bar b^{j}_r-\tilde b^{j}_r|]^2\Big] dr.  \eeaa
  Note that, by \reff{DXest}, \beaa &\dis | \bar b^{i}_r-\tilde
  b^{i}_r| = \Big|b(s, \bar X^{i}_s, \bar \mu^{N}_s, \tilde\a(s,
  \tilde X^{i}_s, \bar \mu^{N}_s)) - b(s, \tilde X^{i}_s, \tilde
  \mu^{N}_s, \tilde \a(s, \tilde X^{i}_s, \tilde \mu^{N}_s))\Big| \le
  C_L|\D x|\\ &\dis | \bar b^{j}_r-\tilde b^{j}_r|\le {C_L\over N} |\D
  x|,\q j\neq i.  \eeaa Then, since $\G^j_s\le C$, \beaa \D \G^i_s \le
  C\int_t^s \D \G^i_r dr + C_L|\D x|^2,\q \D \G^j_s \le C\int_t^s \D
  \G^j_r dr + {C_L\over N^2} |\D x|^2, ~j\neq i.  
  \eeaa 
  and thus 
  \bea
   \label{Kj}
  \left.\ba{c}
  \dis  \D \G^i_s \le C_L|\D x|^2,\q K^i_s \le {|\D x|\over 2} + { \D \G^i_s\over 2|\D
    x|} \le C_L |\D x|;\ms\\
  \dis \D \G^j_s \le {C_L\over N^2} |\D x|^2, \q K^j_s \le {|\D x|\over 2N} + {N \D
    \G^j_s\over 2|\D x|}\le {C_L\over N}|\D x|,\q j\neq i. 
    \ea\right.
     \eea 
     Then,
  by \reff{vLiJi}, \reff{K0} and \reff{Kj} we have \beaa &\dis
  v^{N,L}_i\big(t, (\vec x^{-i}, \bar x), \vec \a\big) - J_i(t, (\vec
  x^{-i}, \tilde x), (\vec \a^{-i}, \tilde \a)) \le K_0 + C K^i_s +
  C\sum_{j\neq i} K^j_s\\ &\dis \le C|\D x| + C_L|\D x| + C_L
  \sum_{j\neq i} {|\D x|\over N} \le C_L |\D x|.  \eeaa Since
  $\tilde \a\in \cA^L$ is arbitrary, we obtain
  $v^{N,L}_i\big(t, (\vec x^{-i}, \bar x), \vec \a\big)
  -v^{N,L}_i\big(t, (\vec x^{-i}, \tilde x), \vec \a\big)\le C_L |\D
  x|$. Similarly we have
  $v^{N,L}_i\big(t, (\vec x^{-i}, \tilde x), \vec \a\big) - v\big(t,
  (\vec x^{-i}, \bar x), \vec \a\big)\le C_L |\D x|$, and hence
  \reff{vLireg}.
\qed


\begin{thebibliography}{99}

\bibitem{APS} Abreu, D.; Pearce, D.; and Stacchetti, E., {\it Toward a
Theory of Discounted Repeated Games with Imperfect Monitoring}, {\sl
Econometrica}, 58 (1990), 1041-1063.

\bibitem{BardiFischer} Bardi, M. and Fischer, M., {\it On
non-uniqueness and uniqueness of solutions in finite-horizon Mean
Field Games}, ESAIM: Control, Optimization and Calculus of Variations,
25 (2019), 44.
  
\bibitem{BC} Bayraktar, E. and Cohen, A., {\it Analysis of a finite
state many player using its master equation}, {\sl SIAM J. Control
Optim.} 56 (2018), 3538-3568.

\bibitem{BCCD1} Bayraktar, E.; Cecchin, A.; Cohen, A.; and Delarue,
F., {\it Finite state mean field games with Wright-Fisher common
noise}, {\sl J. Math. Pures et Appliqu\'ees}, 147 (2021), 98-162.

\bibitem{BCCD2} Bayraktar, E.; Cecchin, A.; Cohen, A.; and Delarue,
F., {\it Finite state mean field games with Wright-Fisher common noise
as limits of $N$-player weighted games}, {\sl
Mathematics of Operations Research}, accepted,  arXiv:2012.04845.

\bibitem{BFY} Bensoussan, A.; Frehse, J.; and Yam, S. C. P., {\sl Mean
Field Games and Mean Field Type Control Theory}, (2013).  New York:
Springer Verlag.


\bibitem{CHM} Caines, P. E.; Huang, M.; and Malham\'e, R. P., {\it
Large population stochastic dynamic games: closed loop McKean-Vlasov
systems and the Nash certainty equivalence principle}, {\sl
Commun. Inf. Syst.}, 6 (2006), 221--251.

\bibitem{Cardaliaguet} Cardaliaguet, P., {\sl Notes on mean field
games}, lectures by P.L. Lions, Coll\`ege de France, (2010).

\bibitem{Cardaliaguet2} Cardaliaguet, P., {\it The convergence problem
in mean field games with local coupling}, {\sl Applied Mathematics \&
Optimization}, 76 (2017), 177-215.

 \bibitem{CDLL} Cardaliaguet, P.; Delarue, F.; Lasry, J.M.; and Lions,
P.L., {\sl The master equation and the convergence problem in mean
field games}, Annals of Mathematics Studies, 201. Princeton University
Press, Princeton, NJ, (2019). x+212 pp.

\bibitem{Carmona} Carmona, G., {\it Nash Equilibria of Games with a
Continuum of Players}, {\sl Game Theory and Information}, 0412009,
University Library of Munich, Germany, 2004.

\bibitem{CD0} Carmona, R. and Delarue, F. {\it Probabilistic Analysis
of Mean-Field Games}, {\sl SIAM Journal on Control and Optimization},
51 (2013), 2705-2734.

\bibitem{CD1} Carmona, R. and Delarue, F., {\sl Probabilistic theory
of mean field games with applications I - Mean field FBSDEs, control,
and games}, Probability Theory and Stochastic Modeling, 83. Springer,
Cham, (2018).

\bibitem{CD2} Carmona, R. and Delarue, F., {\sl Probabilistic theory
of mean field games with applications II - Mean field games with
common noise and master equations}, Probability Theory and Stochastic
Modeling, 84. Springer, Cham, (2018).

\bibitem{CDFP} Cecchin A.; Dai Pra P.; Fischer M.; and Pelino G., {\it
On the convergence problem in mean field games: a two state model
without uniqueness}, {\sl SIAM J. Control Optim.} 57 (2019), 
2443-2466.

\bibitem{CecchinDelarue} Cecchin A. and Delarue F., {\it Selection by
vanishing common noise for potential finite state mean field games},
{\sl Communications in Partial Differential Equations}, 47 (2022), 89-168.

\bibitem{CP} Cecchin, A. and Pelino, G., {\it Convergence,
Fluctuations and Large Deviations for finite state Mean Field Games
via the Master Equation}.  {\sl Stoch. Process. Appl.} 129 (2019),
4510-4555.

 \bibitem{Delarue} Delarue, F., {\it Restoring uniqueness to
mean-field games by randomizing the equilibria}, {\sl Stochastic and
Partial Differential Equations: Analysis and Computations}, 7 (2019),
598-678.

\bibitem{DF} Delarue, F. and Foguen Tchuendom, R., {\it Selection of
equilibria in a linear quadratic mean-field game}, {\sl Stochastic
Process. Appl.} 130 (2020), 1000-1040.

\bibitem{DLR1} Delarue, F.; Lacker D.; and Ramanan, K., {\it From the
master equation to mean field game limit theory: Large deviations and
concentration of measure}. {\sl Ann. Probab.} 48 (2020), 211-263.

\bibitem{DLR2} Delarue, F.; Lacker D.; and Ramanan, K., {\it From the
master equation to mean field game limit theory: A central limit
theorem}. {\sl Electron. J. Probab.} 24 (2019), 1-54.


\bibitem{Djete} Djete, M. F., {\it Large population games with interactions through controls and common noise: convergence results and equivalence between open-loop and closed-loop controls}, ESAIM: COCV, 2023, vol. 29, p. 39. DOI: 10.1051/cocv/2023005.


 \bibitem{Feinstein} Feinstein, Z., {\it Continuity and Sensitivity
Analysis of Parameterized Nash Games.} {\sl Econ Theory Bull} (2022). https://doi.org/10.1007/s40505-022-00228-0.


\bibitem{FRZ} Feinstein, Z.; Rudloff, B.; and Zhang, J. {\it Dynamic
set values for nonzero sum games with multiple equilibria}, {\sl
Mathematics of Operations Research}, 47 (2022), 616-642.

\bibitem{Feleqi} Feleqi, E., {\it The Derivation of Ergodic Mean Field
Game Equations for Several Populations of Players}. {\sl Dynamic Games
and Applications}, 3 (2013), 523-536.

\bibitem{Fischer} Fischer, M., {\it On the connection between
symmetric N-player games and mean field games}. {\sl The Annals of
Applied Probability}, 27 (2017), 757-810.

\bibitem{FischerSilva} Fischer, M. and Silva, F.J., {\it On the
Asymptotic Nature of First Order Mean Field Games} Applied Mathematics
\& Optimization (2020), 1432-0606.
  
 \bibitem{F} Foguen Tchuendom, R., {\it Uniqueness for
linear-quardratic mean field games with common noise}, {\sl Dynamic
Games and Applications} 8 (2018), 199-210.


 \bibitem{GM} Gangbo, W. and M\'esz\'aros, A.R., {\it Global
well-posedness of master equations for deterministic displacement
convex potential mean field games}, {\sl Comm. Pure Appl. Math.}, (2022), https://doi.org/10.1002/cpa.22069.
   
\bibitem{IZ} M. Iseri and J. Zhang, {\it Set Valued Hamilton-Jacobi-Bellman
      Equations}, preprint, arXiv:2311.05727.


  \bibitem{Lacker1} Lacker, D., {\it A general characterization of the
mean field limit for stochastic differential games}, {\sl
Probab. Th. Rel. Fields} 165 (2016), 581-648.

\bibitem{Lacker2} Lacker, D., {\it On the convergence of closed loop
Nash equilibria to the mean field game limit}, {\sl
Ann. Appl. Probab.} 30 (2020), 1693-1761.


\bibitem{DanielLuc} Lacker, D., Le Flem, L., {\it Closed-loop
    convergence for mean field games with common noise,
  Ann. Appl. Probab., 33 (4), 2681 - 2733, August 2023. DOI:
  10.1214/22-AAP1876.}

  
\bibitem{LL} Lasry, J.-M. and Lions, P.-L., {Mean field games}, {\sl
Jpn. J. Math.} 2 (2007), 229--260.

\bibitem{LT} Lauriere, M. and Tangpi, L. {\it Convergence of large
population games to mean field games with interaction through the
controls}, {\sl SIAM Journal on Mathematical Analysis},  54 (2022), 10.1137/22M1469328.
 
\bibitem{Lions} Lions, P.-L., {\it Cours au Coll\`ege de France},
www.college-de-france.fr.

\bibitem{MZZ}
Ma, J.,  Zhang, J., and  Zhang, Z. {\it Deep learning methods for set valued PDEs}, working paper.

\bibitem{MZ} Mou, C.  and Zhang, J., {\it Wellposedness of second
order master equations for mean field games with nonsmooth data},
{\sl Memoirs of the AMS}, accepted, arXiv:1903.09907.

\bibitem{NST} Nutz, M.; San Martin, J.; and Tan, X., {\it Convergence
to the Mean Field Game Limit: A Case Study}, {\sl Annals of Applied
Probability}, 30 (2020), 259-286.

\bibitem{PZ} Pham, T. and Zhang, J., {\it Two person zero–sum game in
weak formulation and path dependent Bellman-Isaacs equation}. {\sl
SIAM Journal on Control and Optimization}, 52 (2004), 2090-2121.

\bibitem{PT} Possamai, D. and Tangpi, L., {\it Non-asymptotic
convergence rates for mean-field games: weak formulation and
McKean–Vlasov BSDEs}, preprint, arXiv:2105.00484.

\bibitem{Sannikov} Sannikov, Y., {\it Games with Imperfectly
Observable Actions in Continuous Time}, {\sl Econometrica}, 75 (2007),
1285-1329.

\bibitem{Zhang} Zhang, J., {\sl Backward Stochastic Differential
Equations -- from linear to fully nonlinear theory}, Probability
Theory and Stochastic Modeling 86, Springer, New York, 2017.  
\end{thebibliography}
\end{document}